\documentclass[11pt,leqno]{amsart}
\usepackage[a4paper,margin=1in]{geometry}
\usepackage{amsmath,amsthm,amscd,amssymb,amsfonts,amsbsy,mathtools}
\usepackage{thmtools}
\usepackage{latexsym}
\usepackage{xcolor}
\usepackage{exscale}
\usepackage{graphicx}
\usepackage{bbm}
\usepackage[utf8]{inputenc}
\usepackage[colorlinks,citecolor=blue,linkcolor=blue,hypertexnames=false]{hyperref}
\usepackage[capitalize,nameinlink]{cleveref} 
\usepackage{caption}

\makeatletter
\newcommand{\crefnosort}[1]{{\@cref@sortfalse\cref{#1}}}
\makeatother

\newcommand{\nocontentsline}[3]{}
\newcommand{\tocless}[2]{\bgroup\let\addcontentsline=\nocontentsline#1{#2}\egroup}

\usepackage{tikz}
\usetikzlibrary{arrows.meta, positioning}
\tikzstyle{ultrathin}=[line width=.5pt]
\tikzstyle{thiner}=[line width=1pt]
\tikzstyle{thin}=[line width=1.5pt]
\tikzstyle{fat}=[line width=2pt]
\tikzstyle{heavier}=[line width=3pt]
\tikzstyle{ultrafat}=[line width=4pt]
\tikzstyle{ultrafat-pt}=[line width=7pt]
\definecolor{myred}{HTML}{E53935}
\definecolor{myblue}{HTML}{1E88E5}
\definecolor{mygreen}{HTML}{43A047}
\definecolor{myyellow}{HTML}{FDD835}
\definecolor{myorange}{HTML}{FB8C00}
\definecolor{mygold}{HTML}{F9A825}
\definecolor{mypurple}{HTML}{8E24AA}
\definecolor{mygray}{HTML}{BDBDBD}
\definecolor{mybrown}{HTML}{6D4C41}
\definecolor{mynavy}{HTML}{1A237E}
\definecolor{mypink}{HTML}{ffbfca}
\definecolor{myseagreen}{HTML}{26A69A}
\definecolor{myviolet}{HTML}{f07ef0}
\definecolor{mydarkblue}{HTML}{0D47A1}
\definecolor{mydarkcyan}{HTML}{E0FFFF}
\definecolor{darkgray}{rgb}{0.66, 0.66, 0.66}
\definecolor{mydarkgreen}{HTML}{1B5E20}
\definecolor{mydarkmagenta}{HTML}{AD1457}
\definecolor{mydarkorange}{HTML}{EF6C00}
\definecolor{lightblue}{rgb}{0.68, 0.85, 0.9}
\definecolor{lightcyan}{rgb}{0.88, 1.0, 1.0}
\definecolor{lightgray}{rgb}{0.83, 0.83, 0.83}
\definecolor{mylightgreen}{HTML}{81C784}
\definecolor{lightyellow}{rgb}{1.0, 1.0, 0.88}
\definecolor{myshadow}{rgb}{0.5, 0.5, 0.5}
\definecolor{pink}{rgb}{1.0, 0.75, 0.8}
\definecolor{violet}{rgb}{0.93, 0.51, 0.93}
\definecolor{myauxcolor}{RGB}{245, 255, 255}

\usepackage{subcaption} 
\usepackage[english]{babel}
\usepackage[thinc]{esdiff}

\usepackage{enumitem}
\usepackage{multirow}

\newcounter{step}   
\crefname{step}{Step}{Steps}
\Crefname{step}{Step}{Steps}
\newcommand{\namedstep}[1]{%
  \item[\textbf{Step \the\numexpr\value{step}+1\relax: #1.}]%
  \refstepcounter{step}%
}

\crefformat{enumi}{#2#1#3}
\Crefformat{enumi}{#2#1#3}

\AddToHook{cmd/appendix/before}{%
    \crefalias{section}{appendix}%
    \crefalias{subsection}{appendix}
}

\let\oldtocsection=\tocsection
\let\oldtocsubsection=\tocsubsection
\let\oldtocsubsubsection=\tocsubsubsection
\renewcommand{\tocsection}[2]{\hspace{0em}\oldtocsection{#1}{#2}}
\renewcommand{\tocsubsection}[2]{\hspace{2em} \oldtocsubsection{#1}{\small{#2}}}
\renewcommand{\tocsubsubsection}[2]{\hspace{4em}\oldtocsubsubsection{#1}{\scriptsize{#2}}}

\calclayout
\allowdisplaybreaks

\numberwithin{equation}{section}

\declaretheorem[name=Theorem, numberlike=equation]{theorem}
\declaretheorem[name=Lemma, numberlike=equation]{lemma}
\declaretheorem[name=Corollary, numberlike=equation]{corollary}

\declaretheorem[name=Definition, style=definition, numberlike=equation]{definition}

\declaretheorem[name=Remark, style=remark, numberlike=equation]{remark}

\newcommand{\abs}[1]{\left\lvert #1 \right\rvert}

\newcommand{\norm}[1]{\left\lVert #1 \right\rVert}

\newcommand{\eps}{\varepsilon}

\newcommand{\NN}{{\mathbb{N}}}

\newcommand{\CC}{\mathbb{C}}
\newcommand{\R}{\mathbb{R}}

\renewcommand{\Re}{\operatorname{Re}}
\renewcommand{\Im}{\operatorname{Im}}

\renewcommand{\H}{{\mathcal{H}}}
\renewcommand{\S}{{\mathcal{S}}}
\newcommand{\tgamma}{{\widetilde{\gamma}}}
\newcommand{\tGamma}{{\widetilde{\Gamma}}}

\DeclareMathOperator{\diam}{diam}

\DeclareMathOperator{\dist}{dist}

\DeclareMathOperator{\supp}{supp}

\newcommand{\vertiii}[1]{{\left\vert\kern-0.15ex\left\vert\kern-0.15ex\left\vert #1
        \right\vert\kern-0.15ex\right\vert\kern-0.15ex\right\vert}}

\def\Xint#1{\mathchoice
    {\XXint\displaystyle\textstyle{#1}}%
    {\XXint\textstyle\scriptstyle{#1}}%
    {\XXint\scriptstyle\scriptscriptstyle{#1}}%
    {\XXint\scriptscriptstyle%
        \scriptscriptstyle{#1}}%
    \!\int}
\def\XXint#1#2#3{{\setbox0=\hbox{$#1{#2#3}{%
                \int}$ }
        \vcenter{\hbox{$#2#3$ }}\kern-.6\wd0}}
\def\barint{\,\Xint -} 
\def\bariint{\barint_{} \kern-.4em \barint}
\def\bariiint{\bariint_{} \kern-.4em \barint}
\renewcommand{\iint}{\int_{}\kern-.34em \int} 
\renewcommand{\iiint}{\iint_{}\kern-.34em \int} 

\title[Boundary regularity of harmonic functions in \texorpdfstring{$C^1$}{C1} slit domains]{Boundary regularity of\\ harmonic functions in \texorpdfstring{$C^1$}{C1} slit domains}
\author[J. Domingo-Pasarin, P. Hidalgo-Palencia, A. Martínez, and C. Torres-Latorre]{Joan Domingo-Pasarin, Pablo Hidalgo-Palencia, \\ Alejandro Martínez and Clara Torres-Latorre}

\address{Joan Domingo-Pasarin, Departament de Matemàtiques i Informàtica, Universitat de Barcelona, Gran Via de les Corts Catalanes 585, 08007 Barcelona, Spain}
\email{jdomingopasarin@ub.edu}

\address{Pablo Hidalgo-Palencia, Departament de Matemàtiques i Informàtica, Universitat de Barcelona, Gran Via de les Corts Catalanes 585, 08007 Barcelona, Spain}
\email{pablo.hidalgo@ub.edu}

\address{Alejandro Martínez, Departament de Matemàtiques i Informàtica, Universitat de Barcelona, Gran Via de les Corts Catalanes 585, 08007 Barcelona, Spain}
\email{amartinezsanchez@ub.edu}

\address{Clara Torres-Latorre, Instituto de Ciencias Matemáticas, Consejo Superior de Investigaciones Científicas, C/ Nicolás Cabrera, 13-15, 28049 Madrid, Spain}
\email{clara.torres@icmat.es}


\begin{document}

\allowdisplaybreaks

\begin{abstract}
    We establish precise upper and lower estimates for harmonic functions vanishing on the slit (hypersurface with boundary) of a $C^1$ slit domain. In the case when the slit lies inside a hyperplane and the boundary is $C^{1,\alpha}$, it was known that harmonic functions have $\sqrt{d}$ growth near the boundary (De Silva, Savin). When the hypersurface is only $C^1$ and not necessarily flat, we show that the $\sqrt{d}$ growth persists, up to a correction factor
    \begin{equation*}
        \exp\Big( \pm C \int_\rho^r \omega(s)\, \frac{ds}{s} \Big)
    \end{equation*}
    determined by the $C^1$ modulus of continuity $\omega$ of the hypersurface, where $\rho < r$ are the two scales being compared. On the one hand, the integral converges when $\omega$ is Dini, so we extend the $C^{1,\alpha}$ theory to curved $C^{1,\mathrm{Dini}}$ slit domains. On the other hand, this allows us to obtain growth and regularity estimates when $\omega$ is logarithmic, as is the case at singular sets in obstacle-type problems.
\end{abstract}

\maketitle

\setcounter{tocdepth}{1} 
\tableofcontents

\section{Introduction}

Slit domains are domains whose boundary is a hypersurface \textit{with boundary} (the so-called slit), meaning that the ambient space approaches the surface from both sides and meets at its $(n-2)$-dimensional edge (see \cref{fig:slit}). Harmonic functions vanishing on the slit are classically known to exhibit a precise $1/2$-homogeneous growth near the edge. However, the existing results that identify this precise growth rate assume that the slit lies in a fixed hyperplane, and additionally that the edge is of class $C^{1,\alpha}$; see the works of De Silva and Savin \cite{SilvaSavinBHSlit} and Zhang \cite{Zhang}. Neither of these hypotheses reflects the geometries naturally produced by free boundary problems, which raises the questions: for which slits does this square root behavior survive? What behavior appears otherwise?

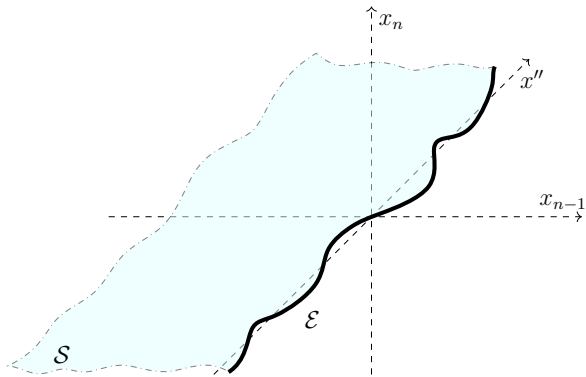
\begin{figure}[h!]
    \centering
    \resizebox{0.5\textwidth}{!}{
    \begin{tikzpicture}[scale=0.8]
        
        \draw[dashed, ->] (4.421, 0.134) -- (11.195, 6.907);
        \draw[dashed, ->] (7.808, 0.134) -- (7.808, 8.036);
        \draw[dashed, ->] (2.164, 3.521) -- (12.324, 3.521);

        \filldraw[dash dot, fill=lightcyan, opacity=0.5] (4.745, 0.188) .. controls (5.117, 0.424) and (5.142, 0.876) .. (5.237, 1.092) .. controls (5.332, 1.308) and (5.497, 1.288) .. (5.772, 1.414) .. controls (6.047, 1.54) and (6.431, 1.813) .. (6.613, 2.077) .. controls (6.795, 2.342) and (6.775, 2.598) .. (6.882, 2.815) .. controls (6.99, 3.032) and (7.225, 3.21) .. (7.401, 3.321) .. controls (7.577, 3.432) and (7.692, 3.476) .. (8.009, 3.594) .. controls (8.325, 3.712) and (8.841, 3.903) .. (9.044, 4.179) .. controls (9.246, 4.454) and (9.135, 4.814) .. (9.159, 5) .. controls (9.183, 5.187) and (9.342, 5.201) .. (9.507, 5.232) .. controls (9.673, 5.263) and (9.844, 5.311) .. (10.011, 5.517) .. controls (10.178, 5.723) and (10.34, 6.087) .. (10.395, 6.309) .. controls (10.451, 6.531) and (10.4, 6.612) .. (10.462, 6.748) -- (10.462, 6.748) .. controls (10.178, 6.654) and (9.977, 6.771) .. (9.726, 6.796) .. controls (9.475, 6.821) and (9.176, 6.754) .. (8.996, 6.711) .. controls (8.817, 6.668) and (8.759, 6.648) .. (8.618, 6.685) .. controls (8.478, 6.723) and (8.256, 6.818) .. (8.009, 6.833) .. controls (7.763, 6.847) and (7.493, 6.781) .. (7.301, 6.768) .. controls (7.108, 6.754) and (6.993, 6.793) .. (6.9, 6.84) .. controls (6.808, 6.887) and (6.739, 6.942) .. (6.688, 6.979) .. controls (6.638, 7.017) and (6.606, 7.037) .. (6.614, 7.026) .. controls (5.927, 6.627) and (5.666, 6.041) .. (5.492, 5.714) .. controls (5.317, 5.387) and (5.227, 5.321) .. (5.018, 5.217) .. controls (4.809, 5.114) and (4.479, 4.972) .. (4.255, 4.772) .. controls (4.031, 4.572) and (3.911, 4.314) .. (3.783, 4.046) .. controls (3.655, 3.777) and (3.518, 3.499) .. (3.285, 3.29) .. controls (3.051, 3.082) and (2.722, 2.944) .. (2.485, 2.705) .. controls (2.249, 2.467) and (2.105, 2.129) .. (1.855, 1.899) .. controls (1.605, 1.669) and (1.249, 1.547) .. (1.019, 1.289) .. controls (0.789, 1.031) and (0.685, 0.637) .. (0, 0.32) .. controls (0.166, 0.316) and (0.414, 0.17) .. (0.624, 0.151) .. controls (0.833, 0.133) and (1.004, 0.241) .. (1.168, 0.25) .. controls (1.333, 0.26) and (1.492, 0.17) .. (1.704, 0.191) .. controls (1.915, 0.212) and (2.18, 0.344) .. (2.475, 0.331) .. controls (2.77, 0.318) and (3.097, 0.159) .. (3.473, 0.201) .. controls (3.848, 0.243) and (4.274, 0.487) .. (4.745, 0.188) -- cycle;
        
        \draw[fat] (4.739, 0.185) .. controls (5.112, 0.42) and (5.136, 0.873) .. (5.231, 1.089) .. controls (5.326, 1.305) and (5.492, 1.284) .. (5.767, 1.41) .. controls (6.042, 1.536) and (6.426, 1.809) .. (6.608, 2.074) .. controls (6.79, 2.338) and (6.769, 2.595) .. (6.877, 2.812) .. controls (6.984, 3.029) and (7.22, 3.206) .. (7.395, 3.318) .. controls (7.571, 3.429) and (7.687, 3.473) .. (8.003, 3.591) .. controls (8.319, 3.709) and (8.836, 3.9) .. (9.039, 4.175) .. controls (9.241, 4.451) and (9.129, 4.81) .. (9.153, 4.997) .. controls (9.177, 5.184) and (9.337, 5.197) .. (9.502, 5.228) .. controls (9.667, 5.259) and (9.839, 5.308) .. (10.005, 5.514) .. controls (10.172, 5.72) and (10.334, 6.084) .. (10.39, 6.306) .. controls (10.446, 6.528) and (10.395, 6.608) .. (10.456, 6.745);

        \node[anchor=center, font=\large] at (6.542, 1.313) {$\mathcal{E}$};
        \node[anchor=center, font=\large] at (8.236, 7.702) {$x_n$};
        \node[anchor=center, font=\large] at (11.927, 3.819) {$x_{n-1}$};
        \node[anchor=center, font=\large] at (11.256, 6.471) {$x''$};
        \node[anchor=center, font=\large] at (1.163, 0.556) {$\mathcal{S}$};
    \end{tikzpicture}
    }
    
    \caption{Illustration of a slit $\mathcal{S}$, and the associated slit domain $\Omega \coloneqq \R^n \setminus \S$. The most interesting behavior, from the PDE perspective, happens near the edge $\mathcal{E}$, which is the $(n-2)$-dimensional boundary of the hypersurface $\S$.}
    \label{fig:slit}
\end{figure}

Indeed, the thin obstacle (or Signorini) problem, a well-known free boundary problem, is itself posed in a slit domain. The solution is constrained to lie above a hypersurface and is harmonic when it lies strictly above it. The coincidence set plays the role of the slit, and its boundary (the free boundary), that of the edge. In this setting, the slit is naturally curved. As for the edge, the free boundary is smooth near regular points, but at singular points the available bounds are logarithmic: every stratum is covered by $C^{1,\log}$ manifolds \cite{ColomboSpolaorVelichkov20}, i.e. $C^{1, \omega}$ manifolds with $\omega(t) = (\log |t^{-1}|)^{-\varepsilon}$ for $\varepsilon > 0$. Similarly, in the classical obstacle problem, for $n \geq 3$, the lower strata of the singular set are only $C^{1,\log}$ \cite{ColomboSpolaorVelichkov18GAFA, FigalliSerra17}. These moduli fall strictly outside the H\"older scale, so no $C^{1,\alpha}$ theory reaches them. Slit geometries do not only appear in free boundary problems. They also arise in fracture mechanics, where the square-root profile above is the crack-tip displacement \cite{Williams57}, and in acoustic diffraction and elasticity theory \cite{Durand83,Krutitskii04,Antipov19,Marshall22}.

The $C^1$ regime below the H\"older threshold is the borderline case for boundary behavior. Below this regularity, in general Lipschitz slit domains (for $n \ge 3$), the $1/2$ homogeneity fails and the growth exponent becomes dependent on the local geometry. The available boundary Harnack principles of De Silva and Savin \cite{DS20}, Ros-Oton and the fourth author \cite{ClaraBoundaryHarnack}, and Petrosyan and Shi \cite{PetrosyanShi14} in the parabolic setting, yield H\"older estimates for quotients but do not identify the precise growth rate. Above the $C^1$ threshold, one could instead trade curvature for variable coefficients by flattening the slit, a regime recently covered by Zhang \cite{Zhang}. However, their flattening of slits either breaks the structural assumptions necessary to apply boundary regularity results or causes the loss of an extra derivative, so the approach is only available for $C^{2,\alpha}$ slits; see \cite[Section 1.6]{Zhang} and \cite[Section 1]{SilvaSavinThinOnePhase}. More fundamentally, in the non-Dini setting ($C^{1,\omega}$ with $\omega$ non-Dini), flattening gives an equation with non-Dini coefficients, a setting where we are not aware of any precise boundary regularity theory.

Our approach treats $C^1$ slit domains directly, without assuming flatness. We establish upper and lower estimates showing that the square-root behavior persists up to an explicit correction factor determined by the modulus of continuity of the domain. This allows us to obtain quantitative growth and regularity estimates for harmonic functions in $C^1$ slit domains with a non-Dini modulus of continuity, such as those found near singular points in obstacle-type problems. Moreover, this correction factor is bounded when the modulus is Dini, so we extend the $C^{1,\alpha}$ theory to curved $C^{1,\mathrm{Dini}}$ slit domains.

\subsection{Main results}

We first fix the geometry. Points of $\R^n$ are written 
\begin{equation*}
    x = (x'', x_{n-1}, x_n) \in \R^{n-2}\times\R\times\R,
\end{equation*} and we abbreviate $x' \coloneqq (x'', x_{n-1}) \in \R^{n-1}$. We consider slit domains of the form $\Omega = \R^n \setminus \S$, where the slit $\S$ is the graph
\begin{equation} \label{eq:general_slit}
    \S \coloneqq \big\{ (x'', x_{n-1}, x_n) \in \R^{n-2}\times\R\times\R \ : \ x_{n-1} \leq \gamma(x''), \ x_n = \Gamma(x') \big\}
\end{equation}
of a function $\Gamma : \R^{n-1} \to \R$ over the subgraph of a function $\gamma : \R^{n-2} \to \R$, and
\begin{equation} \label{eq:edge}
    \mathcal{E} \coloneqq \big\{ x_{n-1} = \gamma(x''), \ x_n = \Gamma(x') \big\}
\end{equation}
is its edge. Upon translation we assume, without loss of generality, that $0 \in \mathcal{E}$, that is $\gamma(0) = 0$ and $\Gamma(0) = 0$. We will center our study of boundary behavior at this point.

We call $\Omega$ a \emph{Lipschitz slit domain} with constant $L$ when $\gamma$ and $\Gamma$ are Lipschitz with constant $L$ (see \cref{def:lip_slit_domain}). Given a modulus of continuity $\omega$, we say that $\Omega$ has \emph{$C^1$-modulus of continuity $\omega$ at $0$} (see \cref{def:C1_slit_domain}) when, in addition,
\begin{equation} \label{eq:C1_modulus_informal}
    \norm{\nabla\gamma - \nabla\gamma(0)}_{L^\infty(B''_r)} \leq \omega(r)
    \quad\text{and}\quad
    \norm{\nabla\Gamma - \nabla\Gamma(0)}_{L^\infty(B'_r)} \leq \omega(r)
    \qquad\text{for all small $r$}.
\end{equation}
Without loss of generality, after performing a rotation, we will assume that $\nabla\gamma(0) = 0$ and $\nabla\Gamma(0) = 0$, so that $\{x_n = 0\}$ is tangent to $\S$ at the origin. The case $\Gamma \equiv 0$ corresponds to flat slits, as considered in \cite{PetrosyanShi14, SilvaSavinThinOnePhase, SilvaSavinBHSlit, Zhang}; we will instead consider curved slits. 

Our first main result is a lower bound on the growth of harmonic functions approaching the edge non-tangentially (see \cref{fig:cones}). It is a Hopf--Ole\u{\i}nik type statement, i.e., the solution grows at least like $\dist(\cdot,\S)^{1/2}$, as in the flat smooth case, up to a correction factor determined explicitly by $\omega$, whose form is explained in \cref{sec:motivation}. 

\begin{theorem}[Hopf--Ole\u{\i}nik type lower bound] \label{th:lower_bound_thm}
    Let $\Omega \subset \R^n$ be a Lipschitz slit domain with $C^1$-modulus of continuity $\omega$ at $0$ in the sense of \cref{def:C1_slit_domain}. Let $u$ be a non-negative weak solution to the problem
    \begin{equation*}
        \left\{
            \begin{aligned}
            - \Delta u &= 0 &&\text{in } \Omega \cap B_1, \\
            u        &= 0 &&\text{on } \partial\Omega \cap B_1.
            \end{aligned}
        \right.
    \end{equation*}
    Fix $0 < \alpha_0 < \pi$ and $0 < \beta_0 < \pi/2$. For any unit vector\footnote{For $n = 2$, there is no $e''$, and $\beta = 0$, i.e. $e = (\cos\alpha, \sin\alpha)$.} $e'' \in \R^{n-2}$, and angles $|\alpha| \leq \alpha_0$ and $|\beta| \leq \beta_0$, let us define the unit $n$-dimensional vector
    \begin{equation} 
        \label{eq:cone_direction}
        e = e(\alpha, \beta) \coloneqq
        (e'' \sin \beta, \cos{\alpha} \cos{\beta}, \sin{\alpha} \cos{\beta}).
    \end{equation}
    Then, there exist $r_0 > 0$ depending only on $n$, $\omega$, $\alpha_0$ and $\beta_0$, and $C > 0$ depending only on $n$, $\alpha_0$ and $\beta_0$, such that for every $0 < \rho < r < r_0$,
    \begin{equation*}
        \frac{u(\rho e)}{\sqrt{\rho}} \geq \frac{1}{C} \frac{u(r e)}{\sqrt{r}} \exp\left(-C \int_{64\rho}^{64r} \omega(s) \frac{ds}{s} \right).
    \end{equation*}
\end{theorem}

\begin{figure}[h]
    \centering
    \resizebox{1.0\textwidth}{!}{
    \begin{tikzpicture}[scale=0.4]
    \filldraw[dash dot, fill=lightcyan, opacity=0.3] (4.745, 5.717) .. controls (5.117, 5.953) and (5.142, 6.405) .. (5.237, 6.621) .. controls (5.332, 6.837) and (5.497, 6.817) .. (5.772, 6.943) .. controls (6.047, 7.069) and (6.431, 7.342) .. (6.613, 7.606) .. controls (6.795, 7.871) and (6.775, 8.127) .. (6.882, 8.344) .. controls (6.99, 8.561) and (7.225, 8.739) .. (7.401, 8.85) .. controls (7.577, 8.961) and (7.692, 9.005) .. (8.009, 9.123) .. controls (8.325, 9.241) and (8.841, 9.432) .. (9.044, 9.708) .. controls (9.246, 9.983) and (9.135, 10.343) .. (9.159, 10.529) .. controls (9.183, 10.716) and (9.342, 10.729) .. (9.507, 10.76) .. controls (9.673, 10.791) and (9.844, 10.84) .. (10.011, 11.046) .. controls (10.178, 11.252) and (10.34, 11.616) .. (10.395, 11.838) .. controls (10.451, 12.06) and (10.4, 12.141) .. (10.462, 12.277) -- (10.462, 12.277) .. controls (10.178, 12.183) and (9.977, 12.3) .. (9.726, 12.325) .. controls (9.475, 12.35) and (9.176, 12.283) .. (8.996, 12.24) .. controls (8.817, 12.196) and (8.759, 12.177) .. (8.618, 12.214) .. controls (8.478, 12.252) and (8.256, 12.347) .. (8.009, 12.362) .. controls (7.763, 12.376) and (7.493, 12.31) .. (7.301, 12.296) .. controls (7.108, 12.283) and (6.993, 12.322) .. (6.9, 12.369) .. controls (6.808, 12.416) and (6.739, 12.471) .. (6.688, 12.508) .. controls (6.638, 12.546) and (6.606, 12.566) .. (6.614, 12.555) .. controls (5.927, 12.156) and (5.666, 11.569) .. (5.492, 11.243) .. controls (5.317, 10.916) and (5.227, 10.85) .. (5.018, 10.746) .. controls (4.809, 10.642) and (4.479, 10.501) .. (4.255, 10.301) .. controls (4.031, 10.101) and (3.911, 9.843) .. (3.783, 9.575) .. controls (3.655, 9.306) and (3.518, 9.028) .. (3.285, 8.819) .. controls (3.051, 8.611) and (2.722, 8.472) .. (2.485, 8.234) .. controls (2.249, 7.996) and (2.105, 7.658) .. (1.855, 7.428) .. controls (1.605, 7.198) and (1.249, 7.076) .. (1.019, 6.818) .. controls (0.789, 6.56) and (0.685, 6.165) .. (0, 5.849) .. controls (0.166, 5.845) and (0.414, 5.699) .. (0.624, 5.68) .. controls (0.833, 5.662) and (1.004, 5.77) .. (1.168, 5.779) .. controls (1.333, 5.788) and (1.492, 5.698) .. (1.704, 5.72) .. controls (1.915, 5.741) and (2.18, 5.873) .. (2.475, 5.86) .. controls (2.77, 5.847) and (3.097, 5.688) .. (3.473, 5.73) .. controls (3.848, 5.772) and (4.274, 6.015) .. (4.745, 5.717) -- cycle;
    \draw[dashed, ->] (3.857, 5.099) -- (11.759, 13.001);

    \draw[very thick] (4.739, 5.714) .. controls (5.112, 5.949) and (5.136, 6.402) .. (5.231, 6.618) .. controls (5.326, 6.834) and (5.492, 6.813) .. (5.767, 6.939) .. controls (6.042, 7.065) and (6.426, 7.338) .. (6.608, 7.603) .. controls (6.79, 7.867) and (6.769, 8.124) .. (6.877, 8.341) .. controls (6.984, 8.558) and (7.22, 8.735) .. (7.395, 8.846) .. controls (7.571, 8.957) and (7.687, 9.002) .. (8.003, 9.12) .. controls (8.319, 9.237) and (8.836, 9.429) .. (9.039, 9.704) .. controls (9.241, 9.98) and (9.129, 10.339) .. (9.153, 10.526) .. controls (9.177, 10.712) and (9.337, 10.726) .. (9.502, 10.757) .. controls (9.667, 10.788) and (9.839, 10.836) .. (10.005, 11.043) .. controls (10.172, 11.249) and (10.334, 11.612) .. (10.39, 11.835) .. controls (10.446, 12.057) and (10.395, 12.137) .. (10.456, 12.274);

    \draw[dotted] 
    (7.808, 9.05) ellipse[x radius=4.516, y radius=2.258];
    \draw[dotted] 
    (7.808, 9.05) circle[radius=4.516];
    \draw[dotted] 
    (7.808, 9.05) circle[radius=1, cm={1.812,0.271,0.965,5.258,(0,0)}];
    
    \draw[dashed, ->] (7.808, 3.405) -- (7.808, 14.694);
    \draw[dashed] (7.207, 12.571) arc[start angle=-99.81, end angle=-70.972, x radius=3.482, y radius=-3.482];
    \filldraw[blue, draw=blue, thick, opacity=0.7] (8.937, 12.436) arc[start angle=-139.829, end angle=-121.971, x radius=3.42, y radius=-3.42];
    \draw[dashed, opacity=0.7] (5.781, 7.032) arc[start angle=-158.921, end angle=-140.502, x radius=19.553, y radius=-19.553];
    \filldraw[red, draw=red, thick, opacity=0.2] (9.739, 13.131) arc[start angle=-64.679, end angle=0, x radius=4.516, y radius=-4.516];
    \filldraw[red, draw=red, thick, opacity=0.2] (9.739, 13.131) -- (7.808, 9.05) -- (12.324, 9.05);
    \filldraw[blue, draw=blue, opacity=0.2] (9.739, 13.131) -- (7.808, 9.05) -- (8.937, 12.436);
    \filldraw[draw=purple, thick, fill=blue, opacity=0.5] (7.808, 9.05) -- (8.937, 12.436);
    \filldraw[purple, draw=purple] 
    (8.937, 12.436) circle[radius=0.121];

    \draw[blue, opacity=0.7] (8.586, 11.306) arc[start angle=-139.828, end angle=-116.079, x radius=1.481, y radius=-1.481];

    \draw[red, <-, opacity=0.7] (8.774, 11) arc[start angle=-63.549, end angle=0, x radius=2.178, y radius=-2.178];
    \draw[shift={(14.581, 9.05)}, scale=2, dashed, ->] (0, 0) -- (5.644, 0);
    \draw[shift={(20.226, 3.405)}, scale=2, dashed, ->] (0, 0) -- (0, 5.644);
    \draw[shift={(26.999, 9.05)}, scale=2, dashed, ->] (0, 0) -- (5.644, 0);
    \draw[shift={(32.644, 3.405)}, scale=2, dashed, ->] (0, 0) -- (0, 5.644);
    \draw[shift={(15.737, 9.564)}, scale=2, very thick] (0, 0) .. controls (0.591, -0.408) and (0.783, -0.208) .. (0.955, -0.207) .. controls (1.127, -0.207) and (1.279, -0.405) .. (1.425, -0.426) .. controls (1.571, -0.448) and (1.71, -0.292) .. (1.843, -0.233) .. controls (1.975, -0.175) and (2.101, -0.215) .. (2.168, -0.235) .. controls (2.236, -0.256) and (2.244, -0.257) .. (2.244, -0.257);
    \draw[shift={(28.172, 9.67)}, scale=2, very thick] (0, 0) .. controls (0.655, -0.697) and (0.838, -0.363) .. (1.004, -0.277) .. controls (1.171, -0.191) and (1.321, -0.354) .. (1.44, -0.342) .. controls (1.559, -0.33) and (1.648, -0.144) .. (1.775, -0.109) .. controls (1.903, -0.074) and (2.07, -0.191) .. (2.153, -0.249) .. controls (2.236, -0.308) and (2.236, -0.309) .. (2.304, -0.276) .. controls (2.371, -0.243) and (2.507, -0.176) .. (2.658, -0.224) .. controls (2.809, -0.273) and (2.976, -0.437) .. (3.136, -0.428) .. controls (3.296, -0.418) and (3.45, -0.236) .. (3.555, -0.234) .. controls (3.659, -0.233) and (3.715, -0.413) .. (3.848, -0.384) .. controls (3.981, -0.355) and (4.192, -0.116) .. (4.49, -0.221);
    \draw[dotted] 
    (20.226, 9.05) circle[radius=4.516];
    \draw[dotted] 
    (32.644, 9.05) circle[radius=4.516];
    \filldraw[shift={(20.226, 9.05)}, scale=2, fill=lightcyan] (0, 0) -- (-2.019, 1.01);
    \draw[shift={(20.226, 9.05)}, scale=2] (0, 0) -- (-2.019, -1.01);
    \draw[shift={(32.644, 9.05)}, scale=2] (0, 0) -- (-1.693, 1.493);
    \draw[shift={(32.644, 9.05)}, scale=2] (0, 0) -- (1.693, 1.493);
    \draw[shift={(14.581, 11.307)}, scale=2, line cap=round] (0, 0) -- (0, 0);
    \draw[->] (21.355, 9.05) arc[start angle=0, end angle=153.437, radius=1.129];
    \draw[->] (32.644, 10.179) arc[start angle=-91.498, end angle=-41.414, x radius=1.104, y radius=-1.104];
    \fill[lightcyan, opacity=0.3] (16.187, 11.069) arc[start angle=-153.435, end angle=0, x radius=4.516, y radius=-4.516];
    \fill[lightcyan, opacity=0.3] (24.741, 9.05) arc[start angle=0, end angle=153.435, x radius=4.516, y radius=-4.516];
    \fill[shift={(24.741, 9.05)}, scale=2, lightcyan, opacity=0.3] (0, 0) -- (-4.277, -1.01) -- (-2.258, 0) -- (0, 0) -- cycle;
    \fill[shift={(24.741, 9.05)}, scale=2, lightcyan, opacity=0.3] (0, 0) -- (-2.258, 0) -- (-4.277, 1.01) -- cycle;
    \fill[lightcyan, opacity=0.3] (29.257, 12.036) arc[start angle=-138.59, end angle=-41.41, x radius=4.516, y radius=-4.516];
    \fill[shift={(32.644, 9.05)}, scale=2, lightcyan, opacity=0.3] (0, 0) -- (-1.693, 1.493) -- (1.693, 1.493) -- (1.693, 1.493) -- cycle;

    \node at (6.816, 7.143) {$\mathcal{E}$};
    \node at (8.518, 14.5) {$x_n$};
    \draw[dashed, ->] (2.164, 9.05) -- (13.453, 9.05);
    \node at (11.863, 12.413) {$x''$};
    \node at (12.972, 9.5) {$x_{n-1}$};
    \node at (1.265, 6.245) {$\mathcal{S}$};
    \node[text=blue] at (9.183, 12.06) {$\beta$};
    \node[text=red] at (10.141, 10.188) {$\alpha$};
    \node[text=purple] at (8.521, 12.958) {$e$};

    \node at (16.534, 9.8) {$\mathcal{S}$};
    
    \node at (20.984, 14.5) {$x_n$};
    \node at (25.37, 9.5) {$x_{n-1}$};
    \node at (20.797, 10.397) {$\alpha_0$};
    \node at (29.003, 9.8) {$\mathcal{S}$};
    \node at (37.941, 9.75) {$x''$};
    \node at (33.215, 10.689) {$\beta_0$};
    \node at (33.47, 14.5) {$x_n$};

\end{tikzpicture}
}
    
    \caption{Projections of the conical region determined by the angles $\alpha_0$ and $\beta_0$ for a $C^1$ slit $\S$. The vector $e = e(\alpha, \beta)$ is defined via certain spherical-like coordinates, where the \textit{north pole} sits at $(e'', 0, 0) \in \R^{n-2} \times \R \times \R$ and the equator lies in the $x_{n-1}x_n$ plane; so $\alpha$ acts as the longitude (in the plane $x_{n-1}x_n$), and $\beta$ as the latitude. Since the slit is $C^1$ at 0, it gets infinitesimally flat around the origin, so these cones will lie strictly inside $\Omega\cap B_{r_0}$ as soon as the scale $r_0$ is small enough (see \cref{lem:spherical_distance_boundary}). }
    \label{fig:cones}
\end{figure}
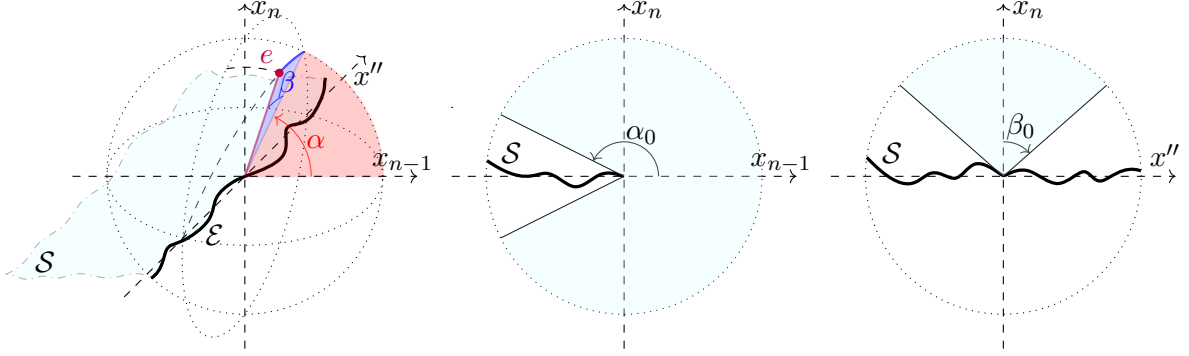 

Our second main result is the corresponding upper bound, with the same $\dist(\cdot,\S)^{1/2}$ behavior along with a correction factor. It also holds with a right-hand side and non-zero boundary data.

\begin{theorem}[Boundary growth] \label{th:upper_bound_thm}
    Let $\Omega \subset \R^n$ be a Lipschitz slit domain with $C^1$-modulus of continuity $\omega$ at $0$ in the sense of \cref{def:C1_slit_domain}. Let $u$ be a weak solution to
    \begin{equation*}
        \left\{
            \begin{aligned}
            -\Delta u &= f &&\text{in } \Omega \cap B_1, \\
            u        &= g &&\text{on } \partial\Omega \cap B_1,
            \end{aligned}
        \right.
    \end{equation*}
    where $f \in L^p(\Omega \cap B_1)$ with $p > \frac{2n}{3}$ and $g$ is a continuous function with $C^{\frac{1}{2}}$-modulus of continuity $\omega_g$ at 0 (see \cref{def:C1/2_modulus_of_cont}). Set $v(x) \coloneqq u(x) - g(0)$. Then there exist constants $C > 0$ depending only on $n$ and $p$, and $r_0 > 0$ depending only on $n$ and $\omega$, such that for every $0 < \rho < r < r_0$,
    \begin{equation*}
        \frac{\|v\|_{L^\infty(B_\rho)}}{\sqrt{\rho}}
        \leq
        C \left( \frac{\|v\|_{L^\infty(B_r)}}{\sqrt{r}} + r^{\frac{3}{2} - \frac{n}{p}}\|f\|_{L^p(\Omega \cap B_{2r})} + \int_{8\rho}^{8r} \omega_g(s)\frac{ds}{s} \right) \exp\left(C \int_{256\rho}^{256r} \omega(s) \frac{ds}{s} \right).
    \end{equation*}
\end{theorem}

The condition $p > \frac{2n}{3}$ is natural in the sense that, by Calder\'on--Zygmund estimates and the Morrey embedding, solutions of $-\Delta w = f \in L^p$ (with $p < \infty$) belong to $W^{2,p}$, which embeds into $C^{0,\alpha}$ with $\alpha > \frac{1}{2}$ (which corresponds to the square root behavior) precisely when $p > \frac{2n}{3}$; see \cref{lem:bound_poisson_eq} for more details. Combining \cref{th:upper_bound_thm} with interior estimates yields a modulus of continuity for $u$ across the entire edge $\mathcal{E}$, which is the content of the next result.

\begin{corollary}[Boundary regularity] \label{cor:boundary_regularity}
    Under the hypotheses of \cref{th:upper_bound_thm}, assume in addition that $\Omega$ has $C^1$-modulus $\omega$ in $B_1$ (see \cref{def:C1_slit_domain}) and $g$ has $C^{\frac{1}{2}}$-modulus $\omega_g$ in $B_1$ (see \cref{def:C1/2_modulus_of_cont}). Then, after decreasing $r_0$ if necessary, $u \in C^{\omega_u}(\overline{\Omega} \cap B_{r_0/16})$ where the modulus of continuity $\omega_u$ is given by
    \begin{equation*}
        \omega_u(t)
        \coloneqq
        C \sqrt{t} \left( \frac{\|u\|_{L^\infty(B_{r_0})}}{\sqrt{r_0}} + r_0^{\frac{3}{2} - \frac{n}{p}}\|f\|_{L^p(\Omega \cap B_{r_0})} + \int_t^{8r_0} \omega_g(s) \frac{ds}{s} \right) \exp \left( C \int_t^{256r_0} \omega(s) \frac{ds}{s} \right),
    \end{equation*}
    and the constant $C$ depends only on $n$ and $p$.
\end{corollary}

When $\omega$ is Dini, our correction terms collapse and we recover the same estimates as in the flat and smooth case \cite{SilvaSavinBHSlit}, with no loss.

\begin{corollary}[Dini slits] 
    \label{cor:dini}
    Let $\Omega$ be a Lipschitz slit domain with $C^1$-modulus of continuity $\omega$ at $0$, and assume $\omega$ is Dini, that is $\int_0^1 \omega(s)\,\frac{ds}{s} < \infty$. Then, with $r_0$ and $e$ as above, the following hold:
    \begin{enumerate}[label=(\roman*)]
        \item If $u \geq 0$ is harmonic in $\Omega \cap B_1$ and vanishes on $\partial\Omega \cap B_1$, then
        \begin{equation*}
            \frac{u(\rho e)}{\sqrt{\rho}}
            \ \geq \
            \frac{1}{C}\,\frac{u(r e)}{\sqrt{r}}
            \qquad \text{for all } 0 < \rho < r < r_0;
        \end{equation*}
        where $C$ depends only on $n$, $\alpha_0$, $\beta_0$ and $\int_0^1 \omega(s)\frac{ds}{s}$.
        \item If moreover $\omega_g$ is Dini, $\Omega$ has $C^1$-modulus $\omega$ in $B_1$ and $g$ has $C^{\frac{1}{2}}$-modulus $\omega_g$ in $B_1$, then every $u$ as in \cref{th:upper_bound_thm} satisfies $u \in C^{0, \frac{1}{2}}(\overline{\Omega} \cap B_{r_0/16})$. In this case, the $C^{\frac{1}{2}}$-seminorm of $u$ is bounded by
        \begin{equation*}
            C\left( \frac{\|u\|_{L^\infty(B_{r_0})}}{\sqrt{r_0}} + r_0^{\frac{3}{2} - \frac{n}{p}}\|f\|_{L^p(\Omega \cap B_{r_0})} + \int_0^{8r_0} \omega_g(s) \frac{ds}{s} \right),
        \end{equation*}
        where the constant $C$ depends only on $n$, $p$ and $\int_0^{256r_0} \omega(s) \frac{ds}{s}$.
    \end{enumerate}
\end{corollary}

In thick domains (with full dimensional exterior, as opposed to slit domains), $C^{1,\mathrm{Dini}}$ is the classical condition under which solutions attain the boundary behavior of the smooth case, going back to V\'yborn\'y \cite{Vyborny63,Vyborny64} (in a form later shown by Lieberman \cite{Lieb} to be equivalent to $C^{1,\mathrm{Dini}}$). This condition is sharp, i.e., the exterior $C^{1,\mathrm{Dini}}$ condition is necessary \cite{Widman, VerzhbinskiiMazya}. \Cref{cor:dini} is its slit counterpart, and we expect the Dini condition to be necessary as well.

\begin{remark}[Sharpness]
    \label{rmk:optimality}
    We expect the correction factor to be sharp (up to constants), as suggested by the work of Kozlov and Maz'ya \cite{KozlovMazya03} in thick domains, who obtain an identity for harmonic functions carrying exactly this factor. In the same way, for a self-similar Lipschitz cone the solution behaves like $\dist(\cdot,\S)^{\frac{1}{2} \pm \varepsilon}$, with $\varepsilon$ depending on the opening of the cone, which coincides (up to constants) with the exponent produced by our estimates.
\end{remark}

\begin{remark}[Growth bounds]
    For any $r > 0$, since $\omega(s)\to0$ as $s\to0$, $\int_\rho^r\omega(s)\,\frac{ds}{s}=o(\log(r/\rho))$ as $\rho \to 0$. Consequently, for positive harmonic functions vanishing on the slit, \cref{th:lower_bound_thm,th:upper_bound_thm} imply that along any admissible non-tangential direction $e$, for any $\varepsilon > 0$ there exists $C_\varepsilon > 0$ such that for $\rho > 0$ small enough it holds 
    \begin{equation*}
        C_\varepsilon^{-1} \rho^{\frac12 + \varepsilon}
        \leq 
        u(\rho e) 
        \leq 
        C_\varepsilon \rho^{\frac12 - \varepsilon}.
    \end{equation*}
    Thus, even without a Dini assumption, the correction factors contribute only a vanishing change in the growth exponent.
\end{remark}

\begin{remark}[Asymptotic behavior]
    \label{rmk:asymptotics}
    \cref{th:lower_bound_thm,th:upper_bound_thm} do not establish an asymptotic expansion near the edge. For a positive harmonic function vanishing on the slit, a natural further question in the Dini setting is to use these two-sided bounds to determine whether $u(\rho e)/\sqrt{\rho}$ converges as $\rho\to0$, and to identify the limit profile. Beyond the Dini regime, the correction factors make this a nontrivial question, and we believe that the barriers developed here are the right tool to answer it.
\end{remark}

\subsection{Key ideas of the proof} \label{sec:motivation}

Our starting point is the behavior of the complex square root. In two dimensions, consider the model slit
\begin{equation} \label{eq:complex_slit}
    \H_\CC \coloneqq \big\{ z \in \CC: \Re(z) \leq 0, \, \Im(z) = 0 \big\}.
\end{equation}
The principal branch of $z \mapsto \sqrt{z}$ maps the model slit domain $\CC\setminus\H_\CC$ conformally onto a half-plane. Since harmonic functions in the half-plane are well understood, this transformation immediately yields precise information on harmonic functions in the model slit domain, including the characteristic $\dist(\cdot,\H_\CC)^{1/2}$ behavior near the origin.

The same idea extends to higher dimensions: now our model slit is
\begin{equation} \label{eq:model_slit}
    \H \coloneqq \big\{ (x'', x_{n-1}, x_n) \in \R^{n-2} \times \R \times \R \ : \ x_{n-1} \leq 0 \text{ and } x_n = 0 \big\}.
\end{equation}
Then, applying the complex square root to the last two coordinates in the model slit domain $\R^n\setminus\H$ gives the model harmonic function
\begin{equation*}
    U_0(x)
    =
    U_0(x'', x_{n-1}, x_n)
    \coloneqq
    \Re \big(\sqrt{x_{n-1} + ix_n}\big),
\end{equation*}
which vanishes continuously on the slit $\H$ and, away from it, satisfies $U_0(x) \approx \dist(x, \H)^{1/2}$ near the edge. Our aim is to construct analogues of $U_0$ when the slit is curved and merely Lipschitz.

The difficulty is that the square root owes its usefulness to being holomorphic, which ties it to the flat geometry of $\H$. For a general slit $\S$ as in \eqref{eq:general_slit}, with $\gamma$ and $\Gamma$ not identically zero, there is no explicit analogue which is harmonic, already in two dimensions. The natural attempt is to flatten $\S$ onto $\H$ by
\begin{equation*}
    (x'', x_{n-1}, x_n)
    \mapsto
    (x'', x_{n-1} - \gamma(x''), x_n - \Gamma(x')),
\end{equation*}
and use the composition of $U_0$ with this map as a substitute for the model profile. However, the map is only Lipschitz, leaving us with no control on its second derivatives, which is needed
to analyze the resulting PDE.

Our main technical contribution is a regularized version of this flattening (see \cref{th:complex_coordinates} for a more precise version). Instead of using $\gamma$ and $\Gamma$ directly, we replace them by regularizations at a scale comparable to the distance to the slit. This produces a change of coordinates that stays close to the natural flattening, maps $\R^n\setminus\S$ onto the model slit domain $\R^n\setminus\H$, and additionally enjoys quantitative second-order estimates depending only on the Lipschitz constant of the slit. The regularization scale is what makes the second derivatives blow up no faster than the inverse distance to $\S$, which is the expected homogeneity.

\begin{theorem}[Regularized change of coordinates in Lipschitz slit domains] \label{th:complex_coordinates_intro}
    Let $\Omega = \R^n \setminus \S \subset \R^n$ be a Lipschitz slit domain in the sense of \cref{def:lip_slit_domain}, with Lipschitz constant $L \leq L_0$, where $L_0 = L_0(n) > 0$ is small enough. Then there exist $\zeta_1, \zeta_2 \in C^{1, 1}_{\mathrm{loc}}(\Omega)$ such that
    \begin{equation*}
        \Phi(y) \coloneqq \big(y'', \zeta_1(y), \zeta_2(y)\big)
    \end{equation*}
    maps $\Omega = \R^n \setminus \S$ onto the model slit domain $\R^n \setminus \H$. Moreover, there exists $C = C(n) > 0$ such that the following estimates hold for every $y \in \Omega$
    \begin{align*}
        \abs{\zeta_1(y) - (y_{n-1} - \gamma(y''))} +
        \abs{\zeta_2(y) - (y_n - \Gamma(y'))}
        & \leq
        CL \, \dist(y, \S),
        \\
        \abs{\nabla \zeta_1(y) - e_{n-1}} +
        \abs{\nabla \zeta_2(y) - e_n}
        & \leq
        CL,
    \end{align*} 
    and, for a.e.\ $y \in \Omega$,
    \begin{equation*}
        \abs{D^2 \zeta_1(y)} +
        \abs{D^2 \zeta_2(y)}
        \leq
        C\frac{L}{\dist(y, \S)}.
    \end{equation*}
\end{theorem}

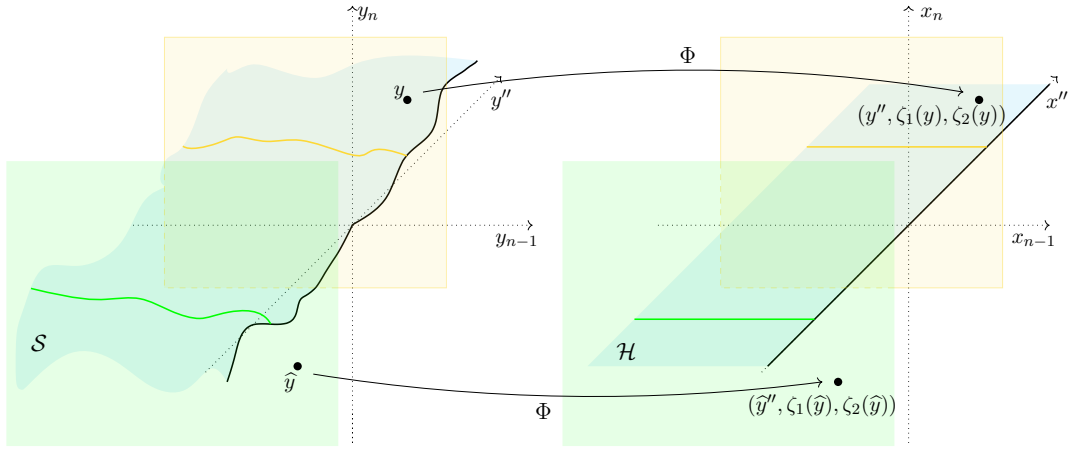
\begin{figure}[ht]
    \centering
    \resizebox{0.9\textwidth}{!}{
    \begin{tikzpicture}[scale=1]
        
        \draw[dotted, <-] (6.209, 31.289) -- (6.209, 23.386) -- (6.209, 23.951);
        \draw[dotted, ->] (2.258, 27.337) -- (9.454, 27.337);
        \draw[dotted, <-] (8.89, 30.019) -- (3.528, 24.656);        
        \draw[dotted, <-] (16.228, 31.289) -- (16.228, 23.386);
        \draw[dotted, ->] (11.712, 27.337) -- (18.768, 27.337);
        \draw[dotted, <-] (18.909, 30.019) -- (13.547, 24.656);
    
        \fill[cyan, opacity=0.1] 
        (6.209, 27.337)
         .. controls (6.209, 27.337) and (6.209, 27.337) .. (6.326, 27.408)
         .. controls (6.444, 27.479) and (6.679, 27.62) .. (6.82, 27.831)
         .. controls (6.961, 28.043) and (7.009, 28.325) .. (7.126, 28.513)
         .. controls (7.244, 28.702) and (7.432, 28.796) .. (7.549, 28.913)
         .. controls (7.667, 29.031) and (7.714, 29.172) .. (7.738, 29.337)
         .. controls (7.761, 29.501) and (7.761, 29.689) .. (7.855, 29.83)
         .. controls (7.949, 29.972) and (8.137, 30.066) .. (8.255, 30.136)
         .. controls (8.373, 30.207) and (8.42, 30.254) .. (8.443, 30.277)
         .. controls (8.467, 30.301) and (8.467, 30.301) .. (8.467, 30.301)
         .. controls (8.467, 30.301) and (8.467, 30.301) .. (8.467, 30.301)
         .. controls (8.467, 30.301) and (8.467, 30.301) .. (8.255, 30.324)
         .. controls (8.043, 30.348) and (7.62, 30.395) .. (7.244, 30.395)
         .. controls (6.867, 30.395) and (6.538, 30.348) .. (6.279, 30.254)
         .. controls (6.021, 30.16) and (5.833, 30.019) .. (5.597, 29.995)
         .. controls (5.362, 29.972) and (5.08, 30.066) .. (4.868, 30.136)
         .. controls (4.657, 30.207) and (4.516, 30.254) .. (4.374, 30.23)
         .. controls (4.233, 30.207) and (4.092, 30.113) .. (4.022, 30.089)
         .. controls (3.951, 30.066) and (3.951, 30.113) .. (3.951, 30.136)
         .. controls (3.951, 30.16) and (3.951, 30.16) .. (3.951, 30.16)
         .. controls (3.951, 30.16) and (3.951, 30.16) .. (3.881, 30.089)
         .. controls (3.81, 30.019) and (3.669, 29.877) .. (3.575, 29.713)
         .. controls (3.481, 29.548) and (3.434, 29.36) .. (3.363, 29.172)
         .. controls (3.293, 28.984) and (3.199, 28.796) .. (3.104, 28.702)
         .. controls (3.01, 28.607) and (2.916, 28.607) .. (2.846, 28.513)
         .. controls (2.775, 28.419) and (2.728, 28.231) .. (2.752, 28.137)
         .. controls (2.775, 28.043) and (2.869, 28.043) .. (2.822, 28.019)
         .. controls (2.775, 27.996) and (2.587, 27.949) .. (2.446, 27.761)
         .. controls (2.305, 27.573) and (2.211, 27.243) .. (2.023, 27.008)
         .. controls (1.834, 26.773) and (1.552, 26.632) .. (1.294, 26.608)
         .. controls (1.035, 26.585) and (0.8, 26.679) .. (0.635, 26.561)
         .. controls (0.47, 26.444) and (0.376, 26.114) .. (0.306, 25.903)
         .. controls (0.235, 25.691) and (0.188, 25.597) .. (0.165, 25.432)
         .. controls (0.141, 25.268) and (0.141, 25.033) .. (0.141, 24.844)
         .. controls (0.141, 24.656) and (0.141, 24.515) .. (0.165, 24.445)
         .. controls (0.188, 24.374) and (0.235, 24.374) .. (0.259, 24.374)
         .. controls (0.282, 24.374) and (0.282, 24.374) .. (0.423, 24.515)
         .. controls (0.564, 24.656) and (0.847, 24.939) .. (1.129, 25.033)
         .. controls (1.411, 25.127) and (1.693, 25.033) .. (1.929, 24.892)
         .. controls (2.164, 24.75) and (2.352, 24.562) .. (2.564, 24.445)
         .. controls (2.775, 24.327) and (3.01, 24.28) .. (3.199, 24.374)
         .. controls (3.387, 24.468) and (3.528, 24.703) .. (3.645, 24.75)
         .. controls (3.763, 24.797) and (3.857, 24.656) .. (3.904, 24.586)
         .. controls (3.951, 24.515) and (3.951, 24.515) .. (3.951, 24.515)
         .. controls (3.951, 24.515) and (3.951, 24.515) .. (3.951, 24.515)
         .. controls (3.951, 24.515) and (3.951, 24.515) .. (3.975, 24.586)
         .. controls (3.998, 24.656) and (4.045, 24.797) .. (4.092, 24.986)
         .. controls (4.139, 25.174) and (4.186, 25.409) .. (4.327, 25.503)
         .. controls (4.469, 25.597) and (4.704, 25.55) .. (4.868, 25.55)
         .. controls (5.033, 25.55) and (5.127, 25.597) .. (5.174, 25.691)
         .. controls (5.221, 25.785) and (5.221, 25.926) .. (5.268, 25.997)
         .. controls (5.315, 26.067) and (5.409, 26.067) .. (5.503, 26.162)
         .. controls (5.597, 26.256) and (5.691, 26.444) .. (5.762, 26.561)
         .. controls (5.833, 26.679) and (5.88, 26.726) .. (5.95, 26.844)
         .. controls (6.021, 26.961) and (6.115, 27.149) .. (6.162, 27.243)
         .. controls (6.209, 27.337) and (6.209, 27.337) .. cycle;

        \draw[thick] (3.951, 24.515) .. controls (3.985, 24.618) and (3.998, 24.659) .. (4.015, 24.713) .. controls (4.032, 24.768) and (4.052, 24.838) .. (4.065, 24.884) .. controls (4.078, 24.93) and (4.084, 24.953) .. (4.089, 24.975) .. controls (4.095, 24.997) and (4.1, 25.019) .. (4.106, 25.04) .. controls (4.111, 25.062) and (4.117, 25.084) .. (4.122, 25.106) .. controls (4.128, 25.127) and (4.133, 25.148) .. (4.139, 25.168) .. controls (4.145, 25.188) and (4.15, 25.207) .. (4.156, 25.226) .. controls (4.162, 25.244) and (4.168, 25.261) .. (4.177, 25.283) .. controls (4.185, 25.305) and (4.197, 25.331) .. (4.207, 25.352) .. controls (4.217, 25.373) and (4.225, 25.388) .. (4.24, 25.408) .. controls (4.254, 25.428) and (4.275, 25.453) .. (4.292, 25.471) .. controls (4.309, 25.488) and (4.323, 25.498) .. (4.344, 25.511) .. controls (4.366, 25.524) and (4.395, 25.539) .. (4.426, 25.547) .. controls (4.457, 25.555) and (4.489, 25.557) .. (4.516, 25.559) .. controls (4.543, 25.56) and (4.566, 25.561) .. (4.593, 25.561) .. controls (4.621, 25.561) and (4.652, 25.56) .. (4.681, 25.559) .. controls (4.71, 25.557) and (4.736, 25.556) .. (4.764, 25.554) .. controls (4.793, 25.553) and (4.823, 25.551) .. (4.86, 25.552) .. controls (4.896, 25.553) and (4.938, 25.556) .. (4.973, 25.562) .. controls (5.007, 25.567) and (5.034, 25.575) .. (5.056, 25.585) .. controls (5.079, 25.595) and (5.098, 25.606) .. (5.115, 25.621) .. controls (5.133, 25.636) and (5.148, 25.654) .. (5.161, 25.675) .. controls (5.174, 25.695) and (5.184, 25.718) .. (5.191, 25.737) .. controls (5.198, 25.756) and (5.202, 25.771) .. (5.207, 25.791) .. controls (5.212, 25.81) and (5.217, 25.836) .. (5.224, 25.862) .. controls (5.23, 25.889) and (5.237, 25.917) .. (5.245, 25.94) .. controls (5.253, 25.964) and (5.261, 25.981) .. (5.271, 25.996) .. controls (5.28, 26.01) and (5.29, 26.021) .. (5.302, 26.03) .. controls (5.313, 26.039) and (5.326, 26.047) .. (5.344, 26.057) .. controls (5.361, 26.067) and (5.384, 26.079) .. (5.416, 26.101) .. controls (5.448, 26.124) and (5.489, 26.157) .. (5.517, 26.184) .. controls (5.545, 26.211) and (5.561, 26.231) .. (5.58, 26.258) .. controls (5.599, 26.286) and (5.621, 26.321) .. (5.639, 26.35) .. controls (5.658, 26.38) and (5.671, 26.404) .. (5.686, 26.429) .. controls (5.7, 26.454) and (5.716, 26.481) .. (5.748, 26.532) .. controls (5.78, 26.583) and (5.829, 26.658) .. (5.86, 26.706) .. controls (5.892, 26.754) and (5.906, 26.774) .. (5.918, 26.793) .. controls (5.93, 26.812) and (5.941, 26.829) .. (5.951, 26.846) .. controls (5.962, 26.864) and (5.972, 26.882) .. (5.986, 26.906) .. controls (6, 26.931) and (6.017, 26.963) .. (6.034, 26.994) .. controls (6.051, 27.025) and (6.067, 27.057) .. (6.091, 27.103) .. controls (6.115, 27.15) and (6.146, 27.212) .. (6.167, 27.255) .. controls (6.188, 27.298) and (6.198, 27.323) .. (6.204, 27.335) .. controls (6.209, 27.347) and (6.209, 27.347) .. (6.248, 27.369) .. controls (6.287, 27.392) and (6.365, 27.436) .. (6.431, 27.478) .. controls (6.497, 27.52) and (6.55, 27.559) .. (6.594, 27.594) .. controls (6.637, 27.629) and (6.671, 27.659) .. (6.703, 27.693) .. controls (6.735, 27.726) and (6.766, 27.762) .. (6.795, 27.801) .. controls (6.823, 27.841) and (6.85, 27.883) .. (6.881, 27.948) .. controls (6.912, 28.013) and (6.948, 28.1) .. (6.972, 28.161) .. controls (6.997, 28.222) and (7.01, 28.257) .. (7.029, 28.304) .. controls (7.049, 28.351) and (7.074, 28.41) .. (7.094, 28.45) .. controls (7.113, 28.49) and (7.126, 28.512) .. (7.145, 28.537) .. controls (7.163, 28.563) and (7.187, 28.593) .. (7.215, 28.624) .. controls (7.244, 28.654) and (7.277, 28.685) .. (7.318, 28.721) .. controls (7.36, 28.756) and (7.41, 28.796) .. (7.45, 28.828) .. controls (7.489, 28.861) and (7.519, 28.887) .. (7.545, 28.913) .. controls (7.571, 28.939) and (7.594, 28.966) .. (7.61, 28.988) .. controls (7.627, 29.009) and (7.637, 29.025) .. (7.646, 29.041) .. controls (7.655, 29.057) and (7.663, 29.073) .. (7.67, 29.089) .. controls (7.677, 29.105) and (7.684, 29.122) .. (7.69, 29.138) .. controls (7.696, 29.155) and (7.702, 29.171) .. (7.706, 29.188) .. controls (7.711, 29.206) and (7.716, 29.223) .. (7.722, 29.255) .. controls (7.728, 29.287) and (7.736, 29.334) .. (7.741, 29.369) .. controls (7.746, 29.403) and (7.749, 29.426) .. (7.754, 29.465) .. controls (7.759, 29.504) and (7.766, 29.558) .. (7.775, 29.604) .. controls (7.784, 29.651) and (7.794, 29.69) .. (7.8, 29.71) .. controls (7.805, 29.73) and (7.805, 29.73) .. (7.811, 29.743) .. controls (7.817, 29.756) and (7.829, 29.782) .. (7.841, 29.804) .. controls (7.854, 29.826) and (7.867, 29.845) .. (7.883, 29.865) .. controls (7.9, 29.885) and (7.919, 29.905) .. (7.942, 29.926) .. controls (7.964, 29.946) and (7.989, 29.967) .. (8.021, 29.989) .. controls (8.052, 30.012) and (8.091, 30.037) .. (8.128, 30.06) .. controls (8.165, 30.083) and (8.2, 30.104) .. (8.231, 30.128) .. controls (8.262, 30.151) and (8.289, 30.178) .. (8.326, 30.202) .. controls (8.363, 30.226) and (8.411, 30.247) .. (8.459, 30.302);

        \fill[cyan, opacity=0.1] (13.688, 24.797) -- (10.442, 24.797) -- (15.522, 29.877) -- (18.768, 29.877) -- cycle;

        \draw[thick] (13.688, 24.797) -- (18.768, 29.877);

        \fill[myyellow, opacity=0.1] (2.822, 30.724) rectangle (7.902, 26.209);
        \draw[myyellow, opacity=0.4] (2.822, 28.466) .. controls (2.822, 30.724) and (2.822, 30.724) .. (2.822, 30.724) .. controls (2.822, 30.724) and (2.822, 30.724) .. (3.669, 30.724) .. controls (4.516, 30.724) and (6.209, 30.724) .. (7.056, 30.724) .. controls (7.902, 30.724) and (7.902, 30.724) .. (7.902, 30.724) .. controls (7.902, 30.724) and (7.902, 30.724) .. (7.902, 29.972) .. controls (7.902, 29.219) and (7.902, 27.714) .. (7.902, 26.961) .. controls (7.902, 26.209) and (7.902, 26.209) .. (7.902, 26.209) .. controls (7.902, 26.209) and (7.902, 26.209) .. (5.927, 26.209);
        \draw[myyellow, dashed, opacity=0.4] (2.822, 28.466) .. controls (2.822, 26.209) and (2.822, 26.209) .. (2.822, 26.209) .. controls (2.822, 26.209) and (2.822, 26.209) .. (5.927, 26.209);
        
        \draw[myyellow, thick] (7.188, 28.588) .. controls (6.635, 28.794) and (6.548, 28.675) .. (6.452, 28.622) .. controls (6.355, 28.569) and (6.25, 28.583) .. (6.044, 28.648) .. controls (5.838, 28.714) and (5.533, 28.833) .. (5.286, 28.865) .. controls (5.04, 28.896) and (4.852, 28.841) .. (4.664, 28.848) .. controls (4.475, 28.854) and (4.285, 28.923) .. (4.155, 28.928) .. controls (4.024, 28.933) and (3.953, 28.875) .. (3.763, 28.811) .. controls (3.574, 28.746) and (3.265, 28.675) .. (3.154, 28.762);

        \fill[myyellow, opacity=.1] (12.841, 30.724) rectangle (17.921, 26.209);
        \draw[myyellow, opacity=0.4] (12.841, 28.466) .. controls (12.841, 30.724) and (12.841, 30.724) .. (12.841, 30.724) .. controls (12.841, 30.724) and (12.841, 30.724) .. (13.688, 30.724) .. controls (14.534, 30.724) and (16.228, 30.724) .. (17.074, 30.724) .. controls (17.921, 30.724) and (17.921, 30.724) .. (17.921, 30.724) .. controls (17.921, 30.724) and (17.921, 30.724) .. (17.921, 29.972) .. controls (17.921, 29.219) and (17.921, 27.714) .. (17.921, 26.961) .. controls (17.921, 26.209) and (17.921, 26.209) .. (17.921, 26.209) .. controls (17.921, 26.209) and (17.921, 26.209) .. (15.946, 26.209);
        \draw[myyellow, dashed, opacity=.4] (12.841, 28.466) -- (12.841, 26.209) -- (15.946, 26.209);

        \draw[myyellow, thick] (17.639, 28.749) -- (14.393, 28.749);

        \filldraw[draw=green, thin, fill=green, opacity=0.1] (0, 28.466) rectangle (5.927, 23.386);

        \draw[green, thick] (4.736, 25.558) .. controls (4.586, 25.843) and (4.214, 25.815) .. (3.97, 25.77) .. controls (3.726, 25.724) and (3.61, 25.661) .. (3.459, 25.658) .. controls (3.307, 25.656) and (3.12, 25.714) .. (2.935, 25.794) .. controls (2.751, 25.875) and (2.569, 25.978) .. (2.382, 26.013) .. controls (2.195, 26.049) and (2.002, 26.017) .. (1.769, 26.003) .. controls (1.537, 25.989) and (1.264, 25.993) .. (0.417, 26.199);
        
        \filldraw[draw=green, thin, fill=green, opacity=0.1] (10.019, 28.466) rectangle (15.946, 23.386);

        \draw[green, thick] (14.534, 25.644) -- (11.289, 25.644);

        \draw[->] (7.479, 29.736) arc[start angle=-99.298, end angle=-80.702, x radius=30.132, y radius=-30.132];
        \draw[->] (5.503, 24.656) arc[start angle=-99.107, end angle=-82.656, radius=32.059];

        \node[circle, fill, inner sep=1.5pt] at (7.197, 29.595) {};
        \node[circle, fill, inner sep=1.5pt] at (17.498, 29.595) {};
        \node[circle, fill, inner sep=1.5pt] at (14.958, 24.515) {};
        \node[circle, fill, inner sep=1.5pt] at (5.221, 24.797) {};

        \node[anchor=center] at (7.056, 29.736) {$y$};
        \node[anchor=center] at (16.651, 29.313) {$(y'', \zeta_1(y), \zeta_2(y))$};
        \node[anchor=center] at (5.08, 24.515) {$\widehat{y}$};
        \node[anchor=center] at (14.676, 24.092) {$(\widehat{y}'', \zeta_1(\widehat{y}), \zeta_2(\widehat{y}))$};
        \node[anchor=center] at (9.172, 27.055) {$y_{n-1}$};
        \node[anchor=center] at (8.89, 29.595) {$y''$};
        \node[anchor=center, font=\large] at (12.25, 30.4) {$\Phi$};
        \node[anchor=center, font=\large] at (9.65, 24) {$\Phi$};
        \node[anchor=center] at (6.491, 31.147) {$y_n$};
        \node[anchor=center] at (16.651, 31.147) {$x_n$};
        \node[anchor=center] at (18.909, 29.595) {$x''$};
        \node[anchor=center] at (18.486, 27.055) {$x_{n-1}$};
        \node[anchor=center, font=\large] at (0.564, 25.221) {$\mathcal{S}$};
        \node[anchor=center, font=\large] at (11.148, 25.08) {$\mathcal{H}$};
    \end{tikzpicture}
    }
    \caption{The change of coordinates sends $\Omega = \R^n \setminus \S$ to the model flat slit domain $\R^n \setminus \H$. All points $y \in \Omega$ sharing the same first coordinates $y''$ (green and yellow planes in the figure) get sent to the same plane, whose section is just the complex flat slit domain $\CC \setminus \H_\CC$.}
    \label{fig:coordinates}
\end{figure} 

We also remark that our change of coordinates should be distinguished from two similar constructions in the literature. First, it is not a regularized distance in the sense of Lieberman \cite{Lieb} as, in our case, what is regularized is the slit itself (at a scale comparable to the distance to it): the output is a map onto a model domain, rather than a substitute for the distance function. Moreover, it is also not the pair of edge-adapted coordinates used by De Silva and Savin \cite{SilvaSavinBHSlit, SilvaSavinThinOnePhase}, which measures the distance to an edge already assumed to lie in a hyperplane; the point of \cref{th:complex_coordinates_intro} is precisely that no such hyperplane is available. Since the construction is geometric and not tied to any particular equation, we expect it to be useful for other problems.

Using this regularized flattening, we construct quantitative sub- and superharmonic barriers (\cref{th:barriers}) that play the role of $U_0$ in Lipschitz slit domains. These barriers are the main ingredient in the proofs of \cref{th:lower_bound_thm,th:upper_bound_thm}.

Finally, the correction factor in those theorems has a simple multiscale explanation. The proof builds on the multiscale iteration developed in \cite{ClaraC1}, whose underlying philosophy goes back to Ladyzhenskaya and Uraltseva~\cite{LU}: at each scale, one approximates the solution by a profile adapted to the geometry at that scale, and the successive approximation errors are controlled by the modulus of continuity of the boundary. In our setting, the profiles are the barriers from the previous paragraph. Passing from one scale to the next requires comparing barriers adapted to consecutive dyadic scales, which introduces a multiplicative error of the form $1\pm\mathcal{O}(\omega(2^{-k}))$, the sign depending on whether one considers upper or lower bounds (see \cref{sec:comparison_barriers}). As these errors accumulate throughout the iteration, one is naturally led to
\begin{equation*}
    \prod_{k = 1}^{m} \bigl(1\pm\mathcal{O}(\omega(2^{-k}))\bigr)
    \approx \exp\!\left(\pm C\sum_{k = 1}^m \omega(2^{-k})\right)
    \approx \exp\!\left(\pm C\int_{2^{-m}}^1 \omega(s)\,\frac{ds}{s}\right),
\end{equation*}
which is the correction factor measuring the changes of the geometry across scales.

\subsection{Historical comments}

Boundary regularity for harmonic functions in slit domains is closely tied to the Signorini problem. For slit domains whose slit lies in a fixed hyperplane, De Silva and Savin proved in \cite{SilvaSavinBHSlit} an expansion of $u$ against an explicit half-power profile when the edge is $C^{1,\alpha}$, and a higher order boundary Harnack principle in the case of a $C^{k,\alpha}$ edge, giving smoothness of the free boundary at regular points, with a parallel result for the thin one-phase problem in \cite{SilvaSavinThinOnePhase} (the corresponding statement in thick $C^{k,\alpha}$ domains is \cite{SilvaSavinThick}). The Lipschitz case, still for slits inside a hyperplane, is due to De Silva and Savin \cite{DS20}, and to Petrosyan and Shi \cite{PetrosyanShi14} in the parabolic setting; Ros-Oton and the fourth author \cite{ClaraBoundaryHarnack} allow the slit to lie inside a Lipschitz graph and add a right-hand side. Koch, Petrosyan and Shi \cite{KochPetrosyanShi2015} used a boundary Hopf estimate in $C^{1,\alpha}$ slit domains to prove analyticity of the Signorini free boundary. In these works, the conclusion is a bound on the quotient of two solutions; we are not aware of any result identifying the growth rate itself under hypotheses weaker than a $C^{1,\alpha}$ edge inside a fixed hyperplane, nor of any treatment of the Dini case for slit domains.

Regularized distances have long been a useful tool in classical analysis (see, for instance, \cite[Theorem 2.1]{Necas}, \cite[Section VI.2.1]{Stein}, \cite[Section 3.2]{Triebel}), and their explicit use in the construction of PDE barriers originates in the work of Lieberman \cite{Lieb}. Since then, this technique has proven remarkably versatile for studying fine boundary regularity in more complex settings. Indeed, regularized distances have been successfully adapted to parabolic equations \cite{Lieb_book,PedraClara}, free boundary problems \cite{RosOtonRestrepo2025} or higher codimensional boundaries \cite{DFM}. For a more comprehensive historical discussion, we refer to \cite[Section 1.2]{ClaraC1}.

Multiscale iterations to obtain fine boundary regularity properties of solutions are classical in the field of PDE, starting with Wiener's construction \cite{Wiener24} of barriers from equilibrium potentials, which characterizes the regularity of boundary points through integrals involving capacities. Quantitative versions were derived by Maz'ya \cite{Mazya63, Mazya84} using energy methods. Exact boundary asymptotics for divergence-form PDEs in Lipschitz domains were developed by Kozlov and Maz'ya \cite{KozlovMazya03, KozlovMazya05} (see also \cite{MazyaSurvey11}) by reducing the boundary value problem, through a change of variables, to an evolution equation. Their arguments are potential-theoretic and spectral rather than based on explicit pointwise barriers, as opposed to our more geometric approach. The work of Ladyzhenskaya and Uraltseva \cite{LU} is closer in spirit to our results. There, barriers are constructed in a more geometric fashion, adapted to each scale, and not necessarily combined into a global barrier. This way, one obtains estimates involving the modulus of continuity of the boundary instead of capacitary objects. The strategy was refined by the fourth author in \cite{ClaraC1}, whose framework we follow and expand.


\subsection{Organization of the paper} \label{subsec:organization}
In \cref{sec:notation_and_setting} we introduce the main definitions and collect several preliminary results used throughout the paper. In \cref{sec:reg_complex_coord} we introduce the regularized complex coordinates and prove several growth estimates associated with them. In \cref{sec:barriers} we use the regularized complex coordinates to construct barriers mimicking the 2-dimensional case. In \cref{sec:C1_slit_domains} we combine the results from the previous sections to prove \cref{th:lower_bound_thm,th:upper_bound_thm,cor:boundary_regularity}. Finally, \cref{sec:sec_ord_IFT} contains a calculation of the second derivatives of an inverse function, and \cref{sec:sec_corollary_appendix} contains an auxiliary lemma used in the proof of \cref{cor:boundary_regularity}.

\resizebox{.95\textwidth}{!}{
    \begin{tikzpicture}[
      box/.style={draw, rounded corners=3pt, align=center,
        text width=3.75cm, minimum height=2.6cm, inner sep=5pt,
        font=\small},
      auxiliary/.style={box, fill=black!3, minimum height=1.25cm},
      arrow/.style={-{Latex[length=2mm]}, line width=.5pt},
      method/.style={font=\scriptsize, align=center, inner sep=2pt},
      node distance=11mm and 17mm
    ]
    \node[box] (coordinates) {Regularized coordinates\\in Lipschitz\\slit domains\\[3pt]
      {\footnotesize \cref{sec:reg_complex_coord}\\\cref{th:complex_coordinates}}};
    \node[box, right=of coordinates] (barriers) {Barriers in\\Lipschitz slit domains\\[3pt]
      {\footnotesize \cref{sec:barriers}}};
    \node[box, right=of barriers] (comparison) {Comparison of barriers\\at consecutive scales\\[3pt]
      {\footnotesize \cref{sec:comparison_barriers}}};
    \node[box, right=of comparison] (growth) {Lower and upper\\boundary growth\\estimates\\[3pt]
      {\footnotesize \cref{subsec:lower_bound_thm_proof,subsec:upper_bound_thm_proof}\\\cref{th:lower_bound_thm,th:upper_bound_thm}}};
    \node[auxiliary, above=of barriers] (positivity) {Almost positivity\\of harmonic functions\\[3pt]
      {\footnotesize \cref{subsec:almost_positivity}}};
    \node[auxiliary, above=of comparison] (harnack) 
      {Boundary Harnack\\inequality\\[3pt]
      {\footnotesize \cref{th:bdry_harnack}}};
    \node[box, below=of growth] (regularity) {Boundary regularity\\[3pt]
      {\footnotesize \cref{subsec:boundary_regularity_proof}\\\cref{cor:boundary_regularity}}};
    \node[auxiliary] (inhomogeneous) at (growth.center |- positivity.center)
      {$L^\infty$ estimate for\\inhomogeneous equations\\[3pt]
      {\footnotesize \cref{th:GilbargTrudinger_8.16}}};
    \draw[arrow] (inhomogeneous.south) -- (growth.north);
    \draw[arrow] (coordinates) -- node[method, above] {Square-root\\construction} (barriers);
    \draw[arrow] (barriers) -- (comparison);
    \draw[arrow] (harnack.south) -- (comparison.north);
    \draw[arrow] (harnack) -- (growth);
    \draw[arrow] (positivity) -- (comparison);
    \draw[arrow] (comparison) -- node[method, above] {Multiscale\\iteration} (growth);
    \draw[arrow] (growth) -- node[method, right] {Upper bound\\+ interior estimates} (regularity);
    \end{tikzpicture}
}

\subsection{Acknowledgments} JD and AM were supported by the European Research Council under the Grant Agreement No. 101123223 (SSNSD), and by AEI project PID2024-156429NB-I00 (Spain). PHP was supported by the project PCI2024-155066-2, funded by \mbox{MICIU}/\allowbreak AEI/\allowbreak 10.13039/\allowbreak 501100011033 and cofunded by the European Union. CTL has received funding from the European Research Council (ERC) under the Grant Agreement No. 862342, from AEI project PID2024-156429NB-I00 (Spain), and from the Grant CEX2023-001347-S funded by \mbox{MICIU}/\allowbreak AEI/\allowbreak 10.13039/\allowbreak 501100011033 (Spain).

\subsection{Declaration of AI usage} The mathematical ideas and arguments of this work were devised, developed and written by the authors before July 2026, except \cref{sec:sec_corollary_appendix}, which was written from September 2026 onwards using AI assistance and verified by the authors. AI was also used for stylistic reviews of the manuscript, and for suggestions regarding the exposition in the introduction. The figures were created by the authors, with the exception of the box diagram in Subsection~\ref{subsec:organization}, which was drawn using AI. We take full responsibility for the content of the manuscript.


\section{Preliminaries}
\label{sec:notation_and_setting}

\subsection{Notation and setting} \label{subsec:notation}

Throughout the paper the dimension will be $n \geq 2$. We denote by $e_1, \ldots, e_n$ the canonical basis of $\R^n$. For $x \in \R^n$, we will use the notation $x'' = (x_1,\ldots,x_{n-2})$ and $x' = (x'',x_{n-1})$ so that $x = (x',x_n) = (x'',x_{n-1},x_n)$. We will also write $B''_r(x'')$ to denote balls in $\R^{n-2}$ and $B'_r(x')$ to denote balls in $\R^{n-1}$. Whenever the center of these balls is the origin, we will simply abbreviate the notation to $B''_r$ and $B'_r$, respectively.

Since our study will be connected to the model case of the complex square root, as explained in \cref{sec:motivation}, it will be useful to denote the model flat slit, which is just a half-hyperplane,
\begin{equation*}
    \H
    \coloneqq
    \left\{ (x'', x_{n-1}, x_n) \in \R^{n-2} \times \R \times \R : \; x_{n-1} \leq 0 \text{ and } x_n = 0 \right\}.
\end{equation*}
Also, we will denote by
\begin{equation} \label{def:rho}
    \rho(x) = \rho(x'',x_{n-1},x_n) \coloneqq \sqrt{(x_{n-1})_+^2 + x_n^2}
\end{equation}
the distance function to the model slit $\H$. Since $\rho$ only depends on $x_{n-1}$ and $x_n$, we will frequently make an abuse of notation and write $\rho(x_{n-1},x_n)$ instead of $\rho(x)$.

\begin{remark}
    The following properties of $\rho$ directly stem from the definition:
    \begin{itemize}
        \item $\rho = 0$ on the model slit $\H$, and $\rho > 0$ in $\R^n \setminus \H$,
        \item $\rho(x) = |x_n|$ when $x_{n-1} \leq 0$, and $\rho(x) = |(x_{n-1}, x_n)| = \sqrt{x_{n-1}^2 + x_n^2}$ when $x_{n-1} > 0$.
    \end{itemize}
\end{remark}

The main object of study throughout the paper are slit domains:
\begin{definition}[Slit domain] \label{def:slit_domain}
    Given $\gamma : \R^{n-2} \rightarrow \R$ and $\Gamma : \R^{n-1} \rightarrow \R$, we define the slit domain $\Omega_{\gamma, \Gamma} \subset \R^n$ by
    \begin{equation*}
        \Omega_{\gamma, \Gamma}
        \coloneqq
        \R^n \setminus \left\{ x \in \R^n : \; x_{n-1} \leq \gamma(x'') \text{ and } x_n = \Gamma(x') \right\}.
    \end{equation*}
\end{definition}
Given a slit domain $\Omega = \Omega_{\gamma, \Gamma}$ as in \cref{def:slit_domain} (we will omit the sub-indices when they are clear from the context), we will denote the \textbf{slit} by

\begin{equation*}
    \S
    \coloneqq
    \left\{ x \in \R^n : \; x_{n-1} \leq \gamma(x'') \text{ and } x_n = \Gamma(x') \right\},
\end{equation*}
(see \cref{fig:slit}) so that
\begin{equation*}
    \Omega
    =
    \R^n \setminus \S.
\end{equation*}
    
A relevant part of the slit is its boundary (as a lower dimensional object), what we call the \textbf{edge}:
\begin{equation*}
    \mathcal{E}
    \coloneqq
    \left\{ x \in \R^n : \; x_{n-1} = \gamma(x'') \text{ and } x_n = \Gamma(x') \right\}
    =
    \left\{ (x'', \gamma(x''), \Gamma(x'',\gamma(x''))) : x'' \in \R^{n-2} \right\}.
\end{equation*}
    
In fact, the most interesting analysis of the boundary behavior of harmonic functions in slit domains happens near the edge, because away from the edge, the domain is locally the union of two classical Lipschitz domains lying on either side of the slit, hence well-established theory applies (see e.g. \cite{MazyaBook2018}). 

Also, without loss of generality we will assume throughout the rest of the paper that $\gamma(0) = \Gamma(0) = 0$ and $\nabla\gamma(0) = \nabla\Gamma(0) = 0$. To prove the main results of the paper we will mostly work with Lipschitz slit domains.
\begin{definition}[Lipschitz slit domain] \label{def:lip_slit_domain}
    We say a slit domain $\Omega = \Omega_{\gamma,\Gamma}$ is a Lipschitz slit domain with Lipschitz constant $L$ if $\gamma$ and $\Gamma$ are Lipschitz continuous functions and 
    \begin{equation*}
        \norm{\nabla\gamma}_{L^\infty} \leq L,
        \qquad 
        \norm{\nabla\Gamma}_{L^\infty} \leq L.
    \end{equation*}
\end{definition}

\begin{definition}[Modulus of continuity]
    We say $\omega:[0,t_0) \to [0,\infty)$ is a modulus of continuity if it is continuous, strictly increasing and $\omega(0) = 0$.
\end{definition}

\begin{definition}[$C^1$-modulus of continuity of a slit domain] \label{def:C1_slit_domain}
    Let $\Omega = \Omega_{\gamma,\Gamma}$ be a Lipschitz slit domain in the sense of \cref{def:lip_slit_domain}. We say $\Omega$ is a slit domain with $C^1$-modulus of continuity $\omega$ at a given point $x_0 \in \partial\Omega = \S$ if, for all small $r$,
    \begin{equation*}
        \norm{\nabla\gamma - \nabla\gamma(x_0'')}_{L^\infty(B''_r(x_0''))} \leq \omega(r), \qquad
        \norm{\nabla\Gamma - \nabla\Gamma(x_0')}_{L^\infty(B'_r(x_0'))} \leq \omega(r).
    \end{equation*}
    Given an open set $U \subset \R^n$, we say $\Omega$ has $C^1$-modulus of continuity $\omega$ in $U$ if the two inequalities above hold for every $x_0 \in \S \cap U$ and for all sufficiently small $r$, uniformly in $x_0$.
\end{definition}
Note that it follows directly from the definition that if $\Omega$ has $C^1$-modulus $\omega$ at $x_0$, then, after a rotation so that $\nabla\gamma(x_0'') = 0$ and $\nabla\Gamma(x_0') = 0$, for every $L > 0$ there is $r_L > 0$ such that $\Omega \cap B_{r_L}(x_0)$ is a Lipschitz slit domain with constant $L$. Additionally, since all our main results are local in nature, we may replace $\gamma$ and $\Gamma$ outside $B_{r_L}(x_0)$ with Lipschitz extensions with the same constant $L$ so that the Lipschitz constant of $\Omega$ is $L$ globally (as opposed to only in $B_{r_L}(x_0)$). This will be needed to apply some of our results concerning Lipschitz slit domains such as \cref{th:complex_coordinates,lem:spherical_distance_boundary,lem:h_comparable_sqrt_r}.

\begin{definition}[$C^{\frac{1}{2}}$-modulus of continuity of functions] \label{def:C1/2_modulus_of_cont}
    Let $\omega$ be a modulus of continuity. We say $g$ has $C^{\frac{1}{2}}$-modulus of continuity $\omega$ at $x_0 \in \R^n$ if, for all small $r$,
    \begin{equation*}
      \norm{g(\cdot) - g(x_0)}_{L^\infty(B_r(x_0))} \leq r^{1/2} \omega(r).
    \end{equation*}
    Given an open set $U \subset \R^n$, we say $g$ has $C^{\frac{1}{2}}$-modulus of continuity $\omega$ in $U$ if the above inequality holds for every $x_0 \in \S \cap U$ and for all sufficiently small $r$, uniformly in $x_0$.
\end{definition}


Let us also include an elementary geometric fact about our choice of coordinates for \cref{th:lower_bound_thm} (see also \cref{fig:cones}) which, roughly speaking, explains that these are appropriate non-tangential regions, in the sense that they are conical and stay away from the slit at small enough scales (or equivalently, if the slit is flat enough).

\begin{lemma} \label{lem:spherical_distance_boundary}
    Let $0 < \alpha_0 < \pi$, $0 < \beta_0 < \pi/2$. Let $\Omega$ be a Lipschitz slit domain, with small enough Lipschitz constant depending on $\alpha_0$ and $\beta_0$. Then there exists $c_0 = c_0(n, \alpha_0, \beta_0) > 0$ such that, for any unit vector $e'' \in \R^{n-2}$, and angles $\abs{\alpha} \leq \alpha_0$, $\abs{\beta} \leq \beta_0$, if we define 
    \begin{equation*}
        e \coloneqq e(\alpha, \beta) \coloneqq (e''\sin\beta, \cos \alpha \cos \beta, \sin \alpha \cos \beta),
    \end{equation*}
    then, for any $r > 0$, it holds 
    \begin{equation*}
        \dist(re, \partial \Omega) \geq c_0 r.
    \end{equation*}
\end{lemma}
\begin{proof}
    Define $\widehat{\alpha} \coloneqq (\alpha_0 + \pi) / 2 \in (\alpha_0, \pi)$ and $\widehat{\beta} \coloneqq (\beta_0 + \pi/2) / 2 \in (\beta_0, \pi/2)$. With these angles, we denote the conical regions 
    \begin{align*}
        \mathcal{C} &\coloneqq \{ te(\alpha, \beta) : t \geq 0,\, e'' \in \mathbb{S}^{n-3},\, \abs{\alpha} \leq \alpha_0,\, \abs{\beta} \leq \beta_0 \}, \\
        \widehat{\mathcal{C}} &\coloneqq \{ te(\alpha, \beta) : t \geq 0,\, e'' \in \mathbb{S}^{n-3},\, \abs{\alpha} \leq \widehat{\alpha},\, \abs{\beta} \leq \widehat{\beta} \}.
    \end{align*}
    Upon choosing the Lipschitz constant of $\Omega$ small enough depending on $\alpha_0$ and $\beta_0$, it is clear that $\widehat{\mathcal{C}} \setminus \{0\} \subset \Omega$ (see \cref{fig:cones} for intuition). This implies that, for $e$ as in the statement, 
    \begin{equation*}
        \dist(re, \partial \Omega) 
        \geq 
        \dist(re, \partial \widehat{\mathcal{C}})
        =
        r \dist(e, \partial \widehat{\mathcal{C}}),
    \end{equation*}
    where in the last equality we have used the homogeneity of the conical regions. It is now elementary to check that, since $e$ is a unit vector inside the conical region $\mathcal{C}$, its distance to $\partial \widehat{\mathcal{C}}$ only depends on $\alpha_0, \beta_0$ and $n$, which finishes the proof.    
\end{proof}


\subsection{Auxiliary results}

We will make use of the following $L^\infty$ bound for weak subsolutions to inhomogeneous equations.
\begin{theorem}[{\cite[Theorem 8.16]{GilbargTrudinger}}]
    \label{th:GilbargTrudinger_8.16}
    Assume that $\Omega \subset \R^n$ is a bounded domain. Let $u \in W^{1,2}(\Omega) \cap C(\overline{\Omega})$ satisfy 
    \begin{equation*}
        -\Delta u \le f \quad \text{weakly in } \Omega,
    \end{equation*}
    with $f \in L^p(\Omega)$ and $p > \frac{n}{2}$. Then, it holds 
    \begin{equation*}
        \sup_{\Omega} u \le \sup_{\partial\Omega} u^+ + C(n,p)\diam(\Omega)^{2-\frac{n}{p}}\norm{f}_{L^p(\Omega)}.
    \end{equation*}
\end{theorem}

We will also utilize the following version of the boundary Harnack for harmonic functions in slit domains.
\begin{theorem}[{\cite[Theorem 3.4]{DS20}}]
    \label{th:bdry_harnack}
    Let $\Omega$ be a Lipschitz slit domain as in \cref{def:lip_slit_domain} with Lipschitz constant $L$. Let $u$ and $v$ be non-negative functions vanishing continuously on $\S \cap B_1$ and such that
    \begin{equation*}
        \Delta u = 0 = \Delta v \quad\text{in } \Omega \cap B_1.
    \end{equation*}
    Furthermore, assume $v(\frac{e_n}{2}) = 1 = v(-\frac{e_n}{2})$. Then, there exists $C = C(n,L) > 0$ such that
    \begin{equation*}
        C^{-1}\min\left\{u\left(\frac{e_n}{2}\right),u\left(-\frac{e_n}{2}\right)\right\}
        \leq
        \frac{u}{v}
        \leq
        C\max\left\{u\left(\frac{e_n}{2}\right),u\left(-\frac{e_n}{2}\right)\right\}
        \quad\text{in } \Omega \cap B_{1/2}.
    \end{equation*}
\end{theorem}

\begin{remark} \label{rmk:bdry_harnack_one_side_condition}
    As observed in \cite[Section 3.3]{DS20}, the conditions on $u$ and $v$ at the point $-\frac{e_n}{2}$ can be dropped from the statement of \cref{th:bdry_harnack} provided that there is a Harnack chain (that is, overlapping balls whose radius is comparable to its distance to $\partial \Omega$) of radius $\mu$ connecting the points $\pm\frac{e_n}{2}$, and in such case the constant $C$ will also depend on $\mu$.
\end{remark}

\begin{remark}\label{rmk:bdry_harnack_neg}
    As a consequence of \cref{th:bdry_harnack}, if $u$ is any function, not necessarily non-negative, vanishing on $\S \cap B_1$ and such that $\Delta u = 0$ in $\Omega \cap B_1$, and $v$ is as in \cref{th:bdry_harnack}, 
    \begin{equation*}
        \frac{u}{v} \leq C\|u_+\|_{L^\infty(\Omega \cap B_1)} \quad\text{in } \Omega \cap B_{1/2}.
    \end{equation*}
    Indeed, consider the harmonic replacement $h$ of $u_+$ in $\Omega \cap B_1$:
    \begin{equation*}
        \left\{
        \begin{aligned}
            \Delta h &= 0 \quad&\text{in } \Omega \cap B_1, \\
            h &= u_+ &\text{on } \partial(\Omega \cap B_1).
        \end{aligned}
        \right.
    \end{equation*}
    Then, by construction $h = u_+ = 0$ on $\S \cap B_1$ and by the comparison principle $u \leq h$ in $\Omega \cap B_1$ and $0 \leq h \leq \|u_+\|_{L^\infty(\Omega \cap B_1)}$. Hence, applying \cref{th:bdry_harnack} to $h$ we obtain, in $\Omega \cap B_{1/2}$,
    \begin{align*}
        \frac{u}{v} \! \leq \!
        \frac{h}{v} \! \leq \! 
        C \max\big\{h({\textstyle \frac{e_n}{2}}),h({\textstyle -\frac{e_n}{2}})\big\} \! \leq
        C\|u_+\|_{L^\infty(\Omega \cap B_1)}.
    \end{align*}
\end{remark}


\section{Regularized complex coordinates in Lipschitz slit domains} \label{sec:reg_complex_coord}

The main technical contribution of this paper is the construction of a regularized flattening of Lipschitz slit domains. In this section, we build a change of coordinates mapping
$\Omega=\R^n\setminus \S$ onto the model slit domain $\R^n\setminus \H$, while retaining quantitative control up to second derivatives.

The naive flattening
\begin{equation*}
    \Omega = \R^n \setminus \S \ni y = (y'', y_{n-1}, y_n) 
    \mapsto 
    (y'', y_{n-1} - \gamma(y''), y_n - \Gamma(y')) \in \R^n \setminus \H,
\end{equation*}
already maps the domains correctly, but is only Lipschitz and therefore unsuitable for our applications to second-order PDE. The construction below provides a regularized substitute that remains close to this natural map, but additionally possesses the second-order estimates required for the barrier arguments of the next section.

Let us introduce $\sigma$ as the quantity that controls the oscillations of $\gamma, \Gamma$ at scale of $\rho(x) = ((x_{n-1})_+^2 + x_n^2)^{1/2}$ from \eqref{def:rho}:
\begin{equation} \label{def:sigma}
    \sigma(x)
    \coloneqq
    \norm{\nabla\gamma}_{L^\infty(B''_{\rho(x)}(x''))}
    + \norm{\nabla\Gamma}_{L^\infty(B'_{2\rho(x)}(x'', x_{n-1} + \gamma(x'')))}.
\end{equation}
With this in mind, the main result of this section, which is a more precise version of \cref{th:complex_coordinates_intro}, reads:

\begin{theorem} \label{th:complex_coordinates}
    Let $\Omega \subset \R^n$ be a Lipschitz slit domain as in \cref{def:lip_slit_domain} with Lipschitz constant $L$. If $L > 0$ is small enough, then there exist functions $\zeta_1, \zeta_2 : \Omega \rightarrow \R$, with $\zeta_i \in C^{1, 1}_{\mathrm{loc}}(\Omega)$, which map $\Omega$ to our model slit domain, namely:
    \begin{equation} \label{eq:zeta_maps}
        y = (y'', y_{n-1}, y_n) \in \Omega 
        \; \implies \; 
        (y'', \zeta_1(y), \zeta_2(y)) \in \R^n \setminus \H,
    \end{equation}
    Moreover, there exists $C > 0$ depending only on $n$, so that if we denote $\zeta(y) \coloneqq (\zeta_1(y), \zeta_2(y))$, the following estimates hold:
    \begin{alignat}{3}
        \label{eq:zeta_1_gradient}
        \abs{\nabla \zeta_1(y) - e_{n-1}}
        &\leq
        C \, \sigma(y'', \zeta_1(y), \zeta_2(y))
        &&\leq
        CL,
        \qquad &&\forall \, y \in \Omega, 
        \\
        \label{eq:zeta_2_gradient}
        \abs{\nabla \zeta_2(y) - e_n}
        &\leq
        C \, \sigma(y'', \zeta_1(y), \zeta_2(y))
        &&\leq
        CL, 
        \qquad &&\forall \, y \in \Omega,
        \\
        \label{eq:zeta_hessian}
        \abs{D^2 \zeta_1(y)},
        \abs{D^2 \zeta_2(y)}
        &\leq
        C \, \frac{\sigma(y'', \zeta_1(y), \zeta_2(y))}{\rho(\zeta(y))}
        &&\leq
        C\frac{L}{\rho(\zeta(y))},
        \qquad && \text{a.e. } \, y \in \Omega,
    \end{alignat}
    where we recall that $\rho$ and $\sigma$ are defined in \eqref{def:rho} and \eqref{def:sigma}, respectively.
    Moreover, if we define 
    \begin{equation} \label{def:sigma_tilde}
        \widetilde{\sigma}(x)
        \coloneqq
        \norm{\nabla\gamma}_{L^\infty(B''_{\rho(x)}(x''))}
        + \norm{\nabla\Gamma}_{L^\infty(B'_{C\rho(x)}(x'', x_{n-1} + \gamma(x'')))}
    \end{equation}
    then the following pointwise estimates also hold:
    \begin{alignat}{2}
        \label{eq:zeta_1_pointwise}
        \abs{\zeta_1(y) - (y_{n-1} - \gamma(y''))}
        &\leq
        C \, \sigma(y'', \zeta_1(y), \zeta_2(y)) \, \rho(\zeta(y))
        & \leq
        CL \, \rho(\zeta(y)), 
        \qquad & \forall \, y \in \Omega, 
        \\
        \label{eq:zeta_2_pointwise}
        \abs{\zeta_2(y) - (y_n - \Gamma(y'))}
        &\leq
        C \, \widetilde{\sigma}(y'', \zeta_1(y), \zeta_2(y)) \, \rho(\zeta(y))
        & \leq
        CL \, \rho(\zeta(y)),
        \qquad & \forall \, y \in \Omega, 
    \end{alignat}
\end{theorem}

The rest of the section is devoted to the proof of this result. We will proceed by a regularization argument which will undergo two different stages: first we will regularize the edge of the slit, and then the rest of the slit. Both regularizations will be made at the variable scale dictated by $\rho(x)$, which is deeply connected to the geometry of slits.


\subsection{The mollification scale, and mollifiers} \label{sec:mollifiers}

The regularity estimates proved in this section rely on a variable-scale mollification. Choosing the mollification radius to be $\rho$ (defined in \eqref{def:rho}) ensures that the smoothing becomes finer near the slit, while remaining quantitative because of the following estimates:

\begin{lemma} \label{lem:estimates_rho}
    If $\rho$ is defined as in \eqref{def:rho}, then $\rho \in C^1(\R^n \setminus \H)$, and actually it holds 
    \begin{equation} \label{eq:Drho}
        \abs{\nabla \rho(x)} = 1
        \qquad 
        \forall \, x \in \R^n \setminus \H.
    \end{equation}
    Moreover, it also holds 
    \begin{equation} \label{eq:D2rho}
        \abs{D^2 \rho(x)} \leq \frac{1}{\rho(x)},
        \qquad
        \forall \, x \in \R^n \setminus (\H \cup \{x_{n-1} = 0\}).
    \end{equation}
\end{lemma}
\begin{proof}
    This follows by direct explicit computation. Indeed, it holds, for $x_{n-1} > 0$:
    \begin{equation*}
        \partial_{x_{n-1}} \rho(x) = \frac{x_{n-1}}{\rho(x)},
        \qquad
        \partial_{x_n} \rho(x) = \frac{x_n}{\rho(x)},
    \end{equation*}
    and similarly for the second derivatives:
    \begin{equation*}
        \partial_{x_{n-1} x_{n-1}} \rho(x) = \frac{x_n^2}{\rho(x)^3},
        \quad
        \partial_{x_n x_n} \rho(x) = \frac{x_{n-1}^2}{\rho(x)^3},
        \quad
        \partial_{x_{n-1} x_n} \rho(x) = - \frac{x_{n-1}x_n}{\rho(x)^3}.
    \end{equation*}
    In turn, if $x_{n-1} < 0$, it holds $\partial_{x_{n-1}} \rho(x) = 0$, $\partial_{x_n} \rho(x) = \operatorname{sgn}(x_n)$, $D^2\rho(x) = 0$.

    The continuity of the first derivatives (in $\R^n \setminus \H$) is then obvious when $x_{n-1} \neq 0$, and one can check that it is also true whenever $x_{n-1} = 0$. It is also trivial to verify \eqref{eq:Drho} from these expressions. In turn, the second derivatives (concretely $\partial_{x_{n-1}x_{n-1}} \rho$) are not continuous across $\{x_{n-1} = 0\}$; but outside this region \eqref{eq:D2rho} readily follows.
\end{proof}

Now we define two mollifiers: one $n-2$ dimensional, and the other one $n-1$ dimensional. They will help us regularize the edge and the slit, respectively.
Concretely, choose 
\begin{equation} \label{def:eta''}
    \text{$\eta'' \in C^\infty_\mathrm{c}(B''_1)$ \, satisfying \, $\eta'' \geq 0$, \, $\int_{\R^{n-2}} \eta'' = 1$, \, and \, $\norm{\eta''}_{C^2(\R^{n-2})} \leq C(n)$.}
\end{equation}
Similarly, choose 
\begin{equation} \label{def:eta'}
    \text{$\eta' \in C^\infty_\mathrm{c}(B'_1)$ \, satisfying \, $\eta' \geq 0$, \, $\int_{\R^{n-1}} \eta' = 1$, \, and \, $\norm{\eta'}_{C^2(\R^{n-1})} \leq C(n)$.}
\end{equation}
Given a real number $t > 0$, we rescale these by
\begin{equation} \label{def:rescaled}
    \eta''_t(x'') \coloneqq t^{-(n-2)} \eta'' \left( \frac{x''}{t} \right),
    \quad
    \eta'_t(x') \coloneqq t^{-(n-1)} \eta' \left( \frac{x'}{t} \right),
    \quad
    x'' \in \R^{n-2},\ x' \in \R^{n-1}.
\end{equation}

Some classical and elementary properties that we will repeatedly use are the following.

\begin{lemma} \label{lem:mollifiers_basic}
    Let the mollifiers $\eta'', \eta'$ be defined in \eqref{def:eta''} and \eqref{def:eta'}, and their rescalings in \eqref{def:rescaled}. Then, the following are true, for any $t > 0$, with $C > 0$ only depending on $n$ (in particular, independent of $t$):
    \begin{enumerate}
        \item \label{it:moll_0} $\supp \eta''_t \subset B''_t$ and $\supp \eta'_t \subset B'_t$,
        \item \label{it:moll_1} $\int_{\R^{n-2}} \eta''_t = 1$ and $\int_{\R^{n-1}} \eta'_t = 1$,
        \item \label{it:moll_2} $\int_{\R^{n-2}} \partial_t \eta''_t = 0$ and $\int_{\R^{n-1}} \partial_t \eta'_t = 0$,
        \item \label{it:moll_3} $\int_{\R^{n-2}} \abs{\partial_t \eta''_t} \leq C/t$ and $\int_{\R^{n-1}} \abs{\partial_t \eta'_t} \leq C/t$,
        \item \label{it:moll_4} $\int_{\R^{n-2}} \abs{\partial_{x_j} \eta''_t} \leq C/t$ and $\int_{\R^{n-1}} \abs{\partial_{x_j} \eta'_t} \leq C/t$, for any index $j$,
        \item \label{it:moll_5} $\int_{\R^{n-2}} \abs{\partial_{tt} \eta''_t} \leq C/t^2$ and $\int_{\R^{n-1}} \abs{\partial_{tt} \eta'_t} \leq C/t^2$.
    \end{enumerate}
\end{lemma}
\begin{proof}
    Properties \eqref{it:moll_0} and \eqref{it:moll_1} follow immediately from the definition and a change of variables, and \eqref{it:moll_2} follows from differentiating \eqref{it:moll_1}. For \eqref{it:moll_3}, we just differentiate and change variables:
    \begin{align*}
        \int_{\R^{n-2}} \abs{\partial_t \eta''_t(x'')} dx''
        &\lesssim
        \int_{B''_t} t^{-(n-2)-1} \eta'' \left( \frac{x''}{t} \right) dx''
        + \int_{B''_t} t^{-(n-2)} \abs{\nabla \eta'' \left( \frac{x''}{t} \right) \cdot \frac{x''}{t^2}} dx''
        \\ &\leq
        \frac{1}{t} \left( \int_{B''_1} \eta''(y'') dy'' + \int_{B''_1} \abs{\nabla \eta''(y'')} \abs{y''} dy'' \right)
        \lesssim
        \frac{1}{t},
    \end{align*}
    and similarly for $\eta'_t$. One gets \eqref{it:moll_4} in the same way. Differentiating once more,
    \begin{align*}
        \int_{\R^{n-2}} &\abs{\partial_{tt} \eta''_t(x'')} dx''
        \lesssim
        \int_{B''_t} t^{-(n-2)-2} \eta'' \left( \frac{x''}{t} \right) dx''
        + \int_{B''_t} t^{-(n-2)-1} \abs{\nabla \eta'' \left( \frac{x''}{t} \right) \cdot \frac{x''}{t^2}} dx''
        \\ &\qquad \qquad \qquad \qquad  + 
        \int_{B''_t} t^{-(n-2)} \abs{\left(  \frac{x''}{t^2} \right)^\top D^2 \eta'' \left( \frac{x''}{t} \right) \frac{x''}{t^2}} dx''
        \\ &\lesssim
        \frac{1}{t^2} \left( \int_{B''_1} \eta''(y'') dy'' + \int_{B''_1} \abs{\nabla \eta''(y'')} \abs{y''} dy'' + \int_{B''_1} \abs{D^2 \eta''(y'')} \abs{y''}^2 dy'' \right)
        \lesssim
        \frac{1}{t^2},
    \end{align*}
    which establishes \eqref{it:moll_5} for $\eta''_t$, and one argues similarly for $\eta'_t$.
\end{proof}


\subsection{Regularization of the edge}

Let us first regularize, by mollification at the variable scale $\rho(x)$, the edge of the slit, namely 
\begin{equation*}
    \mathcal{E} \coloneqq \left\{(x'', \gamma(x''), \Gamma(x'', \gamma(x''))) : x'' \in \R^{n-2}\right\}. 
\end{equation*}

We consider the following shorthand notation for the size of the oscillations of $\gamma$ at scale $\rho$:
\begin{equation} \label{def:sigma_gamma}
    \sigma_\gamma(x) \coloneqq \norm{\nabla\gamma}_{L^\infty(B''_{\rho(x)}(x''))}.
\end{equation}

\begin{lemma} \label{lem:tgamma}
    Let $\gamma : \R^{n-2} \rightarrow \R$ be a Lipschitz function. With $\rho, \eta''$ as in \cref{subsec:notation,sec:mollifiers}, define $\tgamma : \R^n \setminus \H \rightarrow \R$ as the regularization of $\gamma$ given by the following convolution-like expression adapted to the scale $\rho(x)$:
    \begin{equation*}
        \tgamma(x)
        \coloneqq
        \eta''_{\rho(x)} * \gamma (x'')
        \coloneqq
        \int_{\R^{n-2}} \eta''_{\rho(x)} (w'') \, \gamma(x'' - w'') \, dw''.
    \end{equation*}
    Then $\tgamma$ fulfills the following estimates: 
    \begin{alignat}{2}
        \label{eq:tgamma_0}
        \abs{\tgamma(x) - \gamma(x'')} &\leq C \sigma_\gamma(x) \rho(x), 
        \qquad &&\forall \, x \in \R^n \setminus \H,\\
        \label{eq:tgamma_1}
        \abs{\nabla \tgamma(x)} &\leq C \sigma_\gamma(x), 
        &&\forall \, x \in \R^n \setminus \H,\\
        \label{eq:tgamma_2}
        \abs{D^2 \tgamma(x)} &\leq C \frac{\sigma_\gamma(x)}{\rho(x)},
        && \forall \, x \in \R^n \setminus (\H \cup \{x_{n-1} = 0\}),
    \end{alignat}
    with constants only depending on $n$.
\end{lemma}
\begin{proof}
    To obtain \eqref{eq:tgamma_0}, use \cref{lem:mollifiers_basic}.\eqref{it:moll_1} to get cancellations inside the integrals:
    \begin{align*}
        \abs{\tgamma(x) - \gamma(x'')}
        & =
        \abs{ \int_{B''_{\rho(x)}} \eta''_{\rho(x)} (w'') \gamma(x'' - w'') dw'' - \gamma(x'')}
        \\ & =
        \abs{ \int_{B''_{\rho(x)}} \eta''_{\rho(x)} (w'') \left( \gamma(x'' - w'') - \gamma(x'') \right) dw''}
        \\ & \leq 
        \int_{B''_{\rho(x)}} \eta''_{\rho(x)} (w'') \norm{\nabla \gamma}_{L^\infty(B''_{\rho(x)}(x''))} \abs{w''} dw''
        \\ & \leq
        \sigma_\gamma (x) \rho(x) \int \eta''_{\rho(x)}(w'') \, dw''
        \\ & =
        \sigma_\gamma (x) \rho(x).
    \end{align*}

    To compute the derivatives of $\tgamma$, let us treat it as a composition $\tgamma = f \circ (x \mapsto (x'', \rho(x_{n-1}, x_n)))$, i.e. $f(x'', \rho) \coloneqq \eta''_\rho * \gamma(x'')$ for $x'' \in \R^{n-2}, \rho > 0$ (that is, for the moment, we treat $\rho$ as a number, and not a function). Then, we can estimate the derivatives of $f$:
    \begin{itemize}
        \item First, for $j = 1, \ldots, n-2$, it holds, since $\rho$ does not depend on $x_j$:
        \begin{equation*}
            \abs{\partial_{x_j} f(x'', \rho)}
            =
            \abs{ \int_{B''_\rho} \eta''_\rho (w'') \, \partial_{x_j} \gamma(x'' - w'') \, dw'' }
            \leq
            \norm{\nabla \gamma}_{L^\infty(B''_\rho(x''))}.
        \end{equation*} 

        \item In turn, for the derivative with respect to $\rho$, the cancellation from \cref{lem:mollifiers_basic}.\eqref{it:moll_2}, and later the bound from \eqref{it:moll_3}, yield
        \begin{align*}
            \abs{\partial_\rho f(x'', \rho)}
            &=
            \abs{ \int_{B''_\rho} \partial_\rho \eta''_\rho (w'') \, \gamma(x'' - w'') \, dw'' }
            \\ &=
            \abs{ \int_{B''_\rho} \partial_\rho \eta''_\rho (w'') \left( \gamma(x'' - w'') - \gamma(x'') \right) dw'' }
            \\ &\leq
            \norm{\nabla \gamma}_{L^\infty(B''_\rho(x''))}\rho \int \abs{ \partial_\rho \eta''_\rho}
            \lesssim
            \norm{\nabla \gamma}_{L^\infty(B''_\rho(x''))}.
        \end{align*}

        \item We can proceed similarly for the second derivatives of $f$. Indeed, if $j, k \in \{1, \ldots, n-2\}$, then an integration by parts (using that $\eta''_\rho$ vanishes over $\partial B_\rho''$) and \cref{lem:mollifiers_basic}.\eqref{it:moll_4} implies
        \begin{align*}
            \abs{\partial_{x_j x_k} f(x'', \rho)}
            &=
            \abs{ \int_{B''_\rho} \partial_{x_j} \eta''_{\rho} (w'')\,  \partial_{x_k} \gamma(x'' - w'') \, dw'' }
            \\ &\leq
            \norm{\nabla \gamma}_{L^\infty(B''_\rho(x''))} \int \abs{ \partial_{x_j} \eta''_\rho}
            \lesssim
            \frac{\norm{\nabla \gamma}_{L^\infty(B''_\rho(x''))}}{\rho}.
        \end{align*}

        \item In turn, if  $j \in \{1, \ldots, n-2\}$, by \cref{lem:mollifiers_basic}.\eqref{it:moll_3}, we verify
        \begin{align*}
            \abs{\partial_{x_j \rho} f(x'', \rho)}
            &=
            \abs{ \int_{B''_\rho} \partial_\rho \eta''_{\rho} (w'') \, \partial_{x_j} \gamma(x'' - w'') \, dw'' }
            \\ &\leq
            \norm{\nabla \gamma}_{L^\infty(B''_\rho(x''))} \int \abs{ \partial_\rho \eta''_\rho}
            \lesssim
            \frac{\norm{\nabla \gamma}_{L^\infty(B''_\rho(x''))}}{\rho}.
        \end{align*}

        \item Lastly, by using the cancellation from \cref{lem:mollifiers_basic}.\eqref{it:moll_2} (which also holds trivially for second derivatives of the mollifier), and then \eqref{it:moll_5}, we also estimate
        \begin{align*}
            \abs{\partial_{\rho \rho} f(x'', \rho)}
            &=
            \abs{ \int_{B''_\rho} \partial_{\rho \rho} \eta''_\rho (w'') \left( \gamma(x'' - w'') - \gamma(x'') \right) dw'' }
            \\ &\leq
            \norm{\nabla \gamma}_{L^\infty(B''_\rho(x''))} \rho \int \abs{ \partial_{\rho \rho} \eta''_\rho}
            \lesssim
            \frac{\norm{\nabla \gamma}_{L^\infty(B''_\rho(x''))}}{\rho}.
        \end{align*}
    \end{itemize}

    Now that we have a good control over the derivatives of $f$, let us obtain estimates for $\tgamma$ by the chain rule. We will employ repeatedly the estimates for $f$ just obtained, and the estimates for $\rho$ from \cref{lem:estimates_rho}. We also recall the definition of $\sigma_\gamma$ in \eqref{def:sigma_gamma}.
    \begin{itemize}
        \item First, if $j \in \{ 1, \ldots, n-2 \}$, we have
        \begin{equation*}
            \abs{\partial_{x_j} \tgamma(x)}
            =
            \abs{\partial_{x_j} f(x'', \rho(x_{n-1}, x_n))}
            \leq
            \sigma_\gamma(x).
        \end{equation*}

        \item In turn, if $j \in \{n-1, n\}$, then
        \begin{equation*}
            \abs{\partial_{x_j} \tgamma(x)}
            =
            \abs{\partial_\rho f(x'', \rho(x_{n-1}, x_n)) \, \partial_{x_j} \rho(x)}
            \lesssim
            \sigma_\gamma(x).
        \end{equation*}

        \item Now, for the second derivatives, if $j, k \in \{1, \ldots, n-2\}$, we have
        \begin{equation*}
            \abs{\partial_{x_j x_k} \tgamma(x)}
            =
            \abs{\partial_{x_j x_k} f(x'', \rho(x_{n-1}, x_n))}
            \lesssim
            \frac{\sigma_\gamma(x)}{\rho(x)}.
        \end{equation*}

        \item In turn, if $j \in \{1, \ldots, n-2\}$, $k \in \{n-1, n\}$, we have instead
        \begin{equation*}
            \abs{\partial_{x_j x_k} \tgamma(x)}
            =
            \abs{\partial_{x_j \rho} f(x'', \rho(x_{n-1}, x_n)) \, \partial_{x_k} \rho(x)}
            \lesssim
            \frac{\sigma_\gamma(x)}{\rho(x)}.
        \end{equation*}

        \item Lastly, if $j, k \in \{n-1, n\}$, it holds
        \begin{equation*}
            \abs{\partial_{x_j x_k} \tgamma(x)}
            =
            \abs{\partial_{\rho \rho} f(x'', \rho(x_{n-1}, x_n)) \, \partial_{x_j} \rho(x) \, \partial_{x_k} \rho(x)
            + \partial_\rho f(x'', \rho(x_{n-1}, x_n)) \, \partial_{x_j x_k} \rho(x)}
            \lesssim
            \frac{\sigma_\gamma(x)}{\rho(x)}.
        \end{equation*}
    \end{itemize}

    We note that the estimates for the first derivatives of $\tgamma$ hold pointwise in $\R^n \setminus \H$ because so do the estimates for $\nabla \rho$ (see \eqref{eq:Drho}); but those for second derivatives of $\tgamma$ involve estimates for second derivatives of $\rho$, which are only true outside $\{x_{n-1}=0\}$ (see \eqref{eq:D2rho}), so \eqref{eq:tgamma_2} only holds in the same region, accordingly.
\end{proof}


\subsection{Regularization of the slit}

After having regularized the edge, let us now regularize the whole slit by mollifying at the variable scale $\rho(x)$. As before, we set a similar notation to control the size of the oscillations of $\Gamma$ at the appropriate scale:
\begin{equation} \label{def:sigma_Gamma}
    \sigma_\Gamma(x) \coloneqq \norm{\nabla\Gamma}_{L^\infty(B'_{2\rho(x)}(x'',x_{n-1} + \gamma(x'')))}.
\end{equation}

\begin{lemma} \label{lem:tGamma}
    Let $\Gamma : \R^{n-1} \rightarrow \R$ be a Lipschitz function. Let also $\gamma : \R^{n-2} \rightarrow \R$ be a Lipschitz function with constant $L > 0$. With $\rho, \eta'', \eta'$ as in \cref{subsec:notation,sec:mollifiers}, and $\tgamma$ associated to $\gamma$ as in \cref{lem:tgamma},
    define $\tGamma : \R^n \setminus \H \rightarrow \R$ by
    \begin{align*}
        \tGamma(x)
        &\coloneqq
        \eta'_{\rho(x)} * \Gamma (x'', x_{n-1} + \tgamma(x))
        \\ & \coloneqq
        \int_\R\int_{\R^{n-2}} \eta'_{\rho(x)} (w'',w_{n-1}) \, \Gamma(x'' - w'', x_{n-1} + \tgamma(x) - w_{n-1})  \, dw'' dw_{n-1}.
    \end{align*}
    If $L \leq 1$ is small enough, then the following estimates hold:
    \begin{align}
        \label{eq:tGamma_1}
        |\nabla \tGamma(x)| &\leq C \sigma_\Gamma(x), 
        \qquad \forall \, x \in \R^n \setminus \H,\\
        \label{eq:tGamma_2}
        |D^2 \tGamma(x)| &\leq C \frac{\sigma_\Gamma(x)}{\rho(x)},
        \qquad \forall \, x \in \R^n \setminus (\H \cup \{x_{n-1} = 0\}),
    \end{align}
    with constants only depending on $n$.
\end{lemma}
\begin{proof}
    As in the proof of \cref{lem:tgamma}, we use the factorization $\tGamma = F \circ (x \mapsto (x'', x_{n-1} + \tgamma(x), \rho(x_{n-1}, x_n))$, where $F(x', \rho) \coloneqq \eta'_\rho * \Gamma (x')$ is defined for $x' \in \R^{n-1}$ and $\rho > 0$. Taking into account \cref{lem:mollifiers_basic}, the exact same computations as in the proof of \cref{lem:tgamma} show
    \begin{gather*}
        \abs{\partial_{x_j} F(x', \rho)} \leq \norm{\nabla\Gamma}_{L^\infty(B'_\rho(x'))} \quad \text{ for } j \in \{1, \ldots, n-1\},
        \qquad\quad
        \abs{\partial_\rho F(x', \rho)} \lesssim \norm{\nabla\Gamma}_{L^\infty(B'_\rho(x'))}, \\
        \abs{\partial_{x_j x_k} F(x', \rho)}, \abs{\partial_{x_j \rho} F(x', \rho)}, \abs{\partial_{\rho \rho} F(x', \rho)} \lesssim \frac{\norm{\nabla\Gamma}_{L^\infty(B'_\rho(x'))}}{\rho} \quad \text{ for all } j, k \in \{ 1, \ldots, n-1\}.
    \end{gather*}
    Therefore, by the chain rule and the fact that $\sigma_\gamma(x) \leq L \leq 1$, we obtain the estimates for $\tGamma$, recalling the similar estimates for $\tgamma$ obtained in \cref{lem:tgamma}. The computations are definitely routine and tedious, but we include them for completeness of the argument.
    \begin{itemize}
        \item First, if $j \in \{1, \ldots, n-2\}$, we have
        \begin{align*}
            \abs{\partial_{x_j} \tGamma(x)}
            &=
            \abs{\partial_{x_j} F(x'',x_{n-1} + \tgamma(x), \rho(x))
            + \partial_{x_{n-1}} F(x'',x_{n-1} + \tgamma(x), \rho(x)) \, \partial_{x_j}\tgamma(x)}
            \\ &\lesssim
            \sigma_\Gamma(x) + \sigma_\Gamma(x) \sigma_\gamma(x)
            \lesssim
            \sigma_\Gamma(x).
        \end{align*}
        Here and in the next estimates, we use the following simple observation:
        \begin{align*}
            \norm{\nabla \Gamma}_{L^\infty(B'_{\rho(x)}(x'', x_{n-1} + \tgamma(x)))}
            & \leq 
            \norm{\nabla \Gamma}_{L^\infty(B'_{\rho(x) + \abs{\tgamma(x) - \gamma(x'')}}(x'', x_{n-1} + \gamma(x'')))}
            \\ & \leq 
            \norm{\nabla \Gamma}_{L^\infty(B'_{2\rho(x)}(x'', x_{n-1} + \gamma(x'')))}
            =
            \sigma_\Gamma(x),
        \end{align*}
        where in the last bound we have used \eqref{eq:tgamma_0} choosing $L$ small enough.

        \item Next, if $j = n-1$,
        \begin{align*}
            \abs{\partial_{x_{n-1}} \tGamma(x)}
            &=
            \left|\partial_{x_{n-1}} F(x'',x_{n-1} + \tgamma(x), \rho(x)) \, (1 + \partial_{x_{n-1}}\tgamma(x)) \right.
            \\ &\quad+
            \left.\partial_\rho F(x'', x_{n-1} + \tgamma(x), \rho(x)) \, \partial_{x_{n-1}} \rho(x)\right|
            \\ &\lesssim
            \sigma_\Gamma(x) (1 + \sigma_\gamma(x)) + \sigma_\Gamma(x)
            \lesssim
            \sigma_\Gamma(x).
        \end{align*}

        \item And the last first order derivative:
        \begin{align*}
            \abs{\partial_{x_n} \tGamma(x)}
            &=
            \abs{
            \partial_{x_{n-1}} F(x'',x_{n-1} + \tgamma(x), \rho(x)) \, \partial_{x_n}\tgamma(x)
                + \partial_\rho F(x'', x_{n-1} + \tgamma(x), \rho(x)) \, \partial_{x_n} \rho(x)}
            \\ &\lesssim
            \sigma_\Gamma(x) \sigma_\gamma(x) + \sigma_\Gamma(x)
            \lesssim
            \sigma_\Gamma(x).
        \end{align*}
    \end{itemize}
    All these estimates hold for every $x \in \R^n \setminus \H$ because \eqref{eq:Drho} holds pointwise.

    Let us now estimate the second derivatives. To ease the notation we leave implicit the points where the functions are evaluated.
    \begin{itemize}
        \item First, if $j, k \in \{1, \ldots, n-2\}$,
        \begin{align*}
            \abs{\partial_{x_j x_k} \tGamma(x)}
            &=
            \left|
            \partial_{x_j x_k} F
            + \partial_{x_k x_{n-1}} F \, \partial_{x_j} \tgamma
            + \left(\partial_{x_{n-1} x_j} F + \partial_{x_{n-1} x_{n-1}} F \partial_{x_j} \tgamma\right) \partial_{x_k} \tgamma
            \right.
            \\ & \quad \left.
            + \, \partial_{x_{n-1}} F \, \partial_{x_j x_k} \tgamma
            \right|
            \\ &\lesssim
            \frac{\sigma_\Gamma(x)}{\rho(x)}
            + \frac{\sigma_\Gamma(x)}{\rho(x)} \sigma_\gamma(x)
            + \left( \frac{\sigma_\Gamma(x)}{\rho(x)} + \frac{\sigma_\Gamma(x)}{\rho(x)} \sigma_\gamma(x) \right) \sigma_\gamma(x)
            + \sigma_\Gamma(x) \frac{\sigma_\gamma(x)}{\rho(x)}
            \lesssim
            \frac{\sigma_\Gamma(x)}{\rho(x)}.
        \end{align*}

        \item Next, if $j \in \{1, \ldots, n-2\}$,
        \begin{align*}
            \abs{\partial_{x_j x_{n-1}} \tGamma(x)}
            &=
            |
            \partial_{x_j x_{n-1}} F \, (1 + \partial_{x_{n-1}} \tgamma) + \partial_{x_j \rho} F \, \partial_{x_{n-1}} \rho
            \\ &\quad + \left(\partial_{x_{n-1} x_{n-1}} F \, (1+\partial_{x_{n-1}} \tgamma) + \partial_{x_{n-1} \rho} F \, \partial_{x_{n-1}} \rho \right) \partial_{x_j} \tgamma
            + \partial_{x_{n-1}} F \, \partial_{x_j x_{n-1}} \tgamma
            |
            \\ &\lesssim
            \frac{\sigma_\Gamma}{\rho}  (1 + \sigma_\gamma) + \frac{\sigma_\Gamma}{\rho}
            + \left(\frac{\sigma_\Gamma}{\rho} (1 + \sigma_\gamma) + \frac{\sigma_\Gamma}{\rho}\right) \sigma_\gamma
            + \sigma_\Gamma \frac{\sigma_\gamma}{\rho}
            \lesssim
            \frac{\sigma_\Gamma(x)}{\rho(x)}.
        \end{align*}

        \item Still, if $j \in \{1, \ldots, n-2\}$,
        \begin{align*}
            \abs{\partial_{x_j x_n} \tGamma(x)}
            &=
            \left|\partial_{x_j x_{n-1}} F \, \partial_{x_n} \tgamma 
            + \partial_{x_j \rho} F \, \partial_{x_n} \rho \right.
            \\ &\quad +
            \left.\left(\partial_{x_{n-1} x_{n-1}}F \, \partial_{x_n} \tgamma + \partial_{x_{n-1} \rho} F \,\partial_{x_n} \rho\right) \partial_{x_j} \tgamma + \partial_{x_{n-1}} F \, \partial_{x_j x_n} \tgamma\right|
            \\ &\lesssim
            \frac{\sigma_\Gamma}{\rho} \sigma_\gamma + \frac{\sigma_\Gamma}{\rho}
            + \left( \frac{\sigma_\Gamma}{\rho} \sigma_\gamma + \frac{\sigma_\Gamma}{\rho} \right) \sigma_\gamma
            + \sigma_\Gamma \frac{\sigma_\gamma}{\rho}
            \lesssim
            \frac{\sigma_\Gamma(x)}{\rho(x)}.
        \end{align*}

        \item Next, we have
        \begin{align*}
            \abs{\partial_{x_{n-1} x_{n-1}} \tGamma(x)}
            &=
            \left|
            \left( \partial_{x_{n-1} x_{n-1}} F \, (1+\partial_{x_{n-1}} \tgamma) + \partial_{x_{n-1} \rho} F \, \partial_{x_{n-1}} \rho \right) (1+\partial_{x_{n-1}} \tgamma)\right.
            \\ &\quad +
            \partial_{x_{n-1}} F \,  \partial_{x_{n-1} x_{n-1}} \tgamma
            \\ &\quad +
            \left.\left( \partial_{\rho x_{n-1}} F \, (1+\partial_{x_{n-1}} \tgamma) + \partial_{\rho \rho} F \, \partial_{x_{n-1}} \rho \right) \partial_{x_{n-1}} \rho
            + \partial_\rho F \, \partial_{x_{n-1} x_{n-1}} \rho
            \right|
            \\ &\lesssim
            \left( \frac{\sigma_\Gamma}{\rho} (1 + \sigma_\gamma) + \frac{\sigma_\Gamma}{\rho} \right) (1 + \sigma_\gamma)
            + \sigma_\Gamma \frac{\sigma_\gamma}{\rho}
            + \left( \frac{\sigma_\Gamma}{\rho} (1 + \sigma_\gamma) + \frac{\sigma_\Gamma}{\rho} \right)
            + \sigma_\Gamma \frac{1}{\rho}
            \\ & \lesssim
            \frac{\sigma_\Gamma(x)}{\rho(x)}.
        \end{align*}

        \item Similarly,
        \begin{align*}
            \abs{\partial_{x_{n-1} x_n} \tGamma(x)}
            &=
            \left|
            \left( \partial_{x_{n-1} x_{n-1}} F \, \partial_{x_n} \tgamma + \partial_{x_{n-1} \rho} F \, \partial_{x_n} \rho \right) (1 + \partial_{x_{n-1}} \tgamma)
            + \partial_{x_{n-1}} F \, \partial_{x_{n-1} x_n} \tgamma\right.
            \\ &\quad +
            \left.\left( \partial_{\rho x_{n-1}} F \, \partial_{x_n} \tgamma + \partial_{\rho \rho} F \, \partial_{x_n} \rho \right) \partial_{x_{n-1}} \rho
            + \partial_\rho F \, \partial_{x_{n-1} x_n} \rho\right|
            \\ &\lesssim
            \left( \frac{\sigma_\Gamma}{\rho} \sigma_\gamma + \frac{\sigma_\Gamma}{\rho} \right) (1 + \sigma_\gamma)
            + \sigma_\Gamma \frac{\sigma_\gamma}{\rho}
            + \left( \frac{\sigma_\Gamma}{\rho} \sigma_\gamma + \frac{\sigma_\Gamma}{\rho} \right)
            + \sigma_\Gamma \frac{1}{\rho}
            \lesssim
            \frac{\sigma_\Gamma(x)}{\rho(x)}.
        \end{align*}

        \item And we finish with
        \begin{align*}
            \abs{\partial_{x_n x_n} \tGamma(x)}
            &=
            \left|\left( \partial_{x_{n-1} x_{n-1}} F \, \partial_{x_n} \tgamma + \partial_{x_{n-1} \rho} F \, \partial_{x_n} \rho \right) \partial_{x_n} \tgamma
            + \partial_{x_{n-1}} F \, \partial_{x_n x_n} \tgamma\right.
            \\ &\quad +
            \left.\left( \partial_{\rho x_{n-1}} F \, \partial_{x_n} \tgamma + \partial_{\rho \rho} F  \,\partial_{x_n} \rho \right) \partial_{x_n} \rho
            + \partial_\rho F \, \partial_{x_n x_n} \rho
            \right|
            \\ &\lesssim
            \left( \frac{\sigma_\Gamma}{\rho} \sigma_\gamma + \frac{\sigma_\Gamma}{\rho} \right) \sigma_\gamma
            + \sigma_\Gamma \frac{\sigma_\gamma}{\rho}
            + \left( \frac{\sigma_\Gamma}{\rho} \sigma_\gamma + \frac{\sigma_\Gamma}{\rho} \right)
            + \sigma_\Gamma \frac{1}{\rho}
            \lesssim
            \frac{\sigma_\Gamma(x)}{\rho(x)}.
        \end{align*}
    \end{itemize}
    All these estimates for second derivatives hold only outside $\{x_{n-1}=0\}$ because of \eqref{eq:D2rho}.
\end{proof}


\subsection{Proof of \cref{th:complex_coordinates}} \label{subsec:proof_coordinates}

In this subsection, we prove \cref{th:complex_coordinates}. We proceed in a series of steps. 

\textbf{Step 1: extension of $\tgamma, \tGamma$, definition of $P$.}
Let us note that, since our mollifying scale $\rho$ vanishes over the model slit $\H$, it is reasonable to extend the definitions of $\tgamma$ and $\tGamma$ to $\H$ (and hence $\R^n$) simply by pointwise values:
\begin{equation} \label{def:tgamma_tGamma}
    \tgamma(x)
    \coloneqq
    \begin{cases}
        \gamma(x'') \quad &\text{if } x \in \H, \\
        \eta''_{\rho(x)} * \gamma (x'') &\text{if } x \notin \H,
    \end{cases}
    \qquad
    \tGamma(x)
    \coloneqq
    \begin{cases}
        \Gamma(x'', x_{n-1} + \gamma(x'')) \quad &\text{if } x \in \H, \\
        \eta'_{\rho(x)} * \Gamma(x'', x_{n-1} + \tgamma(x)) &\text{if } x \notin \H,
    \end{cases}
\end{equation}
Note that these are indeed natural extensions of $\tgamma, \tGamma$ from \cref{lem:tgamma} and \cref{lem:tGamma} to the whole $\R^n$, so we will abuse notation by reusing their names.

\begin{lemma} \label{lem:continuous}
    Let $\gamma : \R^{n-2} \rightarrow \R, \Gamma : \R^{n-1} \rightarrow \R$ be Lipschitz functions with constant $L > 0$. Then, the functions $\tgamma, \tGamma$ defined in \eqref{def:tgamma_tGamma} (with $\rho, \eta'', \eta'$ as in \cref{subsec:notation,sec:mollifiers}) are continuous in $\R^n$.
\end{lemma}
\begin{proof}
    We want to show that $\tgamma(x+h) \to \tgamma(x)$ as $h \to 0$. The most interesting case is when $x \in \H$ and $x+h \notin \H$ (the other cases and the adaptation for $\tGamma$ are similar). In such case, using that the rescaled mollifiers integrate to 1 (see \cref{lem:mollifiers_basic}), we have
    \begin{align*}
        \abs{\tgamma(x+h) - \tgamma(x)}
        & =
        \abs{\eta''_{\rho(x+h)} * \gamma (x'' + h'') - \gamma(x'')}
        \\ & \leq 
        \int_{B''_{\rho(x+h)}} \eta''_{\rho(x+h)}(w'') \abs{\gamma(x'' + h'' - w'') - \gamma(x'')} dw''
        \\ & \leq 
        \int_{B''_{\rho(x+h)}} \eta''_{\rho(x+h)}(w'') L \left( \abs{h''} + \rho(x+h) \right) dw''
        \\ & =
        L \left( \abs{h''} + \rho(x+h) \right)
        \underset{h \to 0}{\longrightarrow} 
        0,
    \end{align*}
    where we have used the fact that $\rho$ is continuous in $\R^n$ and $\rho(x) = 0$ because $x \in \H$.
\end{proof}

With this, we are ready to define the appropriate parametrization of our slit domain. Indeed, define $P: \R^n \rightarrow \R^n$ by
\begin{equation} \label{def:P}
    P(x)
    =
    P(x'', x_{n-1}, x_n)
    \coloneqq
    (x'', x_{n-1} + \tgamma(x), x_n + \tGamma(x)).
\end{equation}
It is simple to check that $P$ maps the model slit $\H$ to our slit $\S$:

\begin{lemma} \label{lem:parametrization_slit}
    Let $\gamma : \R^{n-2} \rightarrow \R, \Gamma : \R^{n-1} \rightarrow \R$ be Lipschitz functions. Let the function $P$ be defined as in \eqref{def:P} in terms of $\tgamma, \tGamma$ from \eqref{def:tgamma_tGamma}. Then, it holds $P(\H) = \S$.
\end{lemma}
\begin{proof}
    This follows from the definitions. Indeed, if $x \in \H$, then, simply by \eqref{def:P} and \eqref{def:tgamma_tGamma},
    \begin{equation*}
        (y'', y_{n-1}, y_n) \coloneqq P(x) = (x'', x_{n-1} + \gamma(x''), x_n + \Gamma(x'', x_{n-1} + \gamma(x''))).
    \end{equation*}
    Since $x \in \H$, it holds $x_{n-1} \leq 0$, so $y_{n-1} = x_{n-1} + \gamma(x'') \leq \gamma(x'') = \gamma(y'')$. Moreover, also because $x \in \H$, it holds $x_n = 0$, so $y_n = \Gamma(x'', x_{n-1} + \gamma(x'')) = \Gamma(y')$, confirming $P(x) \in \S$.

    On the other hand, if $s \in \S$, $z \coloneqq (s'', s_{n-1} - \gamma(s''), 0) \in \H$ because $s_{n-1} \leq \gamma(s'')$, and $P(z) = (s'', s_{n-1}, \Gamma(s')) = s$.
\end{proof}

\textbf{Step 2: invertibility of $P$.}
Note that $P$ can be written as
\begin{equation*}
    P(x) = x + (0, \tgamma(x), \tGamma(x)), \qquad x \in \R^n,
\end{equation*}
so that if we denote by $I$ the $n \times n$ identity matrix, it holds
\begin{equation} \label{eq:DP_I-E}
    DP(x) = I +
    \begin{bmatrix}
    0 \\
    \nabla \tgamma(x)\\
    \nabla \tGamma(x)
    \end{bmatrix}
    \eqqcolon
    I - E(x).
\end{equation}
The estimates from \cref{lem:tgamma,lem:tGamma} imply that, for every $x \notin \H$, it holds
\begin{equation} \label{eq:E_size}
    |E(x)|
    \leq
    C (\sigma_\gamma(x) + \sigma_\Gamma(x))
    \leq
    CL.
\end{equation}

\begin{lemma} \label{lem:P_bijective}
    Let $\gamma : \R^{n-2} \rightarrow \R, \Gamma : \R^{n-1} \rightarrow \R$ be Lipschitz functions with constant $L > 0$. Let the function $P$ be defined as in \eqref{def:P} in terms of $\tgamma, \tGamma$ from \eqref{def:tgamma_tGamma}.
    If $L$ is small enough, then $P: \R^n \setminus \H \rightarrow \R^n \setminus \S$ is bijective.
\end{lemma}
\begin{proof}
    Fix $x, y \in \R^n \setminus \H$. By definition of $P$, it holds
    \begin{equation*}
        \abs{P(x) - P(y)}
        \geq
        \abs{x-y} - \abs{\tgamma(x)-\tgamma(y)} - |\tGamma(x)-\tGamma(y)|.
    \end{equation*}
    Let us estimate the size of the last two terms. Assume $x_n \geq 0$ (the negative case is analogous). Then, if $y_n \geq 0$, we may use the mean value theorem (because the segment joining $x$ and $y$ lies in $\R^n \setminus \H$) along with \eqref{eq:E_size} to obtain
    \begin{equation} \label{eq:tgamma_lip}
        \abs{\tgamma(x)-\tgamma(y)}
        \leq
        CL \abs{x-y}.
    \end{equation}
    On the other hand, if $y_n < 0$, we notice that $\tgamma(y) = \tgamma(y'', y_{n-1}, y_n) = \tgamma(y'', y_{n-1}, -y_n)$ because $\rho(y_{n-1}, y_n) = \rho(y_{n-1}, -y_n)$, which means that we can use the above estimate to arrive at
    \begin{equation*}
        \abs{\tgamma(x)-\tgamma(y)}
        =
        \abs{\tgamma(x)-\tgamma(y'', y_{n-1}, -y_n)}
        \leq
        CL \abs{x-(y'', y_{n-1}, -y_n)}
        \leq
        CL \abs{x-y},
    \end{equation*}
    where the last inequality is a simple consequence of the fact that $x_n \geq 0$ and $y_n < 0$. This means that \eqref{eq:tgamma_lip} actually holds for any $x, y \in \R^n \setminus \H$. The exact same argument yields $|\tGamma(x)-\tGamma(y)| \leq CL \abs{x-y}$ for any $x, y \in \R^n \setminus \H$. Inserting these estimates back into the initial bound for $P$, and choosing $L$ small enough, we obtain
    \begin{equation} \label{eq:P_coercive}
        \abs{P(x)-P(y)} \geq \frac{\abs{x-y}}{2}.
    \end{equation}

    Since $P$ is continuous in $\R^n$ (see \cref{lem:continuous}), \eqref{eq:P_coercive} holds for all $x, y \in \R^n$. In particular, $P$ maps $\R^n \setminus \H$ into $\R^n \setminus \S$ by \cref{lem:parametrization_slit}.

    With this estimate in hand, we show that $P(\R^n \setminus \H)$ is relatively closed in $\R^n \setminus \S$. Indeed, pick $y \in \R^n \setminus \S$ and $x_j \in \R^n \setminus \H$ such that $P(x_j) \to y$. Then $P(x_j)$ is Cauchy, so by \eqref{eq:P_coercive} $x_j$ is Cauchy, too. Hence, there exists $x_0 \in \R^n$ such that $x_j \to x_0$, and by continuity of $P$ (from \cref{lem:continuous}), $y = P(x_0)$. And in fact, since $y \in \R^n \setminus \S$, \cref{lem:parametrization_slit} implies that $x_0 \in \R^n \setminus \H$, whence $y \in P(\R^n \setminus \H)$, which ends the proof of the relative closedness.

    Moreover, note that $P(\R^n \setminus \H)$ is relatively open in $\R^n \setminus \S$. Indeed, if $L$ is chosen small enough, the estimate \eqref{eq:E_size} implies that $DP(x)$ is invertible at any $x \in \R^n \setminus \H$ (by the Inverse Function Theorem, or a Neumann series), so actually $P$ is a local diffeomorphism around such $x$, which clearly proves the relative openness.

    Since $\R^n \setminus \S$ is connected (by definition of our slit domains), the previous paragraphs imply that $P(\R^n \setminus \H) = \R^n \setminus \S$, and the injectivity follows from \eqref{eq:P_coercive}.
\end{proof}


\begin{lemma} \label{lem:derivatives_P-1}
    Let $\gamma : \R^{n-2} \rightarrow \R, \Gamma : \R^{n-1} \rightarrow \R$ be Lipschitz functions with constant $L > 0$. Let the function $P$ be defined as in \eqref{def:P} in terms of $\tgamma, \tGamma$ from \eqref{def:tgamma_tGamma}.
    If $L$ is small enough, then $P: \R^n \setminus \H \longrightarrow \R^n \setminus \S$ is invertible, and moreover the following estimates hold: 
    \begin{alignat*}{2}
        \abs{DP^{-1}(y)} &\lesssim 1,
        \qquad && \forall \, y \in \R^n \setminus \S,
        \\ 
        \abs{D^2P^{-1}(y)} & \lesssim \frac{\sigma(P^{-1}(y))}{\rho(P^{-1}(y))},
        \qquad
        && a.e. \; y \in \R^n \setminus \S,
    \end{alignat*}
    where we recall that $\sigma(x) = \sigma_\gamma(x) + \sigma_\Gamma(x)$ (see \eqref{def:sigma}, \eqref{def:sigma_gamma}, and \eqref{def:sigma_Gamma}). In particular, it holds $P^{-1} \in C^{1, 1}_{\mathrm{loc}}(\R^n \setminus \S)$. 
\end{lemma}
\begin{proof}
    \textbf{Invertibility and first derivatives.}
    Choosing $L$ small enough, \eqref{eq:E_size} means that $|E(x)| < 1$ for all $x \in \R^n \setminus \H$, which implies that $DP(x)$ is invertible (by the Inverse Function Theorem, or a Neumann series, see \eqref{eq:DP_I-E}). In fact, having \cref{lem:P_bijective} and the Inverse Function Theorem in mind, it turns out that $P$ is actually globally invertible, i.e. there exists a function $P^{-1}: \R^n \setminus \S \rightarrow \R^n \setminus \H$ which is the pointwise inverse of $P$, and moreover satisfies
    \begin{equation} \label{eq:P-1}
        DP^{-1}(y)
        =
        (DP(P^{-1}(y)))^{-1},
        \qquad
        \forall \, y \in P(\R^n \setminus \H).
    \end{equation}

    To get a quantitative estimate on the size of $DP^{-1}$, let us use the Neumann series. Indeed, it yields, for $x \in \R^n \setminus \H$,
    \begin{equation} \label{eq:geometric_series}
        DP(x)^{-1} = \left( I - E(x) \right)^{-1} = \sum_{k=0}^\infty E(x)^k.
    \end{equation}
    It is easy to check by induction that for every $k \geq 1$, the matrix $E(x)^k$ has only non-zero elements in the last two rows, and that each of these elements is bounded (in absolute value) by $2^{k-1}(CL)^k$. Indeed, for $k=1$, it follows from \eqref{eq:DP_I-E} and \eqref{eq:E_size}. Inductively, if we denote by $v^{(k)},w^{(k)}$ the last two rows of $E(x)^k$, then
    \begin{align*}
        E(x)^{k+1}
        & =
        E(x)^k \cdot E(x)
        =
        \begin{bmatrix}
            0 & 0 & \ldots & 0 \\
            \vdots & \vdots & \ddots & \vdots \\
            0 & 0 & \ldots & 0 \\
            v^{(k)}_1 & v^{(k)}_2 & \ldots & v^{(k)}_n \\
            w^{(k)}_1 & w^{(k)}_2 & \ldots & w^{(k)}_n
        \end{bmatrix}
        \begin{bmatrix}
            0 & 0 & \ldots & 0 \\
            \vdots & \vdots & \ddots & \vdots \\
            0 & 0 & \ldots & 0 \\
            v^{(1)}_1 & v^{(1)}_2 & \ldots & v^{(1)}_n \\
            w^{(1)}_1 & w^{(1)}_2 & \ldots & w^{(1)}_n
        \end{bmatrix}
        \\ & =
        \begin{bmatrix}
            0 & 0 & \ldots & 0 \\
            \vdots & \vdots & \ddots & \vdots \\
            0 & 0 & \ldots & 0 \\
            v^{(k)}_{n-1} v^{(1)}_1 + v^{(k)}_n w^{(1)}_1 
            & v^{(k)}_{n-1} v^{(1)}_2 + v^{(k)}_n w^{(1)}_2
            & \ldots 
            & v^{(k)}_{n-1} v^{(1)}_n + v^{(k)}_n w^{(1)}_n \\
            w^{(k)}_{n-1} v^{(1)}_1 + w^{(k)}_n w^{(1)}_1 
            & w^{(k)}_{n-1} v^{(1)}_2 + w^{(k)}_n w^{(1)}_2
            & \ldots 
            & w^{(k)}_{n-1} v^{(1)}_n + w^{(k)}_n w^{(1)}_n
        \end{bmatrix},
    \end{align*}
    so that if $|v^{(k)}|, |w^{(k)}| \leq 2^{k-1}(CL)^k$ (by induction hypothesis) and $|v^{(1)}|, |w^{(1)}| \leq CL$, clearly the elements of $E(x)^{k+1}$ have size bounded by $2^k(CL)^{k+1}$. Thus, back in \eqref{eq:geometric_series}, we obtain
    \begin{equation*}
        \abs{DP(x)^{-1}}
        \leq
        \sum_{k=0}^\infty |E(x)^k|
        \leq
        \sum_{k=0}^\infty (2CL)^k
        =
        \frac{1}{1-2CL}
        \leq
        2
    \end{equation*}
    as soon as $L$ is small enough. This confirms the first estimate claimed in the statement, taking into account also \cref{lem:P_bijective} and \eqref{eq:P-1}.

    \textbf{Second derivatives.}
    In fact, since we had estimates in \cref{lem:tgamma,lem:tGamma} for the second derivatives of $\tgamma$ and $\tGamma$ in $\R^n \setminus (\H \cup \{x_{n-1} = 0\})$, we can differentiate twice the equations $P \circ P^{-1} = \text{id} = P^{-1} \circ P$ to arrive at the following identity (we carry out the elementary computations in \cref{sec:sec_ord_IFT}), for any $\ell \in \{1, \ldots, n\}$ and a.e. $y \in P(\R^n \setminus \H)$,
    \begin{equation} \label{eq:D^2_P-1}
        \partial_{y_j y_k} (P^{-1})_\ell (y)
        =
        - \sum_{a, b, c = 1}^n \partial_{y_a} (P^{-1})_\ell(y) \, \partial_{x_b x_c} P_a(P^{-1}(y)) \, \partial_{y_j} (P^{-1})_b(y) \, \partial_{y_k} (P^{-1})_c(y),
    \end{equation}
    where sub-indices on functions indicate their components as $\R^n \to \R^n$ functions.

    At this point, we note that, in $x \in \R^n \setminus (\H \cup \{x_{n-1} = 0\})$, it holds (see \eqref{def:P})
    \begin{gather*}
        D^2 P_a = 0 \quad\text{if } a \in \{1,\ldots,n-2\}, \\
        D^2 P_{n-1} = D^2 \tgamma(x),
        \qquad D^2 P_n = D^2 \tGamma(x).
    \end{gather*}
    Thus, for any $y \in P(\R^n \setminus (\H \cup \{x_{n-1} = 0\}))$ (recalling \cref{lem:P_bijective}, this set has full measure inside $\R^n \setminus \S$ because $P$ is Lipschitz in $\R^n \setminus \H$ and $\{x_{n-1}=0\}$ is a null set), \cref{lem:tgamma,lem:tGamma} yield the following bound for $a,b,c \in \{1,\ldots,n\}$
    \begin{equation*}
        \abs{\partial_{x_b x_c} P_a(P^{-1}(y))}
        \leq 
        C \left( \frac{\sigma_\gamma(P^{-1}(y))}{\rho(P^{-1}(y))} + \frac{\sigma_\Gamma(P^{-1}(y))}{\rho(P^{-1}(y))} \right)
        =
        C \frac{\sigma(P^{-1}(y))}{\rho(P^{-1}(y))}.
    \end{equation*}
    This, combined with the estimate $\abs{DP^{-1}} \lesssim 1$ in $\R^n \setminus \S$ (obtained in the first half of this lemma) yields, back in \eqref{eq:D^2_P-1}, the bound claimed in the statement for second derivatives of $P^{-1}$. Finally, on any compact subset of $\R^n \setminus \S$ the quantity $\rho(P^{-1}(\cdot))$ is bounded below, so the two estimates give $P^{-1} \in C^{1, 1}_{\mathrm{loc}}(\R^n \setminus \S)$.
\end{proof}

\begin{lemma} \label{lem:rho_dist_comparable}
    Under the hypotheses of \cref{lem:P_bijective}, if $L$ is small enough then 
    \begin{equation*}
        \tfrac{1}{2}\, \rho(P^{-1}(y)) \ \leq \ \dist(y, \S) \ \leq \ 2\, \rho(P^{-1}(y)),
        \qquad \forall \, y \in \R^n \setminus \S .
    \end{equation*}
\end{lemma}
\begin{proof}
    By \eqref{eq:E_size} and the argument in \cref{lem:P_bijective}, $|\tgamma(x)-\tgamma(z)| + |\tGamma(x)-\tGamma(z)| \leq CL|x-z|$ for $x, z \in \R^n \setminus \H$, hence for all $x,z \in \R^n$ by continuity (see \cref{lem:continuous}). Therefore
    \begin{equation*}
        \tfrac{1}{2}|x-z| \ \leq \ |P(x)-P(z)| \ \leq \ (1+CL)|x-z|,
        \qquad \forall\, x,z \in \R^n,
    \end{equation*}
    the lower bound being \eqref{eq:P_coercive}. In particular $P$ is injective on $\R^n$, and $P(\R^n)$ is closed by the lower bound. Moreover, since $P(\R^n)$ contains $P(\R^n \setminus \H) = \R^n \setminus \S$, which is dense in $\R^n$, it must be $P(\R^n) = \R^n$, whence $P$ is in fact a bijection of $\R^n$, and $P(\H) = \S$ by \cref{lem:parametrization_slit}.

    Let now $y \in \R^n \setminus \S$ and $x \coloneqq P^{-1}(y)$. For $z \in \H$ we have $P(z) \in \S$, so $\dist(y,\S) \leq |P(x)-P(z)| \leq (1+CL)|x-z|$, and taking the infimum over $z \in \H$ gives $\dist(y,\S) \leq (1+CL)\rho(x)$. Conversely, any $s \in \S$ is $s = P(z)$ with $z \in \H$, so $|y-s| = |P(x)-P(z)| \geq \frac{1}{2}|x-z| \geq \frac{1}{2}\rho(x)$, and taking the infimum over $s \in \S$ gives $\dist(y,\S) \geq \frac{1}{2}\rho(x)$. Choosing $L$ small enough that $CL \leq 1$ finishes the proof.
\end{proof}

\textbf{Step 3: definition and properties of $\zeta_1, \zeta_2$.}
We can finally define $\zeta_1$ and $\zeta_2$ as the last two components of $P^{-1}$, namely
\begin{equation} \label{def:zeta}
    (y'', \zeta_1(y), \zeta_2(y))
    \coloneqq
    P^{-1}(y)
    =
    P^{-1}(y'', y_{n-1}, y_n),
    \qquad
    y \in \R^n \setminus \S = \Omega,
\end{equation}
where we recall that $P$, and hence also $P^{-1}$, leave the first $n-2$ coordinates untouched. It is now immediate from \cref{lem:P_bijective} that \eqref{eq:zeta_maps} holds.

To finish the proof of \cref{th:complex_coordinates}, we are left to show all the estimates for $\zeta_1, \zeta_2$ and their derivatives. First, the estimate for the second derivatives from \cref{lem:derivatives_P-1} directly gives \eqref{eq:zeta_hessian}. In particular, since $\zeta_1$ and $\zeta_2$ are the last two components of $P^{-1}$, we have $\zeta_1, \zeta_2 \in C^{1, 1}_{\mathrm{loc}}(\Omega)$. In turn, for the first derivatives, just refining slightly our computation in \cref{lem:derivatives_P-1}, we may note that \eqref{eq:geometric_series} actually rewrites, using \eqref{eq:P-1}, like
\begin{equation*}
    DP^{-1}(y) - I 
    =
    DP(P^{-1}(y))^{-1} - I
    =
    \sum_{k=1}^\infty E(P^{-1}(y))^k,
    \qquad y \in \Omega,
\end{equation*}
and focusing on the last two rows of this matrix identity, we get that (recalling \eqref{eq:E_size} and the definition of $\sigma$)
\begin{equation*}
    \abs{\nabla\zeta_1(y) - e_{n-1}}, \abs{\nabla\zeta_2(y) - e_n}
    \leq
    \abs{ \sum_{k=1}^\infty E(P^{-1}(y))^k }
    \leq
    \sum_{k=1}^\infty (2C\sigma\left(P^{-1}(y)\right))^k
    \leq
    C \sigma\left(P^{-1}(y)\right),
\end{equation*}
which confirms \eqref{eq:zeta_1_gradient} and \eqref{eq:zeta_2_gradient}. Note that we can sum the geometric series because $\sigma(P^{-1}(y)) \leq L$ is small enough.

To finish, we are left to show \eqref{eq:zeta_1_pointwise} and \eqref{eq:zeta_2_pointwise}. For that purpose, note that
\begin{align} 
    (y'', y_{n-1}, y_n)
    &=
    y = P(P^{-1}(y)) = P(y'', \zeta_1(y), \zeta_2(y))
    \nonumber
    \\ &=
    \left(y'', \zeta_1(y) + \tgamma \left(y'', \zeta_1(y), \zeta_2(y)\right), \zeta_2(y) + \tGamma \left(y'', \zeta_1(y), \zeta_2(y)\right)\right), \label{eq:PP-1_expanded}
\end{align}
so exactly as in \eqref{eq:tgamma_0} we obtain
\begin{align} 
    \abs{\zeta_1(y) - (y_{n-1} - \gamma(y''))}
    & =
    \abs{\gamma(y'') - \tgamma (y'', \zeta_1(y), \zeta_2(y))}
    \nonumber
    \\ & \leq
    C \sigma(y'', \zeta_1(y), \zeta_2(y)) \, \rho(\zeta_1(y), \zeta_2(y))
    = 
    C \sigma(P^{-1}(y)) \, \rho(\zeta(y)), \label{eq:zeta_1_aux}
\end{align} 
which confirms \eqref{eq:zeta_1_pointwise}. 

The computation to show \eqref{eq:zeta_2_pointwise} follows a similar spirit. Indeed, using \eqref{eq:PP-1_expanded}, we obtain
\begin{align} 
    |\zeta_2(y) \, - & \left(y_n - \Gamma(y')\right) |
    =
    \abs{\Gamma(y') - \tGamma (y'', \zeta_1(y), \zeta_2(y))}
    \nonumber
    \\ & \;\; \leq
    \iint_{B'_{\rho(\zeta(y))}} \!\!\! \eta'_{\rho(\zeta(y))} (w'',w_{n-1}) \label{eq:zeta_2_triangle}
    \\ & \qquad \qquad  \cdot \abs{ \Gamma(y') - \Gamma\left(y''-w'', \zeta_1(y) + \tgamma (P^{-1}(y)) - w_{n-1}\right)} dw'' dw_{n-1}
    \nonumber
\end{align} 
At this point, we note that both points in the inner difference term can be expressed as 
\begin{equation*}
    (y'',y_{n-1})
    =
    (y'', \zeta_1(y) + \gamma(y''))
    + (0, y_{n-1} - \zeta_1(y) - \gamma(y'')), 
\end{equation*}
where, by \eqref{eq:zeta_1_aux}, we can estimate
\begin{equation*}
    \abs{(0, y_{n-1} - \zeta_1(y) - \gamma(y''))}
    \leq 
    C \rho(P^{-1}(y)),
\end{equation*}
and similarly 
\begin{align*}
    \left(y''-w'', \zeta_1(y) + \tgamma (P^{-1}(y)) - w_{n-1}\right)
    & =
    (y'', \zeta_1(y) + \gamma(y''))
    \\ & \;\; + (-w'', -w_{n-1})
    + (\tgamma (P^{-1}(y)) - \gamma(y'')), 
\end{align*}
and again using \eqref{eq:zeta_1_aux} and the fact that $(w'', w_{n-1}) \in B'_{\rho(\zeta(y))}$, we have 
\begin{equation*}
    \abs{(-w'', -w_{n-1})
    + (\tgamma (P^{-1}(y)) - \gamma(y''))}
    \leq 
    C \rho(P^{-1}(y)).
\end{equation*}
Thus, applying the mean value theorem in \eqref{eq:zeta_2_triangle}, we get 
\begin{align*}
    & \left|\zeta_2(y) - \left(y_n - \Gamma(y')\right)\right|
    \\ & \quad \leq 
    C \norm{\nabla \Gamma}_{L^\infty(B'_{C\rho(P^{-1}(y))}(y'', \zeta_1(y)+\gamma(y'')))} \rho(P^{-1}(y)) \iint_{B'_{\rho(\zeta(y))}} \!\!\! \eta'_{\rho(\zeta(y))}(w'',w_{n-1}) dw'' dw_{n-1}
    \\ & \quad =
    C \, \widetilde{\sigma}(P^{-1}(y)) \, \rho(P^{-1}(y))
\end{align*}
where in the last step we have used the definition of $\widetilde{\sigma}$ in \eqref{def:sigma_tilde} and \eqref{def:zeta}. This finishes the proof of \eqref{eq:zeta_2_pointwise}, and hence of \cref{th:complex_coordinates} (and \cref{th:complex_coordinates_intro}). 


\section{Barriers in Lipschitz slit domains}
\label{sec:barriers}

The regularized change of coordinates constructed in the previous section allows us to transfer the model profile $U_0$ (see \cref{sec:motivation}) from the flat slit $\R^n \setminus \H$ to a general slit domain. Since the geometry is no longer flat, the resulting functions are not exactly harmonic. The key observation is that the quantitative estimates for the regularized coordinates obtained in \cref{sec:reg_complex_coord} allow us to control the error and recover barriers after introducing suitable powers determined by the geometry of the slit.

\begin{theorem} \label{th:barriers}
    Let $\Omega$ be a Lipschitz slit domain as in \cref{def:lip_slit_domain} with Lipschitz constant $L$. If $L$ is small enough, by \cref{th:complex_coordinates} we can define
    \begin{equation*}
        h(y)
        \coloneqq
        \Re\left(\sqrt{\zeta_1(y) + i\zeta_2(y)}\right),
        \qquad
        y \in \Omega,
    \end{equation*} 
    where we have used the principal branch of the complex square root. Then, for any $r > 0$, if we choose $\eps$ sufficiently small satisfying
    \begin{equation*}
        \varepsilon
        \geq
        C\left(\norm{\nabla\gamma}_{L^\infty(B''_{10r})} + \norm{\nabla\Gamma}_{L^\infty(B'_{10r})}\right),
    \end{equation*}
    where $C$ is a uniform constant (depending only on $n$), it holds
    \begin{equation*}
        -\Delta(h^{1+\varepsilon})(y) \leq 0,
        \qquad
        -\Delta(h^{1-\varepsilon})(y) \geq 0,
        \qquad
        y \in B_r \cap \Omega.
    \end{equation*}
\end{theorem}
\begin{proof}
    Let us abbreviate the notation by calling $\zeta(y) \coloneqq \zeta_1(y) + i \zeta_2(y) \in \CC$. If we express $\zeta$ in polar coordinates as $\zeta(y) = \abs{\zeta(y)} e^{i\theta_{\zeta(y)}}$, then our function $h$ reads
    \begin{equation*}
        h(y)
        =
        \abs{\zeta(y)}^{1/2} \cos\left(\frac{\theta_{\zeta(y)}}{2}\right).
    \end{equation*}
    Note that, since we are using the principal branch of the complex square root, $\theta_{\zeta(y)} \in (-\pi, \pi)$ for $y \in \Omega$ (indeed, this follows from \cref{th:complex_coordinates}, because $y \in \Omega$ implies $(y'', \zeta(y)) \notin \H$), and it is also clear that $h \geq 0$ in $\Omega$.

    \begin{enumerate}[wide]
        \namedstep{initial computations}\label{it:barriers_proof_step1} Let us first show that $\Delta (h^{1+\varepsilon}) \geq 0$ in $B_r \cap \Omega$. A simple computation shows that, for $y \in \Omega$,
        \begin{align*}
            \Delta (h^{1+\varepsilon})(y)
            &=
            (1+\varepsilon) \left( \varepsilon \, h^{\varepsilon - 1}(y) \abs{\nabla h(y)}^2 + h^\varepsilon(y)\Delta h(y) \right)
            \\ &=
            (1+\varepsilon) \, h^{\varepsilon - 1}(y) \left(\varepsilon \abs{\nabla h(y)}^2 + h(y)\Delta h(y) \right).
        \end{align*}
        Thus, since we already discussed that $h \geq 0$, we are left to show that
        \begin{equation} \label{eq:claim_barrier}
            h(y) \abs{\Delta h(y)} \leq \varepsilon \abs{\nabla h(y)}^2,
            \qquad
            y \in \Omega \cap B_r.
        \end{equation}

        To verify \eqref{eq:claim_barrier}, we need to compute the derivatives of $h$. If we denote $\phi(z) \coloneqq \sqrt{z}$ the complex square root, an elementary calculation noting that $\phi$ is holomorphic (recall that we are away from $\H$) implies that, if $y \in \Omega$,
        \begin{equation*}
            \partial_{y_j} h(y)
            =
            \partial_{y_j} \Re (\phi(\zeta(y)))
            =
            \Re (\phi'(\zeta(y)) \, \partial_{y_j} \zeta(y))
            =
            \Re \left( \frac{1}{2\sqrt{\zeta(y)}} \left(\partial_{y_j} \zeta_1(y) + i\partial_{y_j} \zeta_2(y)\right) \right),
        \end{equation*}
        and similarly,
        \begin{align*}
            \partial_{y_j y_j} h(y)
            &=
            \Re \left(\phi''(\zeta(y)) \left(\partial_{y_j} \zeta(y)\right)^2 + \phi'(\zeta(y)) \, \partial_{y_j y_j} \zeta(y) \right)
            \\ &=
            \Re \left( - \frac{1}{4(\sqrt{\zeta(y)})^3} \left(\partial_{y_j} \zeta_1(y) + i\partial_{y_j} \zeta_2(y)\right)^2 + \frac{1}{2\sqrt{\zeta(y)}} \left( \partial_{y_j y_j} \zeta_1(y) + i \partial_{y_j y_j} \zeta_2(y) \right) \right).
        \end{align*}

        \namedstep{estimates for the gradient of $h$}\label{it:barriers_proof_step3} If $j = 1, \ldots, n-2$, then \eqref{eq:zeta_1_gradient} and \eqref{eq:zeta_2_gradient} yield
        \begin{align*}
            \abs{\Re \left( \frac{1}{\sqrt{\zeta(y)}} \left(\partial_{y_j} \zeta_1(y) + i \partial_{y_j} \zeta_2(y)\right) \right)}
            &\leq
            \abs{\frac{1}{\sqrt{\zeta(y)}}} \left( \abs{\partial_{y_j} \zeta_1(y)} + \abs{\partial_{y_j} \zeta_2(y)} \right)
            \\ &\leq
            C \sigma\left(P^{-1}(y)\right) \abs{\zeta(y)}^{-1/2},
        \end{align*}
        where we denote $P^{-1}(y) \coloneqq (y'', \zeta_1(y), \zeta_2(y))$ for $y \in \Omega$, in accordance to \cref{subsec:proof_coordinates}, and we denote $\sigma(x) = \sigma_\gamma(x) + \sigma_\Gamma(x)$ as in \cref{lem:derivatives_P-1}. However, for $j = n-1$, we have that, especially recalling that $\nabla \zeta_1$ resembles $e_{n-1}$ by \eqref{eq:zeta_1_gradient}, $\partial_{y_j} \zeta_1$ is close to 1, in the sense that
        \begin{equation*}
            \abs{\Re \left( \frac{1}{\sqrt{\zeta(y)}} \left(\partial_{y_{n-1}} \zeta_1(y) + i\partial_{y_{n-1}} \zeta_2(y)\right) \right) - \Re \left( \frac{1}{\sqrt{\zeta(y)}} \right)}
            \leq
            C\sigma\left(P^{-1}(y)\right) \abs{\zeta(y)}^{-1/2}.
        \end{equation*}
        Similarly, examining \eqref{eq:zeta_2_gradient} yields
        \begin{equation*}
            \abs{\Re \left( \frac{1}{\sqrt{\zeta(y)}} \left(\partial_{y_n} \zeta_1(y) + i\partial_{y_n} \zeta_2(y)\right) \right) - \Re \left( \frac{1}{\sqrt{\zeta(y)}} i \right)}
            \leq
            C\sigma\left(P^{-1}(y)\right) \abs{\zeta(y)}^{-1/2}.
        \end{equation*}
    
        Therefore, putting these 3 estimates together, we obtain
        \begin{align*}
            \abs{2\nabla h(y) - \frac{1}{\sqrt{\zeta(y)}}}^2
            &=
            \left|\left( \Re \left( \frac{1}{\sqrt{\zeta(y)}} \left(\partial_{y_j} \zeta_1(y) + i \partial_{y_j} \zeta_2(y)\right) \right) \right)_{j=1, \ldots, n} \right.
            \\ &\quad-
            \left.\left((0, \ldots, 0, \Re \left( \frac{1}{\sqrt{\zeta(y)}} \right), \Re \left( \frac{1}{\sqrt{\zeta(y)}} i \right)\right)\right|^2
            \\ &=
            \sum_{j=1}^{n-2} \abs{\Re \left( \frac{1}{\sqrt{\zeta(y)}} \left(\partial_{y_j} \zeta_1(y) + i \partial_{y_j} \zeta_2(y)\right) \right)}^2
            \\ &\quad+
            \abs{\Re \left( \frac{1}{\sqrt{\zeta(y)}} \left(\partial_{y_{n-1}} \zeta_1(y) + i\partial_{y_{n-1}} \zeta_2(y)\right) \right) - \Re \left( \frac{1}{\sqrt{\zeta(y)}} \right)}^2
            \\ & \quad +
            \abs{\Re \left( \frac{1}{\sqrt{\zeta(y)}} \left(\partial_{y_n} \zeta_1(y) + i\partial_{y_n} \zeta_2(y)\right) \right) - \Re \left( \frac{1}{\sqrt{\zeta(y)}} i \right)}^2
            \\ &\leq
            C\sigma\left(P^{-1}(y)\right)^2\abs{\zeta(y)}^{-1}.
        \end{align*} 
        Consequently, if $L \leq 1$,
        \begin{align*}
            \abs{4\abs{\nabla h(y)}^2 - \abs{\zeta(y)}^{-1}}
            &=
            \abs{\abs{2\nabla h(y)}^2 - \abs{\frac{1}{\sqrt{\zeta(y)}}}^2}
            \\ &\leq
            \abs{2\nabla h(y) - \frac{1}{\sqrt{\zeta(y)}}}^2
            + 2\abs{\frac{1}{\sqrt{\zeta(y)}}}\abs{2\nabla h(y) - \frac{1}{\sqrt{\zeta(y)}}}
            \\ &\leq
            C\sigma\left(P^{-1}(y)\right)^2\abs{\zeta(y)}^{-1} + 2\abs{\zeta(y)}^{-1/2}C\sigma\left(P^{-1}(y)\right)\abs{\zeta(y)}^{-1/2}
            \\ &\leq
            C\sigma\left(P^{-1}(y)\right)\abs{\zeta(y)}^{-1}.
        \end{align*} 
        Thus, choosing $L$ sufficiently small, we deduce that
        \begin{equation*}
            \abs{\nabla h(y)}^2 \geq \frac18 \abs{\zeta(y)}^{-1}.
        \end{equation*}


        \namedstep{estimates for the Laplacian of $h$}\label{it:barriers_proof_step4} Recalling the expression for the second derivatives of $h$ from \cref{it:barriers_proof_step1}, let us first do the following auxiliary computation:
        \begin{align*}
            \sum_{j=1}^n \left( \partial_{y_j} \zeta_1(y) + i \partial_{y_j} \zeta_2(y) \right)^2
            &=
            \sum_{j=1}^n \left( (\partial_{y_j} \zeta_1(y))^2 - (\partial_{y_j} \zeta_2(y))^2 \right)
            + 2i \sum_{j=1}^n \partial_{y_j} \zeta_1(y) \, \partial_{y_j} \zeta_2(y)
            \\ &=
            \abs{\nabla \zeta_1(y)}^2 - \abs{\nabla \zeta_2(y)}^2 + 2i \nabla \zeta_1(y) \cdot \nabla \zeta_2(y),
        \end{align*}
        Taking into account the estimates from \cref{th:complex_coordinates}, we first obtain
        \begin{equation*}
            1-C\sigma\left(P^{-1}(y)\right)
            \leq
            \abs{\nabla \zeta_1(y)}, \abs{\nabla \zeta_2(y)}
            \leq
            1+C\sigma\left(P^{-1}(y)\right),
        \end{equation*}
        which yields
        \begin{equation*}
            \abs{\abs{\nabla \zeta_1(y)}^2 - \abs{\nabla \zeta_2(y)}^2}
            \leq
            \left(1+C\sigma\left(P^{-1}(y)\right)\right)^2 - \left(1-C\sigma\left(P^{-1}(y)\right)\right)^2
            \leq
            4C\sigma\left(P^{-1}(y)\right).
        \end{equation*}
        On the other hand, again recalling \cref{th:complex_coordinates}, we have
        \begin{align*}
            \abs{\nabla \zeta_1(y) \cdot \nabla \zeta_2(y)}
            &=
            \abs{ \left( e_{n-1} + (\nabla \zeta_1(y) - e_{n-1}) \right) \cdot \left( e_n + (\nabla \zeta_2(y) - e_n) \right)}
            \\ &\leq
            \abs{ e_{n-1} \cdot e_n } + \abs{ (\nabla \zeta_1(y) - e_{n-1}) \cdot e_n } + \abs{e_{n-1} \cdot (\nabla \zeta_2(y) - e_n) }
            \\ &\quad + \abs{ (\nabla \zeta_1(y) - e_{n-1}) \cdot (\nabla \zeta_2(y) - e_n) }
            \\ &\leq
            C\sigma\left(P^{-1}(y)\right) + C\sigma\left(P^{-1}(y)\right)^2
            \leq
            C\sigma\left(P^{-1}(y)\right)
        \end{align*}
        if $L$ is small enough. All these estimates together give
        \begin{equation*}
            \abs{\sum_{j=1}^n \left( \partial_{y_j} \zeta_1(y) + i \partial_{y_j} \zeta_2(y) \right)^2}
            \leq
            C\sigma\left(P^{-1}(y)\right).
        \end{equation*}
    
        Now, recalling the expression for the derivatives of $h$ from \cref{it:barriers_proof_step1}, and also \cref{th:complex_coordinates}, we can finally bound
        \begin{align*}
            \abs{\Delta h(y)}
            &\leq
            C \left( \abs{\zeta(y)}^{-3/2} \abs{\sum_{j=1}^n \left( \partial_{y_j} \zeta_1(y) + i \partial_{y_j} \zeta_2(y) \right)^2}
            + \abs{\zeta(y)}^{-1/2} \left(\abs{D^2 \zeta_1(y)} + \abs{D^2 \zeta_2(y)} \right)
            \right)
            \\ &\leq
            C \left( \abs{\zeta(y)}^{-3/2} \sigma\left(P^{-1}(y)\right) + \abs{\zeta(y)}^{-1/2} \frac{\sigma\left(P^{-1}(y)\right)}{\rho(\zeta(y))} \right).
        \end{align*}


        \namedstep{an estimate for $\sigma(P^{-1}(y))$}\label{it:barriers_proof_step_aux} To combine the previous steps and show \eqref{eq:claim_barrier}, it is relevant to know more about $\sigma(P^{-1}(y))$. For that purpose, we first claim that, if $L$ is small enough, it holds 
        \begin{equation} \label{eq:claim_barrier_aux}
            y \in B_r \cap \Omega 
            \; \implies \; 
            P^{-1}(y) \in B_{3r}.
        \end{equation}
        Indeed, using \eqref{eq:zeta_1_pointwise} we can estimate (recalling that $\gamma(0) = 0$)
        \begin{align*}
            \abs{\zeta_1(y)}
            &\leq
            \abs{\zeta_1(y) - (y_{n-1} - \gamma(y''))} + \abs{y_{n-1}} + \abs{\gamma(y'')-\gamma(0)}
            \\ &\leq
            CL \abs{\zeta(y)} + \abs{y_{n-1}} + L\abs{y''}
            \\ &\leq
            CL \abs{\zeta(y)} + (1 + L)r,
        \end{align*}
        and the same estimate holds for $\zeta_2$ by \eqref{eq:zeta_2_pointwise}. Hence,
        \begin{equation*}
            \abs{\zeta_j(y)}^2
            \leq \left( CL\abs{\zeta(y)} + (1 + L)r \right)^2
            \leq 2C^2L^2\abs{\zeta(y)}^2 + 2(1 + L)^2r^2, \qquad j \in \{1,2\}.
        \end{equation*}
        Adding both estimates and choosing $L$ small enough so that $4C^2L^2 \leq 1 - \frac{(1+L)^2}{2}$ yields the bound $\abs{\zeta(y)}^2 \leq 8r^2$, and therefore
        \begin{equation*}
            \abs{P^{-1}(y)}
            = \left( \abs{y''}^2 + \abs{\zeta(y)}^2 \right)^{1/2}
            \leq \left( r^2 + 8r^2 \right)^{1/2}
            = 3r,
        \end{equation*}
        which finishes the proof of \eqref{eq:claim_barrier_aux}.

        Then, if $y \in \Omega \cap B_r$, it holds $x \coloneqq P^{-1}(y) \in B_{3r}$, whence $\rho(x) < 3r$ (see \eqref{def:rho}). Thus, recalling the definition of $\sigma$ in \eqref{def:sigma}, it holds 
        \begin{align*}
            \sigma(P^{-1}(y))
            & =
            \sigma(x)
            =
            \norm{\nabla \gamma}_{L^\infty(B''_{\rho(x)}(x''))}
            + \norm{\nabla \Gamma}_{L^\infty(B'_{2\rho(x)}(x'', x_{n-1}+\gamma(x''))}
            \\ & \leq 
            \norm{\nabla \gamma}_{L^\infty(B''_{\rho(x)}(x''))}
            + \norm{\nabla \Gamma}_{L^\infty(B'_{2\rho(x) + \abs{\gamma(x'') - \gamma(0)}}(x'))}
            \\ & \leq 
            \norm{\nabla \gamma}_{L^\infty(B''_{3r}(x''))}
            + \norm{\nabla \Gamma}_{L^\infty(B'_{6r + 3Lr}(x'))}
            \\ & \leq 
            \norm{\nabla \gamma}_{L^\infty(B''_{6r})}
            + \norm{\nabla \Gamma}_{L^\infty(B'_{10r})},
        \end{align*}
        where we have also used $L \leq 1/3$.


        \namedstep{end of the proof} The estimates from \cref{it:barriers_proof_step4} and the definition of $h$ allow us to give a precise estimate to the left hand side of \eqref{eq:claim_barrier}:
        \begin{align*}
            h(y) \abs{\Delta h(y)}
            &\leq
            C \abs{\zeta(y)}^{1/2} \cos\left(\frac{\theta_{\zeta(y)}}{2}\right) \left( \abs{\zeta(y)}^{-3/2} \sigma\left(P^{-1}(y)\right) + \abs{\zeta(y)}^{-1/2} \frac{\sigma\left(P^{-1}(y)\right)}{\rho(\zeta(y))} \right)
            \\ &=
            C\sigma\left(P^{-1}(y)\right) \abs{\zeta(y)}^{-1} \cos\left(\frac{\theta_{\zeta(y)}}{2}\right) \left( 1 + \frac{\abs{\zeta(y)}}{\rho(\zeta(y))} \right).
        \end{align*}
        Now note that, by definition of $\rho$ (see \eqref{def:rho}), if $\zeta_1(y) \geq 0$, then $\rho(\zeta(y)) = \abs{\zeta(y)}$, which implies that
        \begin{equation} \label{eq:h_Deltah}
            h(y) \abs{\Delta h(y)}
            \leq
            C\sigma\left(P^{-1}(y)\right) \abs{\zeta(y)}^{-1}.
        \end{equation}
        On the other hand, when $\zeta_1(y) < 0$, we have $\rho(\zeta(y)) = \abs{\zeta_2(y)}$, and we can compute
        \begin{align*}
            \cos\left(\frac{\theta_{\zeta(y)}}{2}\right) \frac{\abs{\zeta(y)}}{\rho(\zeta(y))}
            &=
            \cos\left(\frac{\theta_{\zeta(y)}}{2}\right) \frac{\abs{\zeta(y)}}{\abs{\zeta_2(y)}}
            =
            \frac{\cos(\theta_{\zeta(y)} / 2)}{\abs{\sin(\theta_{\zeta(y)})}}
            \\ &=
            \frac{\cos(\theta_{\zeta(y)} / 2)}{2\abs{\sin(\theta_{\zeta(y)}/2)} \cos(\theta_{\zeta(y)}/2) }
            =
            \frac{1}{2\abs{\sin(\theta_{\zeta(y)}/2)}}
            \leq
            \frac{1}{2\sin(\pi/4)}
            =
            \frac{\sqrt{2}}{2}
        \end{align*}
        because $\theta_{\zeta(y)} \in \left(-\pi, -\frac{\pi}{2}\right) \cup  \left(\frac{\pi}{2}, \pi \right)$ since $\zeta_1(y) < 0$. This implies that \eqref{eq:h_Deltah} also holds when $\zeta_1(y) < 0$, so in fact it is always true.
    
        Thus, recalling the estimates from \cref{it:barriers_proof_step3,it:barriers_proof_step_aux} and our choice of $\varepsilon$ in the statement, we can finally verify \eqref{eq:claim_barrier}:
        \begin{align*}
            h(y) \abs{\Delta h(y)}
            &\leq
            C\sigma\left(P^{-1}(y)\right) \abs{\zeta(y)}^{-1}
            \leq
            C \left( \norm{\nabla\gamma}_{L^\infty(B''_{10r})} + \norm{\nabla\Gamma}_{L^\infty(B'_{10r})} \right) \abs{\zeta(y)}^{-1}
            \\ &\leq
            \frac{\varepsilon}{8} \abs{\zeta(y)}^{-1}
            \leq
            \varepsilon \abs{\nabla h(y)}^2,
        \end{align*}
        which finishes the proof that $-\Delta(h^{1+\varepsilon}) \leq 0$, as explained in \cref{it:barriers_proof_step1}.
        

        \namedstep{the supersolution} The proof of the other estimate, namely $-\Delta (h^{1-\varepsilon})(y) \geq 0$, is very simple now that we have established \eqref{eq:claim_barrier}:
        \begin{align*}
            \Delta (h^{1-\varepsilon})(y)
            &=
            (1-\varepsilon) \left( -\varepsilon h^{-\varepsilon - 1}(y) \abs{\nabla h(y)}^2 + h^{-\varepsilon}(y) \Delta h(y) \right)
            \\ &=
            (1-\varepsilon) h^{-\varepsilon - 1}(y) \left(-\varepsilon \abs{\nabla h(y)}^2 + h(y)\Delta h(y) \right)
            \\ &\leq
            (1-\varepsilon) h^{-\varepsilon - 1}(y) \left(-\varepsilon \abs{\nabla h(y)}^2 + h(y)\abs{\Delta h(y)}
            \right)
            \leq
            0.
        \end{align*}
    \end{enumerate}
\end{proof}


As a consequence of having found these sub- and supersolutions, we are able to construct barriers with very precise growth near the boundary, at given scales, as in \cite[Proposition 3.3]{ClaraC1}.

\begin{corollary} \label{cor:equiv_prop3.3_from_clara}
     Let $\Omega$ be a Lipschitz slit domain as in \cref{def:lip_slit_domain} with Lipschitz constant $L$. Assume $L$ is small enough and let us define $h$ as in \cref{th:barriers}. Also, for each $r > 0$, set $\varepsilon \coloneqq C(\norm{\nabla\gamma}_{L^\infty(B''_{10r})} + \norm{\nabla\Gamma}_{L^\infty(B'_{10r})})$ where $C$ is the dimensional constant from \cref{th:barriers}. Then, there exists a solution to
     \begin{equation*}
        \left\{
        \begin{aligned}
            \Delta \varphi_r &= 0 &&\text{in } \Omega \cap B_r, \\
            \varphi_r &= 0 &&\text{on } \partial \Omega \cap B_r,
        \end{aligned}
        \right.
     \end{equation*}
     which detaches from the boundary growing like
     \begin{equation} \label{eq:varphi_sandwich}
        (3\sqrt{r})^{-\varepsilon} h(y)^{1+\varepsilon}
        \leq
        \varphi_r(y)
        \leq
        (3\sqrt{r})^{\varepsilon} h(y)^{1-\varepsilon},
        \qquad
        y \in \Omega \cap B_r,
     \end{equation}
     and moreover it holds
     \begin{equation*}
        \norm{\varphi_r - h}_{L^\infty (\Omega \cap B_r)}
        \leq
        15\sqrt{r} \varepsilon
        =
        15C(\norm{\nabla\gamma}_{L^{\infty}(B''_{10r})} + \norm{\nabla\Gamma}_{L^{\infty}(B'_{10r})}) \sqrt{r}.
    \end{equation*}
\end{corollary}
\begin{proof}
    Fix $r > 0$ and define $\varphi_r$ as the solution to
    \begin{equation} \label{eq:def_varphi}
        \left\{
        \begin{aligned}
            \Delta \varphi_r &= 0 &&\text{in } \Omega \cap B_r, \\
            \varphi_r &= h &&\text{on } \partial (\Omega \cap B_r).         
        \end{aligned}
        \right.
     \end{equation}
     The pointwise estimates \eqref{eq:zeta_1_pointwise} and \eqref{eq:zeta_2_pointwise} from \cref{th:complex_coordinates} combined with the fact that $\gamma(0) = \Gamma(0) = 0$ imply that, for any $y \in \Omega \cap B_r$,
    \begin{equation*}
        \abs{\zeta_1(y)} \leq \abs{y_{n-1}} + \abs{\gamma(y'')} + CL\abs{\zeta(y)} \leq r(1+L) + CL\abs{\zeta(y)},
    \end{equation*}
    and, analogously, $\abs{\zeta_2(y)} \leq r(1+L) + CL\abs{\zeta(y)}$. Adding the previous bounds yields 
    \begin{equation*}
    \abs{\zeta_1(y)} + \abs{\zeta_2(y)} \leq 2r(1+L) + 2CL (\abs{\zeta_1(y)} + \abs{\zeta_2(y)}),
    \end{equation*}
    and choosing $L$ sufficiently small allows us to absorb the terms on the right-hand side to obtain
    \begin{equation*}
        \norm{\zeta_1}_{L^\infty(\Omega \cap B_r)} \leq 3r
        \quad \text{and} \quad
        \norm{\zeta_2}_{L^\infty(\Omega \cap B_r)} \leq 3r.
    \end{equation*}

     Now, by definition of $h$, it holds
     \begin{equation} \label{eq:bound_h}
        h(y)
        \leq
        \sqrt{\abs{\zeta_1(y) + i\zeta_2(y)}}
        \leq
        \sqrt{ \sqrt{(3r)^2 + (3r)^2} }
        \leq
        3\sqrt{r},
        \quad
        y \in \overline{\Omega \cap B_r}.
    \end{equation}
    This implies that $h \leq (3\sqrt{r})^\varepsilon h^{1-\varepsilon}$ on $\partial (\Omega \cap B_r)$, whence \eqref{eq:def_varphi} and \cref{th:barriers} yield, by the maximum principle,
    \begin{equation*}
        \varphi_r (y)
        \leq
        (3\sqrt{r})^\varepsilon h^{1-\varepsilon},
        \qquad
        y \in \Omega \cap B_r.
    \end{equation*}
    Similarly, $h \geq (3\sqrt{r})^{-\varepsilon} h^{1+\varepsilon}$ on $\partial (\Omega \cap B_r)$, so again using \eqref{eq:def_varphi}, \cref{th:barriers} and the maximum principle, we end the proof of \eqref{eq:varphi_sandwich}.

    We continue by noting that, on top of \eqref{eq:varphi_sandwich}, we know (from \eqref{eq:bound_h}) that $(3\sqrt{r})^{-\varepsilon} h^{1+\varepsilon} \leq h \leq (3\sqrt{r})^\varepsilon h^{1-\varepsilon}$ in $\Omega \cap B_r$. Therefore, we may bound, always having \eqref{eq:bound_h} in mind,
    \begin{align*}
        \norm{\varphi_r - h}_{L^\infty(\Omega \cap B_r)}
        &\leq
        \norm{(3\sqrt{r})^\varepsilon h^{1-\varepsilon} - (3\sqrt{r})^{-\varepsilon} h^{1+\varepsilon}}_{L^\infty(\Omega \cap B_r)}
        \\ &\leq
        \sup_{t \in [0, 3\sqrt{r}]} \left[ (3\sqrt{r})^\varepsilon t^{1-\varepsilon} - (3\sqrt{r})^{-\varepsilon} t^{1+\varepsilon} \right]
        \\ &=
        3\sqrt{r} \sup_{t \in [0, 3\sqrt{r}]} \left[ \left( \frac{t}{3\sqrt{r}} \right)^{1-\varepsilon}  - \left( \frac{t}{3\sqrt{r}} \right)^{1+\varepsilon} \right]
        \\ &=
        3\sqrt{r} \sup_{\tau \in [0, 1]} \left[ \tau^{1-\varepsilon} - \tau^{1+\varepsilon} \right]
        \\ &\leq
        15 \sqrt{r} \varepsilon,
    \end{align*}
    where the last inequality is a simple real variable exercise\footnote{
        Given $\varepsilon > 0$, the maximum of $f_\varepsilon(\tau) \coloneqq \tau^{1-\varepsilon} - \tau^{1+\varepsilon}$ for $\tau \in [0, 1]$ is attained at  $\tau_\varepsilon = \left( \frac{1-\varepsilon}{1+\varepsilon} \right)^{1/(2\varepsilon)} \in (0, 1)$, and
        \begin{equation*}
            f_\varepsilon(\tau_\varepsilon)
            =
            \tau_\varepsilon \left[ \left( \frac{1-\varepsilon}{1+\varepsilon} \right)^{-\frac12} - \left( \frac{1-\varepsilon}{1+\varepsilon} \right)^{\frac{1}{2}} \right]
            \leq
            \sqrt{\frac{1+\varepsilon}{1-\varepsilon}} - \sqrt{\frac{1-\varepsilon}{1+\varepsilon}}
            =
            \frac{ \dfrac{1+\varepsilon}{1-\varepsilon} - \dfrac{1-\varepsilon}{1+\varepsilon} }{ \sqrt{\dfrac{1+\varepsilon}{1-\varepsilon}} + \sqrt{\dfrac{1-\varepsilon}{1+\varepsilon}} }
            =
            \frac{ \dfrac{4\varepsilon}{(1-\varepsilon)(1+\varepsilon)} }{ \sqrt{\dfrac{1+\varepsilon}{1-\varepsilon}} + \sqrt{\dfrac{1-\varepsilon}{1+\varepsilon}} }
            \leq
            5\varepsilon
        \end{equation*}
        as soon as $\varepsilon$ is small enough, which amounts to $L$ being small enough.
    }.
\end{proof}

Note that in the proof of \cref{cor:equiv_prop3.3_from_clara} we have shown that $h(y) \leq 3\sqrt{r}$ for $y \in \overline{\Omega \cap B_r}$. In the following result we prove a similar lower bound for $h$ which holds for points that lie away from the slit or, more precisely, inside certain conical regions of aperture depending on $L$.
\begin{lemma} \label{lem:h_comparable_sqrt_r}
     Let $\alpha_0 < \pi$ and $\beta_0 < \frac{\pi}{2}$ be fixed constants and let $\alpha$ and $\beta$ be angles such that $|\alpha| \leq \alpha_0$ and $|\beta| \leq \beta_0$. 
     Let $\Omega$ be a Lipschitz slit domain as in \cref{def:lip_slit_domain} with Lipschitz constant $L \coloneqq L(\alpha_0, \beta_0)$ small enough, depending on $\alpha_0$ and $\beta_0$. Define $h$ as in \cref{th:barriers}. 
     Then, for all $r \in (0, 1)$ and any unit vector $e'' \in \R^{n-2}$, any point of the form 
     \begin{equation*}
         y \coloneqq (r \sin(\beta) e'', r \cos{\alpha} \cos{\beta}, r \sin{\alpha} \cos{\beta}),
     \end{equation*}
     belongs to $\Omega$. Moreover,
     \begin{equation*}
         \kappa \sqrt{r} \leq h(y) \leq 3\sqrt{r},
     \end{equation*}
     where $0 < \kappa < 1$ is a constant depending only on $\alpha_0$ and $\beta_0$.
\end{lemma}

\begin{proof}
    First, since $\Omega$ is a Lipschitz slit domain, we can ensure that any point of the form $y = (r \sin (\beta) e'', r \cos{\alpha} \cos{\beta}, r \sin{\alpha} \cos{\beta})$ is in $\Omega$ by taking the Lipschitz constant $L$ small enough depending on $\alpha_0$ and $\beta_0$. Moreover, the upper bound follows from \eqref{eq:bound_h} because $y \in \Omega \cap \overline{B_r}$.
    
    Let us then turn to the lower bound. Recalling the explicit formula for $h$ from \cref{th:barriers}
    \begin{equation}
        \label{eq:h_explicit_formula}
        h(y) = |\zeta(y)|^{1/2} \cos \left( \frac{\theta_{\zeta(y)}}{2}\right), 
    \end{equation}
    the strategy will be to obtain lower bounds for $|\zeta(y)|$ and for $ \cos \left( \theta_{\zeta(y)} / 2\right)$ separately. 

    For the lower bound for $|\zeta(y)|$, applying \eqref{eq:zeta_1_pointwise} and \eqref{eq:zeta_2_pointwise} from \cref{th:complex_coordinates} at $y$, we obtain
    \begin{align*}
        |\zeta_1(y) - r \cos{\alpha} \cos{\beta}| 
        & \leq 
        Lr +  C L \big( |\zeta_1(y)| + |\zeta_2(y)|\big),
        \\ 
        |\zeta_2(y) - r\sin{\alpha} \cos{\beta}| 
        & \leq 
        L r + CL \big( |\zeta_1(y)| + |\zeta_2(y)| \big).
    \end{align*}
    Adding up the previous two estimates we get, after a simple triangle inequality, 
    \begin{equation*}
        |\zeta_1(y)| + |\zeta_2(y)| \leq 2Lr + 2CL \left( |\zeta_1(y)| + |\zeta_2(y)|\right) + r (|\!\cos \alpha| + |\!\sin \alpha|)\cos \beta.
    \end{equation*}
    Since $r (|\!\cos \alpha| + |\!\sin \alpha|)\cos \beta \leq r\sqrt{2}$ by Cauchy-Schwarz, the above estimate yields
    \begin{equation*}
        |\zeta_1(y)| + |\zeta_2(y)| \leq r\frac{\sqrt{2} + 2L}{1 - 2CL} \leq 2r
    \end{equation*}
    for $L$ small enough, and therefore, 
    \begin{equation} \label{eq:tmp_zeta_close_rei_alpha}
        \left| (\zeta_1(y) + i\zeta_2(y)) - r\cos{(\beta)} e^{i \alpha}\right|\leq 2Lr + 2C L (|\zeta_1(y)| + |\zeta_2(y)|) \leq (2 + 4C)L r.
    \end{equation}
    Using the reverse triangle inequality in \eqref{eq:tmp_zeta_close_rei_alpha}, we can get a lower bound for $|\zeta(y)|$:
    \begin{equation}
        \label{eq:modulus_lower_bound}
        |\zeta(y)| = |\zeta_1(y) + i\zeta_2(y)| \geq  r \left(\cos\beta  - (2+4C)L\right) \geq r \left(\cos\beta_0  - (2+4C)L\right) \geq r\frac{\cos\beta_0}{2},
    \end{equation}
    where we have chosen $L$ small enough depending on $\beta_0$.

    Let us turn to the lower bound for $\cos \left(\theta_{\zeta(y)}/2\right)$. In fact, this is equivalent to $\theta_{\zeta(y)}$ having an upper bound smaller than $\pi$. Expressing $\zeta(y)$ in polar coordinates as $\zeta(y) = |\zeta(y)|e^{i \theta_{\zeta(y)}}$, then \eqref{eq:tmp_zeta_close_rei_alpha} implies that $\zeta(y)$ is contained in $B \coloneqq B_{rL(2 + 4C)}(r \cos (\beta) e^{i\alpha})$. Then, the maximum possible deviation of $\theta_{\zeta(y)}$ from $\alpha$ is
    \begin{equation*}
        \delta_\mathrm{max} \coloneqq \sup_{\xi \in B} |\theta_{\xi} - \alpha|,
    \end{equation*}
    and corresponds to the angle between the center of $B$ and the tangency point of the ball $B$ with respect to a line from the origin, so
    \begin{equation*}
        \sin{(\delta_{\max})} = \frac{rL(2 + 4C)}{r \cos \beta} \leq \frac{L(2 + 4C)}{\cos\beta_0},
    \end{equation*}
    and choosing $L$ sufficiently small depending on $\alpha_0$ and $\beta_0$, we can make  $\delta_\mathrm{max} \leq \frac{\pi - \alpha_0}{2}$, whence
    \begin{equation*}
        |\theta_{\zeta(y)}| \leq |\alpha| + \delta_{\max} \leq \alpha_0 + \frac{\pi - \alpha_0}{2} = \frac{\pi + \alpha_0}{2} \; (< \pi).
    \end{equation*}
    Consequently,
    \begin{equation}
        \label{eq:cos_lower_bound}
        \cos \left( \frac{\theta_{\zeta(y)}}{2}\right) \geq \cos\left(\frac{\pi+\alpha_0}{4}\right) > 0.
    \end{equation}
    Substituting \eqref{eq:modulus_lower_bound} and \eqref{eq:cos_lower_bound} in \eqref{eq:h_explicit_formula}, we get
    \begin{equation*}
        h(y) = |\zeta(y)|^{1/2} \cos \left( \frac{\theta_{\zeta(y)}}{2}\right) \geq \cos\left(\frac{\pi+\alpha_0}{4}\right) \left( r\frac{\cos\beta_0}{2}\right)^{1/2}
        =:
        \kappa r^{1/2},
    \end{equation*}
    which completes the proof with $\kappa = \kappa(\alpha_0, \beta_0) \in (0, 1)$.
\end{proof}


\section{Boundary behavior in \texorpdfstring{$C^1$}{C1} slit domains} \label{sec:C1_slit_domains}

In this section we prove the main boundary estimates (\cref{th:lower_bound_thm,th:upper_bound_thm}) by combining the barriers constructed in the previous section with an iterative argument across scales. The geometry-dependent exponents of these barriers produce a small loss at each step, and the accumulation of these losses over the iteration yields the correction factors appearing in our main theorems. The iteration is inspired by the one developed in \cite{ClaraC1}, although some new ingredients are required to adapt to the more general geometry considered here.

\subsection{Almost positivity of harmonic functions in slit domains}\label{subsec:almost_positivity}

Before we prove our main theorems, we shall prove some auxiliary results, the first one of which is an almost positivity property for harmonic functions in slit domains. It can be seen as a quantitative version of the maximum principle.
\begin{remark}
    Throughout the section we will work with weak solutions to the Dirichlet problem in $\Omega \cap B_1$. These are continuous in $\Omega \cap B_1$ by elementary interior regularity, and also continuous up to the slit ($\S$ is Lipschitz and has codimension 1, so it carries uniform positive capacity and Wiener's criterion applies). 
\end{remark}

\begin{lemma}\label{lem:almost_pos_sublem}
    Let $\Omega = \R^n \setminus \mathcal{S}$ be a Lipschitz slit domain as in \cref{def:lip_slit_domain} with Lipschitz constant $L \leq 1$. For a given small $\delta > 0$, let $\Omega_{\delta} = \{x \in \Omega : \dist(x,\mathcal{S}) \geq \delta\}$. Let $u$ satisfy
    \begin{equation*}
        \left\{
            \begin{aligned}
              \Delta u &= 0 &&\text{in } \Omega \cap B_1, \\
              u &= 0 &&\text{on } \mathcal{S} \cap B_1, \\
            \end{aligned}
        \right.
        \quad\text{and}\quad
        \left\{
            \begin{aligned}
              u &\geq 1 &&\text{in } \Omega_{\delta} \cap B_1, \\
              u &\geq -\delta &&\text{in } B_1.
            \end{aligned}
        \right.
    \end{equation*}
    Then, there exists a small constant $c = c(n) > 0$ such that for every $k \geq 0$ satisfying $k\delta \leq \frac{1}{10}$, we have
    \begin{equation*}
        u \geq -\delta(1-c)^k \quad\text{in } B_{1-5k\delta}.
    \end{equation*}
\end{lemma}
\begin{proof}
    Define $w \coloneqq \min\{u,0\} + \delta$. Since $u$ is harmonic in $\Omega \cap B_1$ and $u = 0$ on $\mathcal{S} \cap B_1$, it follows that $w$ is superharmonic, namely $-\Delta w \geq 0$ in $B_1$. Moreover, $0 \leq w \leq \delta$ by construction.

    Let $x_0 \in \mathcal{S} \cap B_{1-4\delta}$. Then, by the weak Harnack inequality (see \cite[Lemma B.4]{FR22}), 
    \begin{equation} \label{eq:weak_harnack_almos_pos}
        \inf_{B_{2\delta}(x_0)} w \geq c(n) \, \delta^{-n}\norm{w}_{L^1(B_{2\delta}(x_0))}.
    \end{equation}
    
    The idea now is to show that $w = \delta$ in a uniform portion of $B_{2\delta}(x_0)$, which will give us an effective pointwise lower bound for $w$. For that purpose, let $y_0 \coloneqq x_0 + \sqrt{2}\delta e_n$ and consider the cone with vertex $y_0$, axis $e_n$ and opening of $\pi/4$, namely
    \begin{equation*}
        \mathcal{C}_\delta(x_0) \coloneqq \big\{x \in \R^n : (x - y_0) \cdot e_n \geq {\textstyle \frac{1}{\sqrt{2}}}\abs{x - y_0} \big\}.
    \end{equation*}
    Then, it is clear that $y_0 \in B_{2\delta}(x_0)$ and $\abs{\mathcal{C}_\delta(x_0) \cap B_{2\delta}(x_0)} \geq c(n)\abs{B_{2\delta}(x_0)}$ (see \cref{fig:lemma51}).
    
    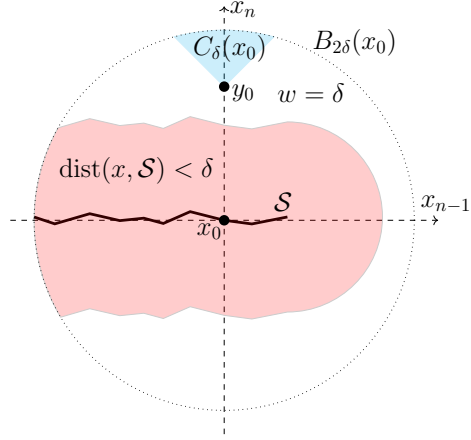
\begin{figure}[h]
        \centering
        \resizebox{0.4\textwidth}{!}{
        \begin{tikzpicture}[scale=0.8]  
            \draw[dashed, ->] (5.08, 0) -- (5.08, 10.16);
            \draw[dashed, ->] (0, 5.08) -- (10.16, 5.08);
            
            \filldraw[cyan, opacity=0.2] (3.897, 9.438) -- (5.08, 8.255) -- (6.263, 9.438) -- (6.263, 9.438) arc[start angle=74.814, end angle=105.186, radius=4.516] -- cycle;
            
            \draw[thin] (0.565, 5.158) -- (0.565, 5.158) -- (1.055, 4.995) -- (1.886, 5.236) -- (2.601, 5.072) -- (3.17, 5.13) -- (3.634, 5.005) -- (4.272, 5.285) -- (5.08, 5.08) -- (5.737, 4.993) -- (6.582, 5.159);

            \filldraw[fill=red, opacity=0.2] (6.582, 7.417) arc[start angle=-90, end angle=90, x radius=2.258, y radius=-2.258] -- (5.737, 2.735) -- (5.08, 2.822) -- (4.272, 3.027) -- (3.634, 2.747) -- (3.17, 2.872) -- (2.601, 2.814) -- (1.886, 2.979) -- (1.195, 2.778) arc[start angle=149.349, end angle=209.094, x radius=4.516, y radius=-4.516] -- (1.886, 7.494) -- (2.601, 7.33) -- (3.17, 7.388) -- (3.634, 7.262) -- (4.272, 7.542) -- (5.08, 7.338) -- (5.737, 7.251) -- cycle;

            \draw[dotted] (5.08, 5.08) circle[radius=4.516];

            \node[circle, fill, inner sep=2pt] at (5.08, 8.255) {};
            \node[circle, fill, inner sep=2pt] at (5.08, 5.08) {};

            \node[anchor=center, font=\Large] at (8.181, 9.31) {$B_{2\delta}(x_0)$};
            \node[anchor=center, font=\Large] at (5.532, 8.096) {$y_0$};
            \node[anchor=center, font=\Large] at (7.145, 8.121) {$w = \delta$};
            \node[anchor=center, font=\Large] at (10.35, 5.473) {$x_{n-1}$};
            \node[anchor=center, font=\Large] at (5.529, 10.083) {$x_n$};
            \node[anchor=center, font=\Large] at (5.219, 9.17) {$C_\delta(x_0)$};
            \node[anchor=center, font=\Large] at (4.73, 4.808) {$x_0$};
            \node[anchor=center, font=\Large] at (6.483, 5.524) {$\mathcal{S}$};
            \node[anchor=center, font=\Large] at (2.995, 6.268) {$\operatorname{dist}(x,\mathcal{S}) < \delta$};       
        \end{tikzpicture}
        }
        \caption{Construction of the cone $\mathcal{C}_\delta(x_0)$.}
        \label{fig:lemma51}
    \end{figure}
    
    Moreover, since $\Omega$ is a Lipschitz slit domain with constant $L$, the distance from $y_0$ to the slit $\S$ can be estimated by the distance to the circular cone of vertex $x_0$, axis $e_n$ and slope $L$, so
    \begin{equation*}
        \dist(y_0,\mathcal{S}) \geq \frac{\sqrt{2}\delta}{\sqrt{1 + L^2}} \geq \delta,
    \end{equation*}
    where we have also used that $L \leq 1$. Therefore, it follows that $\dist(y, \S) \geq \delta$ for every $y \in \mathcal{C}_\delta(x_0) \cap B_{2\delta}(x_0)$ by a simple geometric argument, noting that the aperture of $\mathcal{C}_\delta(x_0)$ is smaller than the one allowed for $\S$ by the Lipschitz constant $L \leq 1$. This shows, recalling also that $w = \delta$ in $\Omega_\delta \cap B_1$,
    \begin{equation*}
        \mathcal{C}_\delta(x_0) \cap B_{2\delta}(x_0) 
        \, \subset \,
        \Omega_\delta \cap B_{2\delta}(x_0) 
        \, \subset \,
        \{w = \delta\} \cap B_{2\delta}(x_0),
    \end{equation*}
    and hence,
    \begin{equation*}
        \norm{w}_{L^1(B_{2\delta}(x_0))} 
        \geq 
        \delta\abs{\{w = \delta\} \cap B_{2\delta}(x_0)} 
        \geq 
        \delta\abs{\mathcal{C}_\delta(x_0) \cap B_{2\delta}(x_0)} 
        \geq 
        c(n) \, \delta^{n+1}.
    \end{equation*}
    Combining this estimate with \eqref{eq:weak_harnack_almos_pos} we conclude that $w \geq c(n) \, \delta$ in $B_{2\delta}(x_0)$ for some $c(n) > 0$. Noting that $x_0 \in \S \cap B_{1-4\delta}$ was arbitrary, and that all constants are uniform (they only depend on $n$), we deduce that 
    \begin{equation*}
        w \geq c\, \delta \; \text{ in $B_{1-5\delta}$}, 
        \quad 
        \text{or, equivalently, }
        \quad 
        \text{$u \geq -\delta(1-c)$ in $B_{1-5\delta}$}.
    \end{equation*}

    At this point we observe that all the hypotheses of the theorem are satisfied in $B_{1-5\delta}$, except for the fact that we have obtained an improvement factor $1-c$ in the condition $u \geq -\delta(1-c)$. Hence, we can run the same arguments of the proof up until this point, but now for the function $w \coloneqq \min\{u,0\} + \delta(1-c)$, thus obtaining that $u \geq -\delta(1-c)^2$ in $B_{1-10\delta}$. We can keep iterating this process to conclude that $u \geq -\delta(1-c)^k$ in $B_{1-5k\delta}$ for $k \geq 1$. Concretely, we prefer to only consider $k \leq \frac{1}{10\delta}$ to ensure that $B_{1/2} \subset B_{1-5k\delta}$ for all our $k$.
\end{proof}

\begin{remark}
    The result of \cref{lem:almost_pos_sublem} is true for more general domains than Lipschitz slit domains, as can be seen by substituting the strategy of truncating the solution and then applying the weak Harnack inequality for subharmonic functions, by the so-called Bourgain's non-degeneracy estimate for harmonic measure, see e.g. \cite[Lemma 8.16]{PratsTolsa}.
\end{remark}

The previous lemma shows that the negativity of the solution $u$ in $B_1$ becomes smaller as we look at smaller and smaller scales. As it turns out, this result can be further improved to yield non-negativity of $u$ whenever the parameter $\delta$ is sufficiently small.
\begin{lemma}\label{lem:amost_pos_int_cond}
    Let $\Omega = \R^n \setminus \mathcal{S}$ be a Lipschitz slit domain as in \cref{def:lip_slit_domain} with Lipschitz constant $L \leq 1$. There exists a small constant $\delta > 0$ depending only on $n$ such that, if $\Omega_\delta \coloneqq \{x \in \Omega : \dist(x,\mathcal{S}) \geq \delta\}$ and $u$ is any function satisfying
    \begin{equation*}
        \left\{
            \begin{aligned}
              \Delta u &= 0 &&\text{in } \Omega \cap B_1, \\
              u &= 0 &&\text{on } \mathcal{S} \cap B_1, \\
            \end{aligned}
        \right.
        \quad\text{and}\quad
        \left\{
            \begin{aligned}
              u &\geq 1 &&\text{in } \Omega_{\delta} \cap B_1, \\
              u &\geq -\delta &&\text{in } B_1,
            \end{aligned}
        \right.
    \end{equation*}
    then $u \geq 0$ in $B_{1/2}$.
\end{lemma}
\begin{proof}
    We begin the proof by first showing the following claim: if $\delta$ is small depending only on $n$, then for some small $a > 0$ we have
    \begin{equation}\label{eq:almost_pos_claim}
        \left\{
            \begin{aligned}
              u &\geq a &&\text{in } B_{1/2} \cap \Omega_{\delta/2}, \\
              u &\geq -\delta a &&\text{in } B_{1/2}.
            \end{aligned}
        \right.
    \end{equation}
    In other words, we want to prove that the negativity of $u$ becomes smaller in $B_{1/2}$ by a factor $a$ in exchange for the bound $u \geq 1$ worsening also by the same factor.

    Let us start by proving the first estimate in \eqref{eq:almost_pos_claim}. Given $x \in B_{1/2} \cap \Omega_{\delta/2}$, we may assume $\frac{\delta}{2} \leq \dist(x,\mathcal{S}) < \delta$ since $u \geq 1$ in $\Omega_{\delta}$ and we will take $a < 1$. We now want to connect $x$ through a Harnack chain to a point $y \in \Omega_\delta \cap B_1$. To that end, we set $y \coloneqq x + \sqrt{2}\delta\nu$ with $\nu \in \mathbb{S}^{n-1}$ and distinguish the following two cases:
    \begin{enumerate}
        \item\label{it:case_1} If $x_{n-1} \leq \gamma(x'')$, we simply take $\nu \coloneqq e_n$ if $x_n > \Gamma(x')$ or $\nu \coloneqq -e_n$ otherwise. In this way, similarly to the proof of \cref{lem:almost_pos_sublem}, $\dist(y,\S)$ can be bounded from below by the distance from $y$ to the cone of vertex $(x'',x_{n-1},\Gamma(x'))$, axis $\nu$ and slope $L$. This, combined with the fact that $L \leq 1$, yields
        \begin{equation*}
            \dist(y,\S)
            \geq \frac{|x_n - \Gamma(x')| + \sqrt{2}\delta}{\sqrt{1 + L^2}}
            \geq \frac{\sqrt{2}\delta}{\sqrt{1 + L^2}}
            \geq \delta.
        \end{equation*}
        \item\label{it:case_2} If $x_{n-1} > \gamma(x'')$, then we set $\nu \coloneqq e_{n-1}$. To see that $\dist(y,\S) \geq \delta$, fix any $\bar{x} \in \S$. Then, it holds $\bar{x}_{n-1} \leq \gamma(\bar{x}'')$ which, combined with the fact that $x_{n-1} > \gamma(x'')$ and $L \leq 1$, implies
        \begin{equation} \label{eq:almost_pos_coord_bound}
            \bar{x}_{n-1} - x_{n-1} < \gamma(\bar{x}'') - \gamma(x'') \leq L|\bar{x}'' - x''| \leq |\bar{x}'' - x''|.
        \end{equation}
        Hence,
        \begin{align*}
            |y - \bar{x}|^2
            &\geq
            |y'' - \bar{x}''|^2 + (y_{n-1} - \bar{x}_{n-1})^2
            \\ &=
            |x'' - \bar{x}''|^2 + (x_{n-1} + \sqrt{2}\delta - \bar{x}_{n-1})^2
            \\ &\geq
            \frac{1}{2}\left( |x'' - \bar{x}''| + x_{n-1} + \sqrt{2}\delta - \bar{x}_{n-1} \right)^2
            \\ &\geq
            \frac{1}{2}\left( \sqrt{2}\delta \right)^2 = \delta^2.
        \end{align*}
        Here we have used: for the equality, the definition of $y$; for the second inequality, the fact that $\frac{1}{2}(a + b)^2 \leq a^2 + b^2$ for any $a,b \in \R$; lastly, for the last inequality, the fact that $|\bar{x}'' - x''| + x_{n-1} - \bar{x}_{n-1} \geq 0$ obtained in \eqref{eq:almost_pos_coord_bound}. Given that $\bar{x} \in \S$ was arbitrary, we conclude that $\dist(y,\S) \geq \delta$.
    \end{enumerate}
    Now, as already mentioned above, we may connect $x$ and $y$ with a Harnack chain of balls in $\Omega \cap B_1$ of radius $\delta/4$ so that, applying the interior Harnack inequality to $w \coloneqq u + \delta \geq 0$ along this chain, we find
    \begin{equation*}
        w(x) \geq \frac{1}{C}w(y) \geq \frac{1}{C},
    \end{equation*}
    where we have also used the fact that $u \geq 1$ in $\Omega_\delta$. Thus, we deduce that for $\delta \leq \frac{1}{2C}$,
    \begin{equation*}
        u(x) = w(x) - \delta \geq \frac{1}{C} - \delta \geq \frac{1}{2C}.
    \end{equation*}
    Since $x \in B_{1/2} \cap \Omega_{\delta/2}$ was arbitrary, the first estimate of \eqref{eq:almost_pos_claim} follows by taking $a \coloneqq \frac{1}{2C}$.

    As for the second estimate in \eqref{eq:almost_pos_claim}, we apply \cref{lem:almost_pos_sublem} with $k = \left\lfloor\frac{1}{10\delta}\right\rfloor \leq \frac{1}{10\delta}$ to obtain
    \begin{equation*}
        u \geq -\delta(1-c)^k \quad\text{in } B_{1/2},
    \end{equation*}
    noting that $B_{1/2} \subset B_{1-5k\delta}$ by our choice of $k$. Taking $\delta$ small enough (which makes $k$ large) depending on $c$ and $C$ (which only depend on $n$) so that $(1-c)^k \leq \frac{1}{2C} = a$, we conclude that $u \geq -\delta a$ in $B_{1/2}$, which finishes the proof of the second estimate of \eqref{eq:almost_pos_claim}.

    With the claim proven, we observe that for any point $x_0  \in \mathcal{S} \cap B_{3/4}$ the rescaled function $\tilde{u}(x) \coloneqq u(x_0 + \frac{1}{4}x)$ satisfies, in the respective rescaled domains $\widetilde{\Omega} \coloneqq 4(\Omega - x_0)$ and $\widetilde{\S} \coloneqq 4(\S - x_0)$,
    \begin{equation*}
        \left\{
            \begin{aligned}
              \Delta \tilde{u} &= 0 &&\text{in } B_1 \cap \widetilde{\Omega}, \\
              \tilde{u} &= 0 &&\text{on } \widetilde{\S} \cap B_1, \\
            \end{aligned}
        \right.
        \quad\text{and}\quad
        \left\{
            \begin{aligned}
              \tilde{u} &\geq 1 &&\text{in } B_1 \cap \widetilde{\Omega}_{4\delta}, \\
              \tilde{u} &\geq -4\delta &&\text{in } B_1.
            \end{aligned}
        \right.
    \end{equation*}
    Thus, taking $\delta$ sufficiently small, we may apply our claim \eqref{eq:almost_pos_claim} to $\tilde{u}$ with $4\delta$ at every dyadic scale and conclude that, after undoing the rescaling,
    \begin{equation} \label{eq:almost_pos_iter}
        \left\{
            \begin{aligned}
              u &\geq a^k &&\text{in } B_{2^{-k-2}}(x_0) \cap \Omega_{2^{-k}\delta}, \\
              u &\geq -4\delta a^k &&\text{in } B_{2^{-k-2}}(x_0).
            \end{aligned}
        \right.
    \end{equation}
    Now, given any point $x \in B_{1/2}$ with $d(x) \coloneqq \dist(x,\S) < \delta$, we may choose $k$ large enough so that $2^{-k}\delta \leq d(x) < 2^{-k+1}\delta$ and therefore it holds $x \in \Omega_{2^{-k}\delta}$. Moreover, if $x_0 \in \S$ is a point such that $d(x) = |x - x_0|$, then we can ensure that $x_0 \in B_{3/4}$ by taking $\delta < 1/4$ so that $|x_0| \leq |x| + d(x) < 1/2 + \delta < 3/4$. Lastly, we also have that $|x - x_0| = d(x) < 2^{-k+1}\delta < 2^{-k-2}$ if we take $\delta < 1/8$. Combining everything, it follows that $x \in B_{2^{-k-2}}(x_0) \cap \Omega_{2^{-k}\delta}$ and, since $x_0 \in \S \cap B_{3/4}$, we can use the first estimate in \eqref{eq:almost_pos_iter} to conclude that $u(x) \geq a^k > 0$. This shows that $u \geq 0$ in $B_{1/2} \setminus \Omega_\delta$ which, combined with the fact that $u \geq 1$ in $B_{1/2} \cap \Omega_\delta$, implies $u \geq 0$ in $B_{1/2}$ and concludes the proof.
\end{proof}

\begin{corollary}\label{cor:almost_pos}
    Let $\Omega = \R^n \setminus \mathcal{S}$ be a Lipschitz slit domain as in \cref{def:lip_slit_domain} with Lipschitz constant $L \leq 1$ and let $K \subset\subset \Omega \cap B_1$ be a compact set. Then, there exists a constant $\mu_0 > 0$ depending only on $n$ and $K$ (through $\dist(K,\mathcal{S})$ and $\dist(K,\partial B_1)$) such that if $u$ satisfies
    \begin{equation*}
        \left\{
            \begin{aligned}
              \Delta u &= 0 &&\text{in } \Omega \cap B_1, \\
              u &= 0 &&\text{on } \mathcal{S} \cap B_1, \\
              u &\geq -\mu_0 &&\text{in } B_1, \\
              u\left(x_0\right) &= 1 &&\text{for some } x_0 \in K,
            \end{aligned}
        \right.
    \end{equation*}
    then $u \geq 0$ in $B_{1/2}$.
\end{corollary}
\begin{proof}
    Let $\delta > 0$ be small from \cref{lem:amost_pos_int_cond}, and define $\Omega_\delta \coloneqq \{x \in \Omega : \dist(x,\S) > \delta\}$. Observe that $x_0 \in \Omega_\delta \cap B_{1-\delta}$ if $\delta$ is small, depending on $K$ (in fact, it is enough to have $\delta < \min \{\dist(K,\S), \dist(K,\partial B_1)\}$). Thus, by the interior Harnack applied to $v \coloneqq u + \mu_0$, which is harmonic in $\Omega \cap B_1$ and non-negative in $B_1$ by hypothesis, we have
    \begin{equation*}
        v \geq \frac{1}{C}v(x_0) = \frac{1}{C}(u(x_0) + \mu_0) = \frac{1}{C}(1 + \mu_0) \quad\text{in } \Omega_\delta \cap B_{1-\delta},
    \end{equation*}
    where the constant $C$ depends only on $n$ and $\delta$. Here we are making use of the good connectivity properties of $\Omega_\delta \cap B_{1-\delta}$ for $\delta > 0$ small\footnote{
        The good connectivity properties of $\Omega_\delta \cap B_{1-\delta}$, for small $\delta$, stem from the fact that the boundary of $\Omega$ (and hence $\Omega_\delta$) is Lipschitz with small constant, and does not disconnect $B_1$ (because the origin belongs to the edge, and the slit $\S = \partial \Omega$ does not continue for positive values of $x_{n-1}$). More rigorously, we may use \cite[Lemma 3.6]{ClaraBoundaryHarnack}. Indeed, given any point $x \in \Omega_\delta \cap B_{1-\delta}$, \cite[Lemma 3.6]{ClaraBoundaryHarnack} implies that we can compare $v(x)$ with either $v(\frac{e_n}{2})$, $v(-\frac{e_n}{2})$ or $v(\frac{e_{n-1}}{2})$ through a Harnack chain whose constant depends only on $n$ and $\delta$. Since the values $v(\frac{e_n}{2})$, $v(-\frac{e_n}{2})$ and $v(\frac{e_{n-1}}{2})$ are also all comparable, it follows that $v(x) \leq Cv(y)$ for every $x,y \in \Omega_\delta \cap B_{1-\delta}$ and for some constant $C$ depending only on $n$ and $\delta$.    
    }.
    
    In particular, taking $\mu_0 \in (0, 1)$ sufficiently small depending on $C$ we have
    \begin{equation*}
        v 
        \geq 
        \frac{1}{C}
        \geq 
        2 \sqrt{\mu_0}
        \geq 
        \sqrt{\mu_0} + \mu_0 \quad\text{in } \Omega_{\delta} \cap B_{1-\delta}.
    \end{equation*}
    Thus, recalling also the hypotheses of this result, we see that $\widetilde{u} \coloneqq \mu_0^{-1/2}u$ is a harmonic function in $\Omega \cap B_1$ which vanishes on $\S \cap B_1$ and satisfies
    \begin{equation*}
        \left\{
            \begin{aligned}
              \widetilde{u} &\geq 1 &&\text{in } \Omega_{\delta} \cap B_{1-\delta}, \\
              \widetilde{u} &\geq -\sqrt{\mu_0} &&\text{in } B_1.
            \end{aligned}
        \right.
    \end{equation*}
    Consequently, choosing $\mu_0$ appropriately small, we can apply \cref{lem:amost_pos_int_cond} (along with a simple covering argument) to conclude that $u \geq 0$ in $B_{1/2}$.
\end{proof}


\subsection{Comparison of barriers at different scales} \label{sec:comparison_barriers}

On top of the almost positivity property, we will also need the following auxiliary results concerning the comparison of the barriers defined in \cref{cor:equiv_prop3.3_from_clara}.

\begin{lemma} \label{lem:v-w_L_infty_bound}
    Let $\alpha_0 < \pi$ and $\beta_0 < \frac{\pi}{2}$ be fixed constants and let $\alpha$ and $ \beta$ be angles such that $|\alpha| \leq \alpha_0$ and $|\beta| \leq \beta_0$. Let $\Omega$ be a Lipschitz slit domain as in \cref{def:lip_slit_domain} with Lipschitz constant $L \coloneqq L(\alpha_0, \beta_0)$ small enough, depending on $\alpha_0$ and $\beta_0$. Define $h$ as in \cref{th:barriers} and let $e = e(\alpha, \beta) \coloneqq (e'' \sin \beta, \cos{\alpha} \cos{\beta}, \sin{\alpha} \cos{\beta})$ be a unit vector (with $e'' \in \R^{n-2}$ being also a unit vector). For any $k \geq 0$, let
    \begin{equation*}
        \varphi_k \coloneqq \varphi_{2^{-k}} 
        \quad \text{and} \quad 
        \eps_k \coloneqq C(\norm{\nabla\gamma}_{L^\infty(B''_{10\cdot2^{-k}})} + \norm{\nabla\Gamma}_{L^\infty(B'_{10\cdot2^{-k}})})
    \end{equation*}
    be defined as in \cref{cor:equiv_prop3.3_from_clara}, i.e. $\varphi_k$ are solutions to
    \begin{equation*}
        \left\{
        \begin{aligned}
            \Delta \varphi_k  &= 0 &&\text{in } \Omega \cap B_{2^{-k}}, \\
            \varphi_k  &= 0 &&\text{on } \partial\Omega \cap B_{2^{-k}}.
        \end{aligned}
        \right.
    \end{equation*}
    For $k \geq 1$, define $v_k$ and $w_k$ as
    \begin{equation*}
        v_k(x) \coloneqq \frac{\varphi_{k-1} (2^{-k}x)}{\varphi_{k-1}(2^{-k-1}e)} \quad \text{and} \quad
        w_k(x) \coloneqq \frac{\varphi_k(2^{-k}x)}{\varphi_k(2^{-k-1}e)},
    \end{equation*}
    with $x \in \widetilde{\Omega}$ and $\widetilde{\Omega} \coloneqq 2^k\Omega = \{2^k x : x \in \Omega\}$. Then, if $\eps_{k-1}$ is sufficiently small, it holds
    \begin{equation*}
        \norm{v_k - w_k}_{L^\infty(\widetilde{\Omega} \cap B_1)} \leq M \eps_{k-1},
    \end{equation*}
    where $M$ is a constant depending only on $n$, $\alpha_0$ and $\beta_0$.
\end{lemma}
\begin{proof}
    From \cref{cor:equiv_prop3.3_from_clara} and the fact that the sequence $\eps_k$ is monotone we have the following estimates:
    \begin{align*}
        \norm{\varphi_k - h}_{L^\infty(B_{2^{-k}})} &\leq 15 \cdot 2^{-\frac{k}{2}} \eps_k \leq 15 \cdot 2^{-\frac{k-1}{2}} \eps_{k-1},
        \\
        \norm{\varphi_{k-1} - h}_{L^\infty(B_{2^{-k}})} &\leq \norm{\varphi_{k-1} - h}_{L^\infty(B_{2^{-k+1}})} \leq 15 \cdot 2^{-\frac{k-1}{2}} \eps_{k-1}.    \end{align*}
    Using the above estimates we get the bound
    \begin{align*}
         \norm{v_k - w_k}_{L^\infty(\widetilde{\Omega} \cap B_1)}
         &\leq
         \abs{\frac{1}{\varphi_{k-1}(2^{-k-1}e)} - \frac{1}{\varphi_k(2^{-k-1}e)}} \norm{\varphi_{k-1}}_{L^\infty(B_{2^{-k}})}
         \\ &\quad+
         \frac{1}{\varphi_k(2^{-k-1}e)} \norm{\varphi_{k-1} - \varphi_k}_{L^\infty(B_{2^{-k}})}
         \\ &\leq
         \frac{\norm{\varphi_{k-1} - \varphi_k}_{L^\infty(B_{2^{-k}})}}{\varphi_{k-1}(2^{-k-1}e) \varphi_k(2^{-k-1}e)} \norm{\varphi_{k-1}}_{L^\infty(B_{2^{-k}})}
         \\ &\quad+
         \frac{1}{\varphi_k(2^{-k-1}e)} \norm{\varphi_{k-1} - \varphi_k}_{L^\infty(B_{2^{-k}})} \\ &\leq
         \frac{30 \cdot 2^{-\frac{k-1}{2}} \eps_{k-1}}{ \left( h(2^{-k-1}e) - 15 \cdot 2^{-\frac{k-1}{2}} \eps_{k-1} \right)^2} \left(\norm{h}_{L^\infty(B_{2^{-k}})} +  15 \cdot 2^{-\frac{k-1}{2}} \eps_{k-1} \right)
         \\ &\quad+
         \frac{30 \cdot 2^{-\frac{k-1}{2}} \eps_{k-1}}{h(2^{-k-1}e) - 15 \cdot 2^{-\frac{k-1}{2}} \eps_{k-1}},
    \end{align*}
    where in the first step we have used the elementary estimate $|\frac{a}{b} - \frac{c}{d}| \leq |\frac{1}{b} - \frac{1}{d}|\cdot|a| + \frac{1}{|d|}\cdot|a - c|$. Now, using that $\kappa \sqrt{r} \leq h(re)$ and $\norm{h}_{L^\infty(B_r)} \leq 3\sqrt{r}$ from \cref{lem:h_comparable_sqrt_r}, we continue the estimate by
    \begin{align*}
        \norm{v_k - w_k}_{L^\infty(\widetilde{\Omega} \cap B_1)} 
        & \leq
        \frac{30 \cdot 2^{-\frac{k-1}{2}} \eps_{k-1}}{ \left( \kappa 2^{-\frac{k+1}{2}} - 15 \cdot 2^{-\frac{k-1}{2}} \eps_{k-1} \right)^2} \left(3 \cdot 2^{-\frac{k}{2}} + 15 \cdot 2^{-\frac{k-1}{2}} \eps_{k-1} \right)
        \\ &\quad+
        \frac{30 \cdot 2^{-\frac{k-1}{2}} \eps_{k-1}}{\kappa 2^{-\frac{k+1}{2}} - 15 \cdot 2^{-\frac{k-1}{2}} \eps_{k-1}}
        \\ & \leq 
        \frac{C \cdot 2^{-k} \varepsilon_{k-1}}{\left(\frac{\kappa}{10} 2^{-\frac{k}{2}}\right)^2}
        + \frac{C \cdot 2^{-\frac{k}{2}} \varepsilon_{k-1}}{\frac{\kappa}{10} 2^{-\frac{k}{2}}} 
        \\ & \leq 
        C(\kappa) \, \varepsilon_{k-1},
    \end{align*}
    where we have used that by definition of $\varepsilon_{k-1}$, it holds $\varepsilon_{k-1} \leq CL \ll \kappa$ if $L$ is sufficiently small, which has allowed us to estimate the denominators.
\end{proof} 

\begin{lemma} \label{lem:varphi_k-1_comp_varphi_k}
    Under the conditions of \cref{lem:v-w_L_infty_bound}, it holds 
    \begin{equation*}
        (1-A\varepsilon_{k-1}) \, \varphi_k \leq \varphi_{k-1} \leq (1 + A\eps_{k-1}) \, \varphi_k
        \quad \text{in } B_{2^{-k-1}},
    \end{equation*}
    where $A > 0$ depends only on $n, \alpha_0, \beta_0$.
\end{lemma}
\begin{proof}  
    Let us first prove the lower bound. Consider, for $\mu_0$ small from \cref{cor:almost_pos} (which is independent of $k$),
    \begin{equation*}
        \psi_k 
        \coloneqq 
        \frac{v_k - (1 - \mu_0^{-1}M\eps_{k-1}) \, w_k}{\mu_0^{-1}M\eps_{k-1}},
    \end{equation*}
    where $v_k, w_k, M$ are as in \cref{lem:v-w_L_infty_bound}. If we set $\widetilde{\Omega} \coloneqq 2^k\Omega = \{2^k x : x \in \Omega\}$, then as a consequence of \cref{lem:v-w_L_infty_bound} it holds $v_k - w_k \geq -M\eps_{k-1}$ in $\widetilde{\Omega} \cap B_1$, and therefore
    \begin{align*}
        \psi_k &= \frac{\mu_0^{-1}M\eps_{k-1}w_k + v_k - w_k}{\mu_0^{-1}M\eps_{k-1}}
        \geq \frac{\mu_0^{-1}M\eps_{k-1}w_k - M\eps_{k-1}}{\mu_0^{-1}M\eps_{k-1}}
        = w_k - \mu_0
        \geq -\mu_0
        \quad\text{in } \widetilde{\Omega} \cap B_1,
    \end{align*}
    where we have used the fact that $w_k$ is non-negative for the last inequality (because so are all the $\varphi_k$, see \cref{cor:equiv_prop3.3_from_clara,lem:h_comparable_sqrt_r}). In addition, the harmonicity of $\varphi_{k-1}$ and $\varphi_k$ in their respective domains implies that both $v_k$ and $w_k$ are harmonic in the rescaled domain $\widetilde{\Omega} \cap B_1$. Hence, we see that the functions $\psi_k$ satisfy
    \begin{equation*}
        \left\{
            \begin{aligned}
            \Delta \psi_k &= 0 &&\text{in } \widetilde{\Omega} \cap B_1, \\
            \psi_k &\geq -\mu_0 &&\text{in } \widetilde{\Omega} \cap B_1, \\
            \psi_k &= 0 &&\text{on } \partial \widetilde{\Omega} \cap B_1, \\
            \psi_k\left(\textstyle\frac{e}{2}\right) &= 1,
            \end{aligned}
        \right.
    \end{equation*}
    for every $k$, where the last two conditions are a consequence of the definition of $v_k$ and $w_k$. Consequently, we may apply \cref{cor:almost_pos} with $K \coloneqq \{\frac{1}{2}e(\alpha,\beta) : \abs{\alpha} \leq \alpha_0,\ \abs{\beta} \leq \beta_0\}$, $e(\alpha, \beta)$ as in \cref{lem:v-w_L_infty_bound} and $x_0 \coloneqq \frac{e}{2}$ to deduce that $\psi_k \geq 0$ in $\widetilde{\Omega} \cap B_{1/2}$ for $\mu_0$ sufficiently small depending only on $n$, $\alpha_0$ and $\beta_0$. Equivalently, 
    \begin{equation*}
        v_k \geq (1 - \mu_0^{-1}M\eps_{k-1})w_k 
        \quad 
        \text{in $\widetilde{\Omega} \cap B_{1/2}$ }
    \end{equation*}
    and therefore, by the definitions of $v_k, w_k$ (see \cref{{lem:v-w_L_infty_bound}}),
    \begin{equation} \label{eq:varphi_k-1_geq}
        \varphi_{k-1} 
        \geq
        (1 - \mu_0^{-1}M\eps_{k-1})\frac{\varphi_{k-1}(2^{-k-1}e)}{\varphi_k(2^{-k-1}e)} \varphi_k
        \quad \text{in } \Omega \cap B_{2^{-k-1}}.
    \end{equation}
    
    On the other hand, \cref{cor:equiv_prop3.3_from_clara,lem:h_comparable_sqrt_r} imply
    \begin{equation}\label{eq:varphi_quotient_lower_bound}
        \frac{\varphi_{k-1}(2^{-k-1}e)}{\varphi_k(2^{-k-1}e)}
        \geq 
        \frac{h(2^{-k-1}e) - 15 \cdot 2^{-\frac{k-1}{2}}\eps_{k-1}}{h(2^{-k-1}e) + 15 \cdot 2^{-\frac{k-1}{2}}\eps_{k-1}}
        \geq 
        1 - \frac{30 \eps_{k-1}}{\frac{\kappa}{2} + 15\eps_{k-1}}
        \geq 
        1 - C \eps_{k-1},
    \end{equation}
    where $C = C(\kappa) = C(\alpha_0, \beta_0) > 0$ by \cref{lem:h_comparable_sqrt_r}. Thus, combining \eqref{eq:varphi_k-1_geq} with \eqref{eq:varphi_quotient_lower_bound} yields
    \begin{equation*}
        \varphi_{k-1} 
        \geq 
        (1 - \mu_0^{-1}M\eps_{k-1}) (1 - C \eps_{k-1}) \, \varphi_k 
        \geq 
        (1 - A\eps_{k-1})\, \varphi_k
        \quad\text{in } \Omega \cap B_{2^{-k-1}},
    \end{equation*}
    if $A > 0$ is chosen large enough.

    The proof of the upper bound follows by a completely analogous argument, but this time using the functions
    \begin{equation*}
        \psi_k \coloneqq \frac{(1 + \mu_0^{-1}M \eps_{k-1})w_k - v_k}{\mu_0^{-1}M \eps_{k-1}}.
    \end{equation*}
\end{proof}

\begin{remark}\label{rmk:const_A_depend}
    If $e$ is replaced with $e_{n-1}$ in the proof of \cref{lem:varphi_k-1_comp_varphi_k}, then the constant $A$ becomes independent of $\alpha_0$ and $\beta_0$. Indeed, using $e_{n-1}$ instead of $e$ corresponds to taking $\alpha_0 = \beta_0 = 0$ and $e = e(0,0) = e_{n-1}$ (see \cref{lem:v-w_L_infty_bound} for the definition of $e$) which, in turn, makes the constants $\kappa$, $M$ and $\mu_0$ from \cref{lem:v-w_L_infty_bound,cor:almost_pos,lem:h_comparable_sqrt_r} all independent of $\alpha_0$ and $\beta_0$.
\end{remark}


\subsection{Proof of \cref{th:lower_bound_thm}}\label{subsec:lower_bound_thm_proof}

Fix $\omega_0 > 0$ small enough to be determined later, depending only on $n$, $\alpha_0$ and $\beta_0$. By definition of modulus of continuity, we can find $j_0 \in \NN$ such that $\omega(s) < \omega_0$ for all $s \in [0, 2^{-j_0+5}]$. We will show that \cref{th:lower_bound_thm} holds with $r_0 \coloneqq 2^{-j_0}$. For that purpose, fix $0 < \rho < r < r_0$, and let $j_1, j_2 \in \NN$ be indices such that 
\begin{equation*}
    r \in (2^{-j_1-1}, 2^{-j_1}]
    \quad \text{and} \quad
    \rho \in (2^{-j_2-1}, 2^{-j_2}],
\end{equation*}
so that it holds $j_0 \leq j_1 \leq j_2$. Observe that if $j_1 = j_2$, then the scales $\rho$ and $r$ are comparable (specifically $ r/2 < \rho < r$) and the theorem is a consequence of the interior Harnack inequality. Thus, throughout the rest of the proof we may assume $j_1 < j_2$.

For $k \geq j_0$, consider the sequences $\varphi_k$ and $\varepsilon_k$ as defined in \cref{lem:v-w_L_infty_bound}. We apply the boundary Harnack inequality from \cref{th:bdry_harnack} to $u$ and $\varphi_{j_1}$ (rescaled to $B_{2^{-j_1}}$)\footnote{
    More specifically, we apply \cref{th:bdry_harnack} to $\widehat{u} \coloneqq u(2^{-j_1}\cdot)$ and $\widehat{\varphi}_{j_1} \coloneqq \varphi_{j_1}(2^{-j_1}\cdot) / \varphi_{j_1}(2^{-j_1-1}e_n)$, where we note that the normalization clearly yields $\widehat{\varphi}_{j_1}(e_n/2)=1$, so \cref{rmk:bdry_harnack_one_side_condition} allows to apply \cref{th:bdry_harnack}.
}
\begin{equation}
    \label{eq:bdry_harnack_at_2-j1-1}
    \frac{u}{\varphi_{j_1}} \geq \frac{1}{C}\, \frac{\min\{u(2^{-j_1-1}e_n),\, u(-2^{-j_1-1}e_n)\}}{\varphi_{j_1}(2^{-j_1-1}e_n)}
    \quad\text{in } \Omega \cap B_{2^{-j_1-1}}.
\end{equation}
Moreover, as a consequence of our choice of $e$ in \cref{th:lower_bound_thm}, it holds (see \cref{lem:spherical_distance_boundary}) $\dist(re, \S) \geq c_0 r$ for some $c_0 \coloneqq c_0(n, \alpha_0, \beta_0) > 0$ whenever $\omega_0$ is small enough. Then, it is easy to connect $re$ to $\pm 2^{-j_1-1}e_n$ by a Harnack chain of overlapping balls inside $\Omega \cap B_{2^{-j_1}}$, where the number of balls depends only on $n$, $\alpha_0$ and $\beta_0$,\footnote{
    Indeed, to connect these points, one may forget about the slit domain $\Omega$, and just move inside the conical regions defined in \cref{lem:spherical_distance_boundary}, which are clearly Harnack chain connected.
} so that by repeated application of interior Harnack we get
\begin{equation} \label{eq:harnack_chain_bound}
    \min\left\{u(2^{-j_1-1}e_n), u(-2^{-j_1-1}e_n)\right\} \geq C^{-1} u(re).
\end{equation}
Furthermore, by \cref{cor:equiv_prop3.3_from_clara,lem:h_comparable_sqrt_r}, we have $\varphi_{j_1}(2^{-j_1-1}e_n) \leq 3\sqrt{2^{-j_1}} \leq 3\sqrt{2r}$, which combined with \eqref{eq:bdry_harnack_at_2-j1-1} and \eqref{eq:harnack_chain_bound} yields
\begin{equation}
    \label{eq:lower_bound_initial}
    u 
    \geq 
    \frac{1}{C} \frac{u(re)}{\sqrt{r}} \varphi_{j_1} 
    =: 
    c_{j_1} \varphi_{j_1} 
    \quad\text{in } B_{2^{-j_1-1}},
\end{equation}
where $C = C(n,\alpha_0,\beta_0)$.

Then, using the lower bound of \cref{lem:varphi_k-1_comp_varphi_k} iteratively, we obtain that, for $k \geq j_1$, it holds
\begin{equation} \label{eq:lower_bound_induction}
    u \geq c_k \varphi_k \quad \text{in } B_{2^{-k-1}},
\end{equation}
where $c_k = (1 - A \eps_{k-1}) \, c_{k-1}$ and $A = A(n, \alpha_0,\beta_0)$ is a large constant.

Now, by choosing $\omega_0$ small enough, we can ensure that, for $j \geq j_0$, it holds
\begin{align*}
    A \varepsilon_{j-1}
    &\leq 
    A C (\norm{\nabla\gamma}_{L^{\infty}(B''_{10\cdot2^{-j+1}})} + \norm{\nabla\Gamma}_{L^{\infty}(B'_{10\cdot2^{-j+1}})})
    \leq 
    A C \omega( 10\cdot2^{-j+1} )
    \\ &\leq 
    A C \omega (10 \cdot 2^{-j_0+1})
    \leq 
    A C \omega (2^{-j_0+5})
    \leq 
    A C \omega_0
    \leq 
    \frac12,
\end{align*}
where we have used \cref{def:C1_slit_domain} and the definition of $\varepsilon_k$ from \cref{lem:v-w_L_infty_bound}. 
This bound implies that the elementary inequality $1 - t \geq 4^{-t}$ for $t \in [0, 1/2]$ can be used for $A\varepsilon_{j-1}$ to obtain
\begin{equation*}
    c_{j_2-1}
    = 
    c_{j_1} \prod_{j = j_1}^{j_2-2} (1 - A \eps_j) 
    \geq 
    c_{j_1} 4^{-A \sum_{j = j_1}^{j_2-2} \eps_j}
    \geq 
    c_{j_1} \exp\left(-C \sum_{j = j_1}^{j_2-2} \eps_j \right).
\end{equation*}
Thus, recalling the choice of $j_2$, we get, using also \eqref{eq:lower_bound_induction}, \cref{cor:equiv_prop3.3_from_clara} and \cref{lem:h_comparable_sqrt_r}, 
\begin{align*}
    u(\rho e) 
    &\geq c_{j_2-1} \varphi_{j_2-1}(\rho e)
    \geq c_{j_2-1}\left(3\cdot2^{-\frac{j_2-1}{2}}\right)^{-\eps_{j_2-1}}(\kappa\sqrt{\rho})^{1+\eps_{j_2-1}} \\
    &\geq c_{j_2-1} \left(6 \sqrt{\rho} \right)^{-\eps_{j_2-1}} (\kappa \sqrt{\rho})^{1+\eps_{j_2-1}}
    \geq c_{j_1} \exp\left(-C \sum_{j = j_1}^{j_2-2} \eps_j \right) \frac{\kappa^2}{6} \sqrt{\rho}.
\end{align*}
Taking into account that we also showed \eqref{eq:lower_bound_initial}, we are left to bound from above that sum of $\varepsilon_j$. For that, recalling \cref{def:C1_slit_domain} and the definition of $\varepsilon_k$ from \cref{lem:v-w_L_infty_bound}, and using that $\omega$ is monotone by definition of modulus of continuity, we estimate 
\begin{align*}
    \sum_{j=j_1}^{j_2-2} \varepsilon_j
    & \leq 
    C \sum_{j=j_1}^{j_2-2} \omega(10 \cdot 2^{-j})
    \leq
    C \sum_{j=j_1-4}^{j_2-6} \omega(2^{-j}) 
    \\ & \leq 
    \frac{C}{\log 2} \sum_{j=j_1-4}^{j_2-6} \int_{2^{-j}}^{2^{-j + 1}}  \omega(s) \frac{ds}{s} 
    \leq 
    \frac{C}{\log 2} \int_{2^{-j_2+6}}^{2^{-j_1+5}} \omega(s) \frac{ds}{s} 
    \leq 
    \frac{C}{\log 2} \int_{64\rho}^{64r} \omega(s) \frac{ds}{s},
\end{align*}
which is the estimate that we sought in \cref{th:lower_bound_thm}.


\subsection{Proof of \cref{th:upper_bound_thm}}\label{subsec:upper_bound_thm_proof}

Let $\omega_0$ be a small constant to be determined later. By definition of modulus of continuity, we can find an index $j_0 \in \mathbb{N}$ such that $\omega(s) < \omega_0$ for all $s \in [0, 2^{-2j_0+6}]$. We will show that \cref{th:upper_bound_thm} holds with $r_0 \coloneqq 2^{-2j_0}$. As in the proof of \cref{th:lower_bound_thm}, fix $0 < \rho < r < r_0$ and let $j_1, j_2 \in \mathbb{N}$ be indices such that
\begin{equation*}
    r \in (2^{-2j_1 -2}, 2^{-2j_1}] \quad \text{and} \quad \rho \in (2^{-2j_2 - 2}, 2^{-2j_2}],
\end{equation*}
so that it holds $j_0 \leq j_1 \leq j_2$. In a similar way to the proof of \cref{th:lower_bound_thm}, whenever $j_2 \leq j_1 + 1$, we have $r/16 \leq \rho \leq r$ and hence, with $v \coloneqq u - g(0)$ as in the statement,
\begin{equation*}
    \frac{\|v\|_{L^\infty(B_\rho)}}{\sqrt{\rho}}
    \leq \sqrt{\frac{r}{\rho}}\,\frac{\|v\|_{L^\infty(B_r)}}{\sqrt{r}}
    \leq 4\,\frac{\|v\|_{L^\infty(B_r)}}{\sqrt{r}}.
\end{equation*}
Therefore, we may assume $j_2 \geq j_1+2$ for the rest of the proof.

Replacing $u$ with the function $v \coloneqq u - g(0)$, which satisfies that $-\Delta v = f$ and $v = \tilde{g}$ on $\partial\Omega \cap B_1$ with $\tilde{g}(0) = 0$, we may assume without loss of generality that $g(0) = 0$.

For $k \geq j_0$, let $\varphi_k \coloneqq \varphi_{2^{-2k}}$ and $\eps_k \coloneqq C(\norm{\nabla\gamma}_{L^{\infty}(B''_{10\cdot2^{-2k}})} + \norm{\nabla\Gamma}_{L^{\infty}(B'_{10\cdot2^{-2k}})})$ be defined as in \cref{cor:equiv_prop3.3_from_clara}. First, our aim is to prove by induction the following claim for all $k$ from $j_1 + 1$ to $j_2$:
\begin{equation*}
    |u| \leq c_k \varphi_k + 2^{-k} d_k \quad \text{in } \Omega \cap B_{2^{-2k-1}},
\end{equation*}
where
\begin{equation*}
    c_k \coloneqq (1 + A \eps_{k-1})c_{k-1} + Ad_{k-1},
    \qquad
    d_k \coloneqq C(n,p)\left(\omega_g(2^{-2k+1}) + 2^{-k(3 - \frac{2n}{p})}\|f\|_{L^p(\Omega \cap B_{2^{-2k+1}})}\right),
\end{equation*}
and $A$ is a large constant depending only on $n$ that will be fixed later.

\begin{enumerate}[wide,label=\textbf{Step \arabic*.},ref=Step \arabic*]
    \item\label{it:upper_bound_th_step1} Let us start with the base case of the induction. Let $\tilde{u}_{j_1 + 1}$ be the solution to
    \begin{equation*}
        \left\{
            \begin{aligned}
            \Delta \tilde{u}_{j_1 + 1} &= 0    &&\text{in } \Omega \cap B_{2^{-2j_1-2}}, \\
            \tilde{u}_{j_1 + 1}        &= |u|  &&\text{on } \Omega \cap \partial B_{2^{-2j_1-2}}, \\
            \tilde{u}_{j_1 +1} &= 0            &&\text{on } \partial \Omega \cap B_{2^{-2j_1-2}}.
            \end{aligned}
        \right.
    \end{equation*}
    Then, by the boundary Harnack inequality (\cref{th:bdry_harnack}), and arguing as in \cref{th:lower_bound_thm}, 
    \begin{equation}\label{eq:upper_bound_th_base_case}
        \tilde{u}_{j_1 +1} \leq c_{j_1+1} \varphi_{j_1+1} \quad\text{in } \Omega \cap B_{2^{-2j_1-3}},
    \end{equation}
    where
    \begin{equation*}
        c_{j_1+1} \coloneqq C' \frac{\norm{\tilde{u}_{j_1 + 1}}_{L^\infty(\Omega \cap B_{2^{-2j_1 -2}})}}{\sqrt{2^{-2j_1 -2}}}\leq C\frac{\norm{u}_{L^\infty(B_r)}}{\sqrt{r}}.
    \end{equation*}
    Now, consider the difference $u - \tilde{u}_{j_1+1}$ and note that
    \begin{equation*}
        \left\{
            \begin{aligned}
            -\Delta(u - \tilde{u}_{j_1+1}) &= f \leq |f| &&\text{in } \Omega \cap B_{2^{-2j_1-2}}, \\
            u - \tilde{u}_{j_1+1} &= u -|u| \leq 0        &&\text{on } \Omega \cap \partial B_{2^{-2j_1-2}}, \\
            u - \tilde{u}_{j_1+1} &= g                    &&\text{on } \partial \Omega \cap B_{2^{-2j_1-2}},
            \end{aligned}
        \right.
    \end{equation*}
    and hence, we can apply \cref{th:GilbargTrudinger_8.16} to the function $u - \tilde{u}_{j_1+1}$ to obtain
    \begin{align*}
        u - \tilde{u}_{j_1+1} &\leq \|g\|_{L^\infty(\partial\Omega \cap B_{2^{-2j_1-2}})} + C (2^{-2j_1-2})^{2 - \frac{n}{p}}\|f\|_{L^p(\Omega \cap {B_{2^{-2j_1-2}}})} \\
                              &\leq 2^{-j_1-1}\omega_g(2^{-2j_1-2}) + C 2^{-j_1-1}(2^{-j_1-1})^{3-\frac{2n}{p}}\|f\|_{L^p(\Omega \cap {B_{2^{-2j_1-2}}})} \\
                              &\leq 2^{-j_1-1}d_{j_1+1}.
    \end{align*}
    where we have used that $g(0) = 0$ to bound $\|g\|_{L^\infty(\partial\Omega \cap B_{2^{-2j_1-2}})} \leq 2^{-j_1-1}\omega_g(2^{-2j_1-2})$.
    
    Repeating the above argument with the function $-u-\tilde{u}_{j_1+1}$ (which in this case satisfies $-\Delta(-u - \tilde{u}_{j_1+1}) = -f \leq |f|$) we get
    \begin{equation*}
        -u - \tilde{u}_{j_1+1} \leq  \|g\|_{L^\infty(\partial\Omega \cap B_{2^{-2j_1-2}})} + C (2^{-2j_1-2})^{2-\frac{n}{p}} \|f\|_{L^p(\Omega \cap {B_{2^{-2j_1-2}}})} \leq 2^{-j_1-1}d_{j_1 + 1},
    \end{equation*}
    and combining both estimates with \eqref{eq:upper_bound_th_base_case} we conclude
    \begin{equation*}
        |u| \leq c_{j_1+1} \varphi_{j_1+1} +  2^{-j_1-1}d_{j_1+1}
        \quad\text{in } \Omega\cap B_{2^{-2j_1 -3}}.
    \end{equation*}
    \item\label{it:upper_bound_th_step2} Now we proceed with the induction step. Assume the claim holds at step $k-1$, i.e.,
    \begin{equation}\label{eq:upper_bound_thm_ind_hyp}
        |u| \leq c_{k-1}\varphi_{k-1} + 2^{-k+1}d_{k-1} \quad \text{in } \Omega \cap B_{2^{-2k+1}},
    \end{equation}
    and let $\tilde{u}_k$ be the solution to
    \begin{equation*}
        \left\{
            \begin{aligned}
            \Delta\tilde{u}_k &= 0 &&\text{in } \Omega \cap B_{2^{-2k+1}}, \\
            \tilde{u}_k &= |u|     &&\text{on } \Omega \cap \partial B_{2^{-2k+1}}, \\
            \tilde{u}_k &= 0       &&\text{on } \partial \Omega \cap B_{2^{-2k+1}}.
            \end{aligned}
        \right.
    \end{equation*}
    Then, by the comparison principle $\tilde{u}_k \leq c_{k-1}\varphi_{k-1} + 2^{-k+1}d_{k-1}$ in $\Omega \cap B_{2^{-2k+1}}$. Hence, the boundary Harnack inequality \cref{th:bdry_harnack} applied to $\tilde{u}_k - c_{k-1}\varphi_{k-1}$ and $\varphi_{k-1}$ yields, keeping \cref{rmk:bdry_harnack_neg} in mind,
    \begin{equation*}
        \frac{\tilde{u}_k - c_{k-1}\varphi_{k-1}}{\varphi_{k-1}}
        \leq C\frac{\|(\tilde{u}_k - c_{k-1}\varphi_{k-1})_+\|_{L^\infty(\Omega \cap B_{2^{-2k+1}})}}{\varphi_{k-1}(2^{-2k}e_{n-1})}
        \leq C\frac{2^{-k+1}d_{k-1}}{\varphi_{k-1}(2^{-2k}e_{n-1})},
        \quad \text{in } \Omega \cap B_{2^{-2k}},
    \end{equation*}
    where the constant $C$ depends only on the dimension $n$. Further, since
    \begin{equation}\label{eq:eps_modulus_upper_bound}
        \eps_{k-1} \leq C\omega(10\cdot2^{-2k+2}) \leq C\omega(2^{-2j_0+6}) \leq C\omega_0
    \end{equation}
    by definition of $\eps_{k-1}$ and \cref{def:C1_slit_domain}, choosing $\omega_0 \leq (2C)^{-1}$ we have, by \cref{cor:equiv_prop3.3_from_clara,lem:h_comparable_sqrt_r}, that
    \begin{equation*}
    \varphi_{k-1}(2^{-2k}e_{n-1}) \geq (3\cdot2^{-k+1})^{-\eps_{k-1}} h(2^{-2k}e_{n-1})^{1+\eps_{k-1}} \geq \frac{\kappa^2}{6}\, 2^{-k},
    \end{equation*}  
    and therefore,
    \begin{equation*}
        \frac{\tilde{u}_k - c_{k-1}\varphi_{k-1}}{\varphi_{k-1}}
        \leq C\frac{2^{-k+1}d_{k-1}}{\varphi_{k-1}(2^{-2k}e_{n-1})}
        \leq Cd_{k-1},
        \quad\text{in } \Omega \cap B_{2^{-2k}}.
    \end{equation*}
    Consequently, the upper bound of \cref{lem:varphi_k-1_comp_varphi_k} (applied twice, with indices $2k$ instead of $k$ and denoting its constant by $A_0$) implies that, in $\Omega \cap B_{2^{-2k-1}}$,
    \begin{align*}
        \tilde{u}_k
        &\leq (c_{k-1} + C d_{k-1})\varphi_{k-1} \\
        &\leq (c_{k-1} + C d_{k-1})(1 + A_0\eps_{k-1})^2\varphi_k \\
        &\leq \left( (1 + A\eps_{k-1})c_{k-1} + Ad_{k-1}\right) \varphi_k,
    \end{align*}    
    provided that $\eps_{k-1}$ is small enough, which, recalling \eqref{eq:eps_modulus_upper_bound}, can be achieved by choosing $\omega_0$ sufficiently small. Also, the constant $A$ depends only on $n$ since we are using the direction $e_{n-1}$ (see \cref{rmk:const_A_depend}). Lastly, we use once again the same comparison argument we used for $\tilde{u}_{j_1+1}$ in \cref{it:upper_bound_th_step1} to obtain that
    \begin{equation*}
        \pm u - \tilde{u}_k
        \leq 2^{-k}\sqrt{2}\,\omega_g(2^{-2k+1}) + C(2^{-2k+1})^{2-\frac{n}{p}}\|f\|_{L^p(\Omega \cap B_{2^{-2k+1}})}
        \leq 2^{-k}d_k
        \quad\text{in } \Omega \cap B_{2^{-2k+1}}
    \end{equation*}
    and conclude
    \begin{equation*}
        |u| \leq c_k \varphi_k + 2^{-k}d_k
        \quad\text{in } \Omega \cap B_{2^{-2k-1}},
    \end{equation*}
    thus finishing the induction step.
    
    \item The induction argument shows, in particular, that at the target scale $j_2-1$,
    \begin{equation}
        \label{eq:target_scale_inequality}
        |u| \leq c_{j_2-1} \varphi_{j_2-1} + 2^{-j_2+1}d_{j_2-1} \quad \text{in } \Omega \cap B_{\rho}.
    \end{equation}
    Unrolling the recurrence for $c_{j_2-1}$ and using the elementary inequality $1+t \leq e^t$ gives us
    \begin{align*}
        c_{j_2-1} &= c_{j_1+1} \prod_{j=j_1 +1}^{j_2-2} (1 + A \eps_j) + A \sum_{i=j_1 +1}^{j_2-2} \left[d_i\prod_{j=i+1}^{j_2-2}(1 + A \eps_j) \right] \\
                  &\leq \left(c_{j_1 +1} + A \sum_{i=j_1 +1}^{j_2-1} d_i\right)  e^{A \sum_{j=j_1 +1}^{j_2-2} \eps_j},
    \end{align*}
    which, inserted back in \eqref{eq:target_scale_inequality} and combined with $\varphi_{j_2-1} \leq (12\sqrt{\rho})^{\eps_{j_2-1}} h^{1-\eps_{j_2-1}} \leq 36\sqrt{\rho}$ (by \cref{cor:equiv_prop3.3_from_clara,lem:h_comparable_sqrt_r}), yields
    \begin{equation*}
        \norm{u}_{L^\infty(B_\rho)} \leq 36\sqrt{\rho}\left(c_{j_1+1} + A \sum_{i=j_1+1}^{j_2-1} d_i\right)  e^{A \sum_{j=j_1+1}^{j_2-2} \eps_j} + 2^{-j_2+1}d_{j_2-1}.   
    \end{equation*}
    Moreover, since $2^{-j_2+1}d_{j_2-1} \leq 4\sqrt{\rho} d_{j_2-1}$, the last term can be absorbed in the sum to get
    \begin{equation}
        \label{eq:u_upper_bounded_by_sum}
        \frac{\norm{u}_{L^\infty(B_\rho)}}{\sqrt{\rho}} \leq C \left(c_{j_1+1} + A \sum_{i=j_1+1}^{j_2-1} d_i\right)  e^{A \sum_{j=j_1+1}^{j_2-2} \eps_j}.
    \end{equation}
    
    Similarly as in the proof of \cref{th:lower_bound_thm}, recalling \cref{def:C1_slit_domain}, the definition of $\eps_k$ and the monotonicity of the modulus of continuity $\omega$, the sum of $\eps_j$ can be bounded as follows:
    \begin{equation}\label{eq:upper_bound_th_eps_bound}
        \sum_{j=j_1}^{j_2-2} \varepsilon_j
        \leq C \sum_{j=j_1}^{j_2-2} \omega(10\cdot2^{-2j})
        \leq \frac{C}{\log 4} \sum_{j=j_1-2}^{j_2-4} \int_{2^{-2j}}^{2^{-2j+2}} \omega(s) \frac{ds}{s}
        \leq \frac{C}{\log 4} \int_{256\rho}^{256r} \omega(s) \frac{ds}{s}.
    \end{equation}
    Also, recalling that $d_i = C(n,p)\left(\omega_g(2^{-2i+1}) + 2^{-i(3-\frac{2n}{p})}\|f\|_{L^p(\Omega \cap B_{2^{-2i+1}})}\right)$ and $p > \frac{2n}{3}$, we have, in a similar way,
    \begin{equation}\label{eq:upper_bound_th_dj_bound}
        \begin{aligned}
          \sum_{j=j_1+1}^{j_2-1} d_j &\lesssim \sum_{j=j_1+1}^{j_2-1} \omega_g(2^{-2j+1}) + \sum_{j=j_1+1}^{\infty} 2^{-j(3-\frac{2n}{p})} \|f\|_{L^p(\Omega \cap B_{2^{-2j+1}})} \\
                                  &\lesssim \sum_{j=j_1+1}^{j_2-1} \int_{2^{-2j+1}}^{2^{-2j+3}} \omega_g(s) \frac{ds}{s} + \frac{2^{-(j_1+1)(3-\frac{2n}{p})}}{1 - 2^{-(3-\frac{2n}{p})}}\|f\|_{L^p(\Omega \cap B_{2^{-2j_1-1}})} \\
                                  &\lesssim \int_{8\rho}^{8r} \omega_g(s) \frac{ds}{s} + C(n,p)r^{\frac{3}{2}-\frac{n}{p}}\|f\|_{L^p(\Omega \cap B_{2r})},
        \end{aligned}
    \end{equation}
    where the last inequality also used the fact that $2^{-2j_1} < 4r$. Substituting \eqref{eq:upper_bound_th_eps_bound} and \eqref{eq:upper_bound_th_dj_bound} back in \eqref{eq:u_upper_bounded_by_sum} we obtain
    \begin{equation*}
        \frac{\|u\|_{L^\infty(B_\rho)}}{\sqrt{\rho}}
        \leq C \left(\frac{\|u\|_{L^\infty(B_r)}}{\sqrt{r}} + r^{\frac{3}{2}-\frac{n}{p}}\|f\|_{L^p(\Omega \cap B_{2r})} + \int_{8\rho}^{8r} \omega_g(s) \frac{ds}{s} \right) \exp \left(C \int_{256\rho}^{256r} \omega(s) \frac{ds}{s} \right),
    \end{equation*}
    which completes the proof of \cref{th:upper_bound_thm}, because the normalizations we performed at the beginning of the proof make $u$ coincide with $v$ as in the statement.
\end{enumerate}


\subsection{Proof of \cref{cor:boundary_regularity}} \label{subsec:boundary_regularity_proof}

\Cref{cor:boundary_regularity} is a consequence of \cref{th:upper_bound_thm} and standard interior estimates for elliptic equations. Before we proceed with its proof, we present some auxiliary results. One of them, \cref{lem:u-g_non_edge_bound}, is stated and proved in \cref{sec:sec_corollary_appendix}.

For convenience of the reader, we recall that the objective of \cref{cor:boundary_regularity} is to show that $u \in C^{\omega_u}(\overline{\Omega} \cap B_{r_0/16})$, where
    \begin{equation*}
        \omega_u(t)
        \coloneqq
        C \sqrt{t} \left( \frac{\|u\|_{L^\infty(B_{r_0})}}{\sqrt{r_0}} + r_0^{\frac{3}{2} - \frac{n}{p}}\|f\|_{L^p(\Omega \cap B_{r_0})} + \int_t^{8r_0} \omega_g(s) \frac{ds}{s} \right) \exp \left( C \int_t^{256r_0} \omega(s) \frac{ds}{s} \right).
    \end{equation*}

\begin{remark} \label{rem:tilde_omega_sqrt_scaling}
    It is immediate to check that if $\lambda \geq 1$ and $t > 0$ satisfy $\lambda t \leq 8r_0$, then it holds $\omega_u(\lambda t) \leq \sqrt{\lambda}\,\omega_u(t)$.
\end{remark}

Let us start with a standard interior square-root-like elliptic estimate.
\begin{lemma} \label{lem:bound_poisson_eq}
    Let $n \geq 2$ and let $f \in L^p(B_r)$ with $p > \frac{2n}{3}$, for some $r > 0$. If $w$ solves
    \begin{equation*}
        \left\{
          \begin{aligned}
            -\Delta w &= f \quad&&\text{in } B_r, \\
                    w &= 0 &&\text{on } \partial B_r,
          \end{aligned}
        \right.
    \end{equation*}
    in the weak sense, then
    \begin{equation*}
        \abs{w(x) - w(0)} \leq C(n,p) \, r^{\frac32 - \frac{n}{p}} |x|^{1/2} \|f\|_{L^p(B_r)}, 
        \qquad 
        \forall \, x \in B_{r/2}.
    \end{equation*}
\end{lemma}
\begin{proof}
    Upon an elementary rescaling, it suffices to consider the case $r=1$. In such case, since $p > \frac{2n}{3}$ and $\left( \frac{2n}{3} \right)^* = 2n$, the Sobolev embeddings yield 
    \begin{equation*}
        W^{2, p}(B_1)
        \hookrightarrow
        W^{2, \frac{2n}{3}}(B_1)
        \hookrightarrow
        W^{1, 2n}(B_1)
        \hookrightarrow
        C^{0, \frac12}(B_1).
    \end{equation*}
    Combining this with the standard Calder\'on--Zygmund estimates, we obtain
    \begin{equation*}
        |w(x) - w(0)| \leq [w]_{C^{0,1/2}(B_1)} |x|^{1/2} \lesssim \|w\|_{W^{2,p}(B_1)} |x|^{1/2} \lesssim \|f\|_{L^p(B_1)} |x|^{1/2}.
    \end{equation*}
\end{proof}

Next, we focus on controlling the behavior near the boundary, i.e., the slit. We separate the case of the edge, where we can use the upper bounds from \cref{th:upper_bound_thm}, from the rest of the slit, which is handled via explicit and simpler barriers in \cref{lem:u-g_non_edge_bound}, because away from the edge our slit domain locally resembles a thick domain. 

\begin{lemma}[Control near the edge] \label{lem:u-g_edge_bound}
    Under the hypotheses of \cref{cor:boundary_regularity}, and upon making $r_0$ smaller, there exists a constant $C$ depending only on $n$ and $p$ such that, for every $x_0 \in \mathcal{E} \cap B_{r_0/4}$ and every $0 < \rho < r_0/4$, it holds
    \begin{equation*}
        \|u-g(x_0)\|_{L^\infty(B_\rho(x_0))} \leq C\omega_u(\rho).
    \end{equation*}
\end{lemma}
\begin{proof}
    Set $v \coloneqq u - g(x_0)$. Since $B_{1/4}(x_0) \subset B_1$ (because $|x_0| < r_0/4 < 3/4$), we may apply \cref{th:upper_bound_thm} in $B_{1/4}(x_0)$ to obtain that, upon making $r_0$ smaller (uniformly on $x_0$, because $\Omega$ and $g$ have uniform moduli of continuity in $B_1$), for any $0 < \rho < r_0/4$,
    \begin{align*}
        \frac{\|v\|_{L^{\infty}(B_{\rho}(x_0))}}{\sqrt{\rho}} 
        &\leq C \left( \frac{\|v\|_{L^{\infty}(B_{r_0/4}(x_0))}}{\sqrt r_0} + r_0^{\frac{3}{2}-\frac{n}{p}}\|f\|_{L^p(\Omega \cap B_{r_0/2}(x_0))} + \int_{8\rho}^{8r_0} \omega_g(s) \frac{ds}{s} \right) \\
        &\qquad \cdot \exp\left( C \int_{256\rho}^{256r_0} \omega(s) \frac{ds}{s} \right)
        \\ & \leq 
        C \left( \frac{\|v\|_{L^{\infty}(B_{r_0})}}{\sqrt r_0} + r_0^{\frac{3}{2}-\frac{n}{p}}\|f\|_{L^p(\Omega \cap B_{r_0})} + \int_{8\rho}^{8r_0} \omega_g(s) \frac{ds}{s} \right)  \\ & \qquad \cdot \exp\left( C \int_{256\rho}^{256r_0} \omega(s) \frac{ds}{s} \right).
    \end{align*}
    Then, observing that, since $x_0 \in \S$ and $u = g$ on $\S \cap B_1$, it holds $g(x_0) = u(x_0)$ and
    \begin{equation*}
        \|v\|_{L^{\infty}(B_{r_0})} 
        = 
        \|u - g(x_0)\|_{L^{\infty}(B_{r_0})} \leq \|u\|_{L^{\infty}(B_{r_0})} + |u(x_0)| \leq 2\norm{u}_{L^\infty(B_{r_0})},
    \end{equation*}
    the desired result follows.
\end{proof}

\begin{corollary}[Control around the whole slit] \label{cor:u-g_bound}
    Under the hypotheses of \cref{cor:boundary_regularity}, and upon making $r_0$ smaller, there exists a constant $C$ depending only on $n$ and $p$ such that, for every $x_0 \in \S \cap B_{r_0/8}$ and every $0 < \rho < r_0/8$,
    \begin{equation*}
        \|u-g(x_0)\|_{L^\infty(B_\rho(x_0))} \leq C\omega_u(\rho).
    \end{equation*}
\end{corollary}
\begin{proof}
    The result follows immediately from \crefnosort{lem:u-g_edge_bound,lem:u-g_non_edge_bound}.
\end{proof}


\begin{proof}[\textbf{Proof of \cref{cor:boundary_regularity}}]
    The fact that $\omega_u$ is a modulus of continuity is a straightforward computation that can be done with similar steps to \cite[Lemma 2.10]{ClaraC1}. 
    
    It remains to show that
    \begin{equation*}
        |u(x) - u(y)| \leq C\omega_u(|x-y|) \qquad\text{for every } x,y \in \overline{\Omega} \cap B_{r_0/16},
    \end{equation*}
    for some constant $C$ depending only on $n$ and $p$. We may assume $|x - y| < r_0/32$, since otherwise
    \begin{equation*}
        |u(x) - u(y)| 
        \leq 
        2\|u\|_{L^{\infty}(B_{r_0})} 
        \leq 
        2 \sqrt{32} \frac{\sqrt{\abs{x-y}}}{\sqrt{r_0}} \|u\|_{L^{\infty}(B_{r_0})} \leq C\omega_u(\abs{x-y}).
    \end{equation*}
    
    Let $d(x) \coloneqq \dist(x,\S)$ and let $x_0 \in \S$ be a point such that $|x - x_0| = d(x)$. Also, since $x \in B_{r_0/16}$ and $0 \in \S$, it holds $d(x) \leq |x| < r_0/16$ and $|x_0| \leq |x| + d(x) < r_0/8$. We distinguish two cases:
    
    \textbf{Case 1:} $\abs{x-y} < d(x)/2$, i.e., $y \in B_{d(x)/2}(x)$. Since $B_{d(x)}(x) \subset \Omega \cap B_{r_0}$, we may write $u = \bar{u} + w$ where
    \begin{equation*}
        \left\{
          \begin{aligned}
            -\Delta w &= f \quad&&\text{in } B_{d(x)}(x), \\
               w &= 0 &&\text{on } \partial B_{d(x)}(x),
          \end{aligned}
        \right.
        \qquad\text{and}\qquad
        \left\{
          \begin{aligned}
            -\Delta \bar{u} &= 0 \quad&&\text{in } B_{d(x)}(x), \\
                   \bar{u} &= u &&\text{on } \partial B_{d(x)}(x).
          \end{aligned}
        \right.            
    \end{equation*}
    On the one hand, \cref{lem:bound_poisson_eq} yields
    \begin{align*}
        |w(x) - w(y)| &\leq 
        C \, d(x)^{\frac{3}{2}-\frac{n}{p}} \abs{x-y}^{\frac{1}{2}} \|f\|_{L^p(B_{d(x)}(x))}
        \\ 
        &\leq C\sqrt{\abs{x-y}} \, r_0^{\frac{3}{2}-\frac{n}{p}} \|f\|_{L^p(\Omega \cap B_{r_0})} \\
        &\leq C\omega_u(\abs{x-y}),
    \end{align*}
    where we have used that $d(x) < r_0/16$ and $\frac{3}{2}-\frac{n}{p} > 0$.
    On the other hand, the mean value theorem and standard interior estimates for elliptic equations (because $\bar{u} - g(x_0)$ is harmonic in $B_{d(x)}(x)$) yield
    \begin{align*}
        |\bar{u}(x) - \bar{u}(y)| &\leq \abs{x-y} \|\nabla (\bar{u} - g(x_0))\|_{L^{\infty}(B_{d(x)/2}(x))}
        \\ 
        &\leq C\frac{\abs{x-y}}{d(x)} \|\bar{u} - g(x_0)\|_{L^{\infty}(B_{d(x)}(x))} \\
        &\leq C\frac{\abs{x-y}}{d(x)} \|u - g(x_0)\|_{L^{\infty}(B_{d(x)}(x))},
    \end{align*}
    where in the last step we have used the maximum principle ($\bar{u} - g(x_0)$ is harmonic in $B_{d(x)}(x)$ and coincides with $u - g(x_0)$ on $\partial B_{d(x)}(x)$). Hence, applying \cref{cor:u-g_bound} yields 
    \begin{equation*}
        |\bar{u}(x) - \bar{u}(y)|
        \leq 
        C\frac{\abs{x-y}}{d(x)} \|u - g(x_0)\|_{L^{\infty}(B_{2d(x)}(x_0))}
        \leq 
        C\frac{\abs{x-y}}{d(x)} \omega_u(2d(x))
        \leq 
        C \omega_u(\abs{x-y}),
    \end{equation*}
    where in the last inequality we have used \cref{rem:tilde_omega_sqrt_scaling} with $t = \abs{x-y}$ and $\lambda = 2d(x)/t > 1$:
    \begin{equation*}
        \frac{\abs{x-y}}{d(x)} \omega_u(2d(x))
        =
        \frac{t}{d(x)} \omega_u(\lambda t)
        \leq 
        \frac{t}{d(x)} \sqrt{\lambda} \omega_u(t)
        = 
        \sqrt{\frac{2t}{d(x)}} \omega_u(t)
        \leq 
        \omega_u(t)
        =
        \omega_u(\abs{x-y}).
    \end{equation*}
    Finally, putting all the above estimates together gives
    \begin{align*}
      |u(x) - u(y)| &\leq |\bar{u}(x) - \bar{u}(y)| + |w(x) - w(y)| \leq C\omega_u(\abs{x-y}).
    \end{align*}
    
    \textbf{Case 2:} $\abs{x-y} \geq d(x)/2$. In this case we have $|x - x_0| = d(x) \leq 2\abs{x-y}$, and also $|y - x_0| \leq 3\abs{x-y}$, so that $x,y \in B_{4\abs{x-y}}(x_0)$. Thus, \cref{cor:u-g_bound} yields  
    \begin{align*}
        |u(x) - u(y)| &\leq |u(x) - g(x_0)| + |u(y) - g(x_0)| 
        \leq 2\|u - g(x_0)\|_{L^\infty(B_{4|x-y|}(x_0))}
        \\ 
        &\leq C\omega_u(4\abs{x-y}) 
        \leq C\omega_u(\abs{x-y}),
    \end{align*}
    where the last estimate is due to       \cref{rem:tilde_omega_sqrt_scaling}. This finishes the proof.
\end{proof}


\appendix

\tocless\section{The second derivative of the inverse function \label{sec:sec_ord_IFT}}

Let $U, V \subset \R^n$ be open sets. Assume $P:U \rightarrow V$, $Q: V \rightarrow U$ are inverse functions in these sets, which are additionally $C^2$. Then, for every $y \in V$, it holds $P(Q(y))=y$, so $P_a(Q(y)) = y_a$ for every $a=1, \ldots, n$ (sub-indices refer to components of functions). Differentiating with respect to $y_j$ yields
\begin{equation*}
    \sum_b \partial_{x_b} P_a (Q(y)) \, \partial_{y_j} Q_b(y) = \delta_{j, a},
\end{equation*}
where $\delta$ denotes the Kronecker symbol. 
Differentiating again, this time with respect to $y_k$, gives
\begin{align*}
    0
    & =
    \sum_b \partial_{y_k} \left( \partial_{x_b} P_a (Q(y)) \right) \partial_{y_j} Q_b(y)
    + \sum_b \partial_{x_b} P_a (Q(y)) \, \partial_{y_j y_k} Q_b(y)
    \\ & =
    \sum_{b, c} \partial_{x_b x_c} P_a (Q(y)) \, \partial_{y_k} Q_c(y) \, \partial_{y_j} Q_b(y)
    + \sum_b \partial_{x_b} P_a (Q(y)) \, \partial_{y_j y_k} Q_b(y).
\end{align*}
Therefore, multiplying by $\partial_{y_a} Q_\ell(y)$ for $a, \ell \in \{1, \ldots, n\}$, and summing over $a$ yields
\begin{equation} \label{eq:IFT_2}
    \sum_{a, b} \partial_{y_a} Q_\ell(y) \, \partial_{x_b} P_a (Q(y)) \, \partial_{y_j y_k} Q_b(y)
    =
    - \sum_{a, b, c}  \partial_{y_a} Q_\ell(y) \, \partial_{x_b x_c} P_a (Q(y)) \, \partial_{y_j} Q_b(y) \, \partial_{y_k} Q_c(y).
\end{equation}

To simplify this identity, we use the fact that $Q(P(x)) = x$ for every $x \in U$, so $Q_\ell(P(x)) = x_\ell$ for every $\ell \in \{1, \ldots, n\}$, and differentiating with respect to $x_b$ gives
\begin{equation*}
    \sum_a \partial_{y_a} Q_\ell (P(x)) \, \partial_{x_b} P_a (x) = \delta_{\ell, b}.
\end{equation*}
Hence, putting $x = Q(y)$, we can simplify the left hand side of \eqref{eq:IFT_2}:
\begin{multline*}
    \sum_{a, b} \partial_{y_a} Q_\ell(y) \, \partial_{x_b} P_a (Q(y)) \, \partial_{y_j y_k} Q_b(y)
    =
    \sum_b \left( \sum_a \partial_{y_a} Q_\ell (y) \, \partial_{x_b} P_a (Q(y)) \right) \partial_{y_j y_k} Q_b(y)
    \\ =
    \sum_b \delta_{\ell, b} \, \partial_{y_j y_k} Q_b(y)
    =
    \partial_{y_j y_k} Q_\ell(y),
\end{multline*}
which recalling \eqref{eq:IFT_2} implies that, for every $j,k,\ell \in \{1, \ldots, n\}$,
\begin{equation*}
    \partial_{y_j y_k} Q_\ell(y)
    =
    - \sum_{a, b, c}  \partial_{y_a} Q_\ell(y) \, \partial_{x_b x_c} P_a (Q(y)) \, \partial_{y_j} Q_b(y) \, \partial_{y_k} Q_c(y).
\end{equation*}


\tocless\section{Control of oscillations far from the edge of the slit \label{sec:sec_corollary_appendix}}

In \cref{lem:u-g_edge_bound}, we established an $L^\infty$ bound for solutions near the edge of the slit, which is the most significant case because it is where the square-root-like behavior appears. However, to prove \cref{cor:boundary_regularity}, that is, that $u$ has a modulus of continuity in the whole $B_{r_0/16}$, we need to make sure that the solution behaves well near the part of the slit which lies away from the edge. This case is actually easier because, locally, the boundary resembles that of a thick domain. Since the proof does not really reflect the particularities of slit domains, we preferred to move the next result to the appendix.

\begin{lemma}[Control away from the edge]\label{lem:u-g_non_edge_bound}
    Under the hypotheses of \cref{cor:boundary_regularity}, and upon making $r_0$ smaller, there exists a constant $C$ depending only on $n$ and $p$ such that, for every $x_0 \in (\S \setminus \mathcal{E}) \cap B_{r_0/8}$ and every $0 < \rho < r_0/8$,
    \begin{equation*}
        \|u-g(x_0)\|_{L^\infty(B_\rho(x_0))} \leq C\omega_u(\rho).
    \end{equation*}
\end{lemma}

\begin{proof}
    Upon making $r_0$ smaller we may assume that $\omega(r_0) \leq \frac{1}{16n}$.

    Let $d \coloneqq \dist(x_0,\mathcal{E}) > 0$ and let $y_0 \in \mathcal{E}$ be a point such that $|x_0 - y_0| = d$. Since $0 \in \mathcal{E}$ and $x_0 \in B_{r_0/8}$, we have $d \leq |x_0| < r_0/8$ and $|y_0| \leq |x_0| + d < r_0/4$. 
    
    We first claim that
    \begin{equation} \label{eq:non_edge_large_scales}
        \|u - g(x_0)\|_{L^\infty(B_r(x_0))} \leq C\omega_u(r) \qquad\text{for every } d \leq r < r_0/8.
    \end{equation} 
    Indeed, for such $r$, $B_r(x_0) \subset B_{2r}(y_0)$ with $2r < r_0/4$. Moreover, for every $x \in \Omega \cap B_r(x_0)$,
    \begin{align*}
        |u(x) - g(x_0)| &\leq |u(x) - g(y_0)| + |u(x_0) - g(y_0)| \leq 2\|u - g(y_0)\|_{L^\infty(B_{2r}(y_0))} \\
        &\leq 2C\omega_u(2r) \leq C\omega_u(r),
    \end{align*}
    by \cref{lem:u-g_edge_bound} and \cref{rem:tilde_omega_sqrt_scaling}. In particular, \eqref{eq:non_edge_large_scales} with $r = \rho$ shows that the lemma holds whenever $d \leq \rho < r_0/8$, so we assume from now on that $\rho < d$.


    For the proof of the result in the range $\rho < d$, we start by first claiming that the following bound is true:
    \begin{equation} \label{eq:non_edge_small_scales}
        \|u - g(x_0)\|_{L^\infty(B_\rho(x_0))} \leq CM\sqrt{\rho}
        \quad\text{for } \rho < d,
    \end{equation}
    where
    \begin{equation*}
        M \coloneqq \frac{1}{\sqrt{d}}\|u - g(x_0)\|_{L^\infty(B_d(x_0))} + r_0^{\frac{3}{2}-\frac{n}{p}}\|f\|_{L^p(\Omega \cap B_{r_0})} + \omega_g(d) .
    \end{equation*}
    To prove this estimate, we observe that in $B_d(x_0)$ the domain $\Omega$ splits as $\Omega \cap B_d(x_0) = \Omega^+ \cup \Omega^-$, where 
    \begin{equation*}
        \Omega^+ \coloneqq \{x \in B_d(x_0) : x_n > \Gamma(x')\}, 
        \qquad 
        \Omega^- \coloneqq \{x \in B_d(x_0) : x_n < \Gamma(x')\}.
    \end{equation*}
    In particular, it is enough to prove \eqref{eq:non_edge_small_scales} in $\Omega^\pm$ and we focus on the proof for $\Omega^+$; the argument for $\Omega^-$ is analogous. First, split $u - g(x_0)$ in the following manner: extending $f$ by zero on $\S$, let $w$ be the solution to
    \begin{equation*}
        \left\{
          \begin{aligned}
            -\Delta w &= f \quad&&\text{in } B_{2d}(x_0), \\
                    w &= 0      &&\text{on } \partial B_{2d}(x_0),
          \end{aligned}
        \right.
    \end{equation*}
    so that $v \coloneqq u - g(x_0) - (w - w(x_0))$ is harmonic in $\Omega^+$ and continuous in $\overline{\Omega^+}$. By \cref{lem:bound_poisson_eq},
    \begin{equation} \label{eq:non_edge_w_bound}
        |w(x) - w(x_0)| \leq Cd^{\frac{3}{2}-\frac{n}{p}}|x - x_0|^{\frac{1}{2}}\|f\|_{L^p(\Omega \cap B_{r_0})} \leq CM|x - x_0|^{\frac{1}{2}}, \quad x \in \overline{B_d(x_0)},
    \end{equation}
    since $B_{2d}(x_0) \subset B_{r_0}$, $d < r_0$ and $\frac{3}{2} - \frac{n}{p} > 0$. To obtain a similar bound to \eqref{eq:non_edge_w_bound} for $v$ we will employ a barrier argument with an explicitly constructed barrier.
    
    Let $\nu$ be the unit normal vector to the graph of $\Gamma$ at $x_0$, inwards to $\Omega^+$. By \cref{def:C1_slit_domain}, the function $\Gamma(x') - \nabla\Gamma(x_0')\cdot x'$ is Lipschitz in $B'_d(x_0')$ with constant $\omega(d) \leq \frac{1}{16n}$, so an elementary computation\footnote{
        For $x \in \overline{\Omega^+}$ it holds
        \begin{align*}
            (x - x_0)\cdot(-\nabla\Gamma(x_0'),1)
            &= x_n - (x_0)_n - \nabla\Gamma(x_0')\cdot(x' - x_0')
            \geq \Gamma(x') - \Gamma(x_0') - \nabla\Gamma(x_0')\cdot(x' - x_0') \\
            &\geq -\omega(d)|x - x_0|
            \geq -\frac{1}{16n}|x - x_0|.
        \end{align*}
        Since $\nu = \frac{(-\nabla\Gamma(x_0'),1)}{\abs{(-\nabla\Gamma(x_0'),1)}}$ and $\abs{(-\nabla\Gamma(x_0'),1)} \geq 1$, this gives $(x - x_0)\cdot\nu \geq -\frac{1}{16n}|x - x_0|$.
    }
    shows that $(x - x_0)\cdot\nu \geq -\frac{1}{16n}|x - x_0|$ in $\overline{\Omega^+}$, whence the function
    \begin{equation*}
        \psi(x) \coloneqq (x - x_0)\cdot\nu + \frac{1}{8n}|x - x_0|
    \end{equation*}
    satisfies
    \begin{equation} \label{eq:psi_bounds}
        \frac{|x - x_0|}{16n} \leq \psi(x) \leq 2|x - x_0|
        \quad\text{in } \overline{\Omega^+}.
    \end{equation}
    Moreover, $\psi^{1/2}$ is superharmonic in $\Omega^+$ since, using that $\nabla \psi(x) = \nu + \frac{1}{8n} \frac{x-x_0}{\abs{x-x_0}}$ and $\Delta(\abs{x-x_0}) = \frac{n-1}{\abs{x-x_0}}$, we have
    \begin{align*}
        \Delta\left( \psi^{1/2} \right) &= \frac{1}{2}\psi^{-3/2}\left( \psi\Delta\psi - \frac{1}{2}|\nabla\psi|^2 \right) \\
        &\leq \frac{1}{2}\psi^{-3/2}\left( \frac{n-1}{8n}\frac{\psi}{|x - x_0|} - \frac{1}{2}\left(1 - \frac{1}{8n}\right)^2\right) \\
        &< \frac{1}{2}\psi^{-3/2}\left( \frac{1}{4} - \frac{1}{4} \right) = 0.
    \end{align*}
    
    With the barrier $\psi^{1/2}$ constructed, we examine $v$ over $\partial \Omega^+$ to later use the maximum principle. For that purpose, note that for $x \in \partial\Omega^+ \cap \partial B_d(x_0)$ we have $|x - x_0| = d$, so
    \begin{equation*}
        |u(x) - g(x_0)| \leq d^{-1/2}\|u - g(x_0)\|_{L^\infty(B_d(x_0))}|x - x_0|^{1/2},        
    \end{equation*}
    while for $x \in \partial\Omega^+ \cap \S$, where $u = g$, the $C^{\frac{1}{2}}$-modulus of $g$ and the monotonicity of $\omega_g$ give
    \begin{equation*}
        |u(x) - g(x_0)| \leq \omega_g(d)|x - x_0|^{1/2}.
    \end{equation*}
    Together with \eqref{eq:non_edge_w_bound} and the lower bound for $\psi$ in \eqref{eq:psi_bounds}, we deduce that
    \begin{equation*}
        |v(x)| \leq CM|x - x_0|^{1/2} \leq CM\psi(x)^{1/2} \quad\text{for } x \in \partial\Omega^+,
    \end{equation*}
    and therefore, the same inequality is true for $x \in \Omega^+$ by the maximum principle (recall that $v$ is harmonic and $\psi^{1/2}$ is superharmonic). This, along with \eqref{eq:non_edge_w_bound} and \eqref{eq:psi_bounds}, implies that
    \begin{equation*}
        |u(x) - g(x_0)| \leq |v(x)| + |w(x) - w(x_0)| \leq CM|x - x_0|^{1/2} \quad\text{in } \Omega^+,
    \end{equation*}
    which (together with the analogous argument for $\Omega^-$) finishes the proof of the claimed \eqref{eq:non_edge_small_scales}.

    To finish the proof of the lemma, now that we have established \eqref{eq:non_edge_small_scales}, it only remains to show that $M\sqrt{\rho} \leq C\omega_u(\rho)$. Since $\omega_g$ is non-decreasing and $\rho < d < r_0$, we have 
    \begin{equation*}
        \omega_g(d) \leq \frac{1}{\log 8}\int_d^{8d} \omega_g(s) \frac{ds}{s} \leq \frac{1}{\log 8}\int_{\rho}^{8r_0} \omega_g(s) \frac{ds}{s}.
    \end{equation*}
    Hence, using the above estimate, \eqref{eq:non_edge_large_scales} with $r = d$ and \cref{rem:tilde_omega_sqrt_scaling} with $\lambda = d/\rho > 1$ yields
    \begin{align*}
        M\sqrt{\rho}
        &= \frac{\sqrt{\rho}}{\sqrt{d}}\|u - g(x_0)\|_{L^\infty(B_d(x_0))} + \sqrt{\rho}\left( r_0^{\frac{3}{2}-\frac{n}{p}}\|f\|_{L^p(\Omega \cap B_{r_0})} + \omega_g(d) \right) \\
        &\leq C\frac{\sqrt{\rho}}{\sqrt{d}}\,\omega_u(d) + C\sqrt{\rho}\left( r_0^{\frac{3}{2}-\frac{n}{p}}\|f\|_{L^p(\Omega \cap B_{r_0})} + \int_{\rho}^{8r_0} \omega_g(s) \frac{ds}{s} \right) \\
        &\leq 2C\omega_u(\rho),
    \end{align*}
    which finishes the proof.
\end{proof}
\begin{remark}
    It is important to highlight the fact that the term $\exp(C\int \omega(s) \frac{ds}{s})$ corresponding to the modulus of $\S$ does not appear in the proof, and is simply added at no cost at the very last inequality to obtain the precise shape of the modulus $\omega_u$. This agrees with the intuition that, since we are away from the edge, our slit domain behaves locally as a flat thick domain. In fact, this is the main advantage of \cref{lem:u-g_non_edge_bound}: being locally in a thick domain means that we can use a simple explicit barrier (namely $\psi^{1/2}$) instead of the more complex barriers constructed along \cref{sec:barriers,sec:reg_complex_coord}, which are really only needed to treat the edge of the slit.
\end{remark}
\begin{remark}
    It should be noted that the $1/2$ exponent of the barrier $\psi^{1/2}$ in the proof of \cref{lem:u-g_non_edge_bound} is not important, and could be replaced with any exponent of the form $1-\eps$ by simply taking the Lipschitz constant of $\S$ smaller. We choose $1/2$ so that we can match the square-root-like modulus of continuity from the bound of \cref{lem:u-g_edge_bound}.
\end{remark}

\end{document}